\documentclass[11pt,twoside]{article}
\usepackage[T1]{fontenc}
\usepackage{lmodern}
\usepackage{amsmath,amsthm,amssymb}
\usepackage{enumerate}
\usepackage{graphicx}
\usepackage{cite}
\usepackage{comment}
\usepackage{oands}
\usepackage{tikz}
\usepackage{changepage}
\usepackage{bbm}
\usepackage{mathtools}
\usepackage[margin=1in]{geometry}
\usepackage[pagewise,mathlines]{lineno}
\usepackage{appendix}
\usepackage{stmaryrd}
\usepackage{microtype}
\usepackage[colorinlistoftodos]{todonotes}
\usepackage{dsfont}
\usepackage{mathrsfs}  
\usepackage{multirow}
\usepackage{graphicx,subfigure}
\usepackage{array}
\usepackage{bm}
\newcolumntype{P}[1]{>{\centering\arraybackslash}p{#1}}

\definecolor{red}{HTML}{f54b1a}
\definecolor{pink}{HTML}{d19eb1}
\definecolor{orange}{HTML}{d3772e}
\definecolor{yellow}{HTML}{ebe85d}
\definecolor{green}{HTML}{0f6852}
\definecolor{lightblue}{HTML}{01abe9}
\definecolor{darkblue}{HTML}{1b346c}
\definecolor{tan}{HTML}{e5c39e}
\definecolor{darktan}{HTML}{af9e73}
\definecolor{grey}{HTML}{c3ced0}
\definecolor{darkgrey}{HTML}{9dadc4}
\definecolor{black}{HTML}{110d1b}
\definecolor{white}{HTML}{f1f8f1}

\usepackage[pdftitle={Surface sums for lattice Yang--Mills in the large-$N$ limit}, pdfauthor={Jacopo Borga, Sky Cao and Jasper Shogren-Knaak},
colorlinks=true,linkcolor=lightblue,urlcolor=black,citecolor=darkblue,bookmarks=true,bookmarksopen=true,bookmarksopenlevel=2,unicode=true,linktocpage]{hyperref}
\usepackage[capitalise,noabbrev]{cleveref}

\usepackage{libertine}

\usepackage{fancyhdr}
\theoremstyle{plain}
\newtheorem{thm}{Theorem}[section]
\newtheorem{cor}[thm]{Corollary}
\newtheorem{lem}[thm]{Lemma}
\newtheorem{prop}[thm]{Proposition}

\newtheorem{rmk}[thm]{Remark}
\newtheorem{obs}[thm]{Observation}

\def\@rst #1 #2other{#1}
\newcommand\MR[1]{\relax\ifhmode\unskip\spacefactor3000 \space\fi
  \MRhref{\expandafter\@rst #1 other}{#1}}
\newcommand{\MRhref}[2]{\href{http://www.ams.org/mathscinet-getitem?mr=#1}{MR#2}}

\theoremstyle{definition}
\newtheorem{defn}[thm]{Definition}

\numberwithin{equation}{section}

\newcommand{\dsb}{\begin{adjustwidth}{2.5em}{0pt}
\begin{footnotesize}}
\newcommand{\dse}{\end{footnotesize}
\end{adjustwidth}}

\newcommand{\ssb}{\begin{adjustwidth}{2.5em}{0pt}}
\newcommand{\sse}{\end{adjustwidth}}

\newcommand{\aryb}{\begin{eqnarray*}}
\newcommand{\arye}{\end{eqnarray*}}
\def\alb#1\ale{\begin{align*}#1\end{align*}}
\def\allb#1\alle{\begin{align}#1\end{align}}
\newcommand{\eqb}{\begin{equation}}
\newcommand{\eqe}{\end{equation}}
\newcommand{\eqbn}{\begin{equation*}}
\newcommand{\eqen}{\end{equation*}}

\newcommand{\mcl}{\mathcal}

\let\originalleft\left
\let\originalright\right
\renewcommand{\left}{\mathopen{}\mathclose\bgroup\originalleft}
\renewcommand{\right}{\aftergroup\egroup\originalright}

\usepackage{ upgreek }

\def\DD{\mathbb{D}}
\def\EE{\mathbb{E}}

\def\MM{\mathbb{M}}
\def\NN{\mathbb{N}}

\def\RR{\mathbb{R}}
\def\SS{\mathbb{S}}

\def\ZZ{\mathbb{Z}}

\def\cM{\mathcal{M}}
\def\cN{\mathcal{N}}

\def\cP{\mathcal{P}}
\def\cQ{\mathcal{Q}}

\def\cS{\mathcal{S}}

\def\cU{\mathcal{U}}

\newcommand{\sm}{\setminus}

\newcommand{\nWg}[1]{\overline{\text{Wg}}_{#1}}

\def\mob{\text{Möb}}
\def\gec{\eta}

\DeclareMathOperator{\cat}{Cat}
\DeclareMathOperator{\BF}{BF}

\def\npe{\cN\cP\cM}

\DeclareMathOperator{\area}{area}
\DeclareMathOperator{\Wg}{Wg}
\def\nWg{\overline{\Wg}}

\DeclareMathOperator{\Tr}{Tr}
\DeclareMathOperator{\tr}{tr}
\DeclareMathOperator{\AM}{\mathcal{A}\mathcal{M}}
\DeclareMathOperator{\ACP}{\mathsf{A}\mathsf{P}}
\DeclareMathOperator{\VM}{\mathcal{V}\mathcal{M}}
\DeclareMathOperator{\VCP}{\mathsf{V}\mathsf{P}}
\DeclareMathOperator{\IM}{\mathcal{I}\mathcal{M}}
\DeclareMathOperator{\ICP}{\mathsf{I}\mathsf{P}}

\DeclareMathOperator{\AF}{\mathcal{A}\mathcal{F}}
\DeclareMathOperator{\VF}{\mathcal{V}\mathcal{F}}

\DeclareMathOperator{\EL}{EL}
\DeclareMathOperator{\cpp}{\mathsf{P}\mathsf{n}}

\DeclareMathOperator{\pp}{\mathsf{p}\mathsf{n}}
\DeclareMathOperator{\upp}{\mathsf{u}\mathsf{P}\mathsf{n}}
\DeclareMathOperator{\un}{\mathsf{U}}

\def\<{\langle} \def\>{\rangle}

\newcommand{\unitary}{\mathrm{U}}
\newcommand{\SO}{\mathrm{S}\mathrm{O}}
\newcommand{\SU}{\mathrm{S}\mathrm{U}}
\newcommand{\symgrp}{\mathrm{S}}
\newcommand{\ra}{\rightarrow}

\title{Surface sums for lattice Yang--Mills in the large-$N$ limit}
\date{    }
\author{Jacopo Borga\thanks{Department of Mathematics, Massachusetts Institute of Technology, \texttt{\fontfamily{cmr}\selectfont \href{mailto:jborga@mit.edu}{jborga@mit.edu}}} \quad Sky Cao\thanks{Department of Mathematics, Massachusetts Institute of Technology, \texttt{\fontfamily{cmr}\selectfont \href{mailto:skycao@mit.edu}{skycao@mit.edu}}} \quad Jasper Shogren-Knaak\thanks{Courant Institute of Mathematical Sciences, New York University, \texttt{\fontfamily{cmr}\selectfont \href{mailto:jasper.shogren-knaak@cims.nyu.edu}{jasper.shogren-knaak@cims.nyu.edu}}}}

\begin{document}

\maketitle

\begin{abstract}
	We give a sum over weighted \emph{planar} surfaces formula for Wilson loop expectations in the large-$N$ limit of strongly coupled lattice Yang--Mills theory, in any dimension. The weights of each surface are simple and expressed in terms of products of signed Catalan numbers. 
	
	In establishing our results, the main novelty is to convert a recursive relation for  Wilson loop expectations, known as the \emph{master loop equation}, into a new \emph{peeling exploration} of the planar surfaces. This exploration reveals hidden cancellations within the sums, enabling a deeper understanding of the structure of the planar surfaces. 
	
	We view our results as a continuation of the program initiated in \cite{cao2023random} to understand Yang--Mills theories via surfaces and as a refinement of the string trajectories point-of-view developed in \cite{chatterjee_rigorous_2019}. 
\end{abstract}

\tableofcontents

\section{Introduction}\label{sec: intro}

The construction of Euclidean Yang--Mills theories, for instance in dimensions three and four, is a famous open problem both in physics and mathematics~\cite{jaffe2006quantum}. As a first approximation to this continuum theory, one can consider a lattice discretization, which results in lattice Yang--Mills theory.
We refer the reader to \cite{chatterjee2018yangmillsprobabilists} and \cite[Section 1.1]{cao2023random} for a discussion of Yang--Mills theory and for an overview of the existing literature.

The recent work \cite{cao2023random} connected lattice Yang--Mills theory to certain surface sums, with the primary motivation being to eventually analyze them to prove new results about Yang--Mills. 

\begin{defn}
	A \textbf{surface sum} is a sum of the form $\sum_{M\in\cM}w(M)$ where $\cM$ denotes a  collection of planar or high genus maps,\footnote{See Section~\ref{sect:embedded-maps} for further details on maps.} sometimes referred to as \textbf{surfaces}, and $w(M)\in\mathbb{R}$ represents a \textbf{weight} associated with $M$.
\end{defn}

Notably, we do not require the weights to be non-negative, so surface sums do not necessarily correspond to probability measures on spaces of surfaces. The motivation for this is that the surface sums arising from Yang--Mills indeed have signed weights. This complicates things but also provides the opportunity for finding {\bf surface cancellations}, which play a crucial role in our paper.

\begin{defn}\label{def:surface-cancellation}
	We say that we have a \textbf{surface cancellation} when we are able to find a subset of surfaces $\cM'$ such that $\sum_{M\in \cM'}w(M)=0$.
\end{defn}

Besides proving new results, any new approach should also provide alternative perspectives on existing results, which is precisely the goal of the present paper. Then, by using this new perspective, in the companion paper \cite{bcsk2024area2d} we prove new results about lattice Yang--Mills. These new results improve on existing results of Basu--Ganguly \cite{Basu:2016dnp}.

More specifically, in the present paper, we relate Wilson loop expectations $\phi$ in the large-$N$ limit -- a certain limit of lattice Yang--Mills theory introduced in Section~\ref{subsec: lattice Yang--Mills} --  to new surface sums. The limit itself was previously analyzed in the works of Chatterjee \cite{chatterjee_rigorous_2019} and Jafarov \cite{jafarov2016wilson}. We view the main contribution of the present paper as providing new tools to study this limit -- see in particular the discussion after Theorem \ref{thm:main}. Moreover, as previously mentioned, these new tools will be used in \cite{bcsk2024area2d} to derive new results. To help understand the relation between our results and the existing works \cite{chatterjee_rigorous_2019, jafarov2016wilson}, it may be useful to have the following remark in mind when reading the paper.\footnote{The reader that is not familiar with the works~\cite{chatterjee_rigorous_2019, jafarov2016wilson} can skip this remark at first read.}

\begin{rmk}[\textsc{comparison between our surface sums and the string trajectories of} \cite{chatterjee_rigorous_2019, jafarov2016wilson}]\label{remark:comparison-with-chatterjee-vst}
	We discuss the conceptual difference between our surface sums and the ``string trajectories'' of \cite{chatterjee_rigorous_2019, jafarov2016wilson}. Schematically, both these works and our work study the solution $\phi$ to a fixed point equation (called master loop equation) of the form
	\begin{equation}\label{eq:phi-fixed-point-equation}
		\phi = G_\upbeta \phi + F,
	\end{equation}
	where $F$ is some explicit function,
	and $G_\upbeta$ is some explicit map which depends on a parameter $\upbeta \geq 0$, to be introduced in Section \ref{subsec: lattice Yang--Mills}. When $\upbeta$ is small enough, \cite{chatterjee_rigorous_2019} essentially shows that $G_\upbeta$ satisfies an estimate of the form 
	\[ \|G_\upbeta \eta\| \leq \frac{1}{2} \|\eta\| , \]
	where $\|\cdot\|$ is some carefully defined norm on some Banach space and $\eta$ is any element of the Banach space. Given this, the solution to the fixed point equation \eqref{eq:phi-fixed-point-equation} is unique, and moreover it can be obtained by Picard iteration, which results in the following series formula for $\phi$:
	\[ \phi = \sum_{n=0}^\infty G_\upbeta^{n} F. \]
	The series on the right-hand side can be interpreted as a sum of so-called ``vanishing string trajectories'', thus giving the formulas of \cite{chatterjee_rigorous_2019,jafarov2016wilson}. By contrast, the approach of the current paper is to first guess an explicit formula for $\phi$ in terms of a weighted sum over planar maps, and then verify that our ansatz satisfies the fixed point equation \eqref{eq:phi-fixed-point-equation}. As we discuss in Section \ref{subsec: large N surface sum guess}, the form of our ansatz is heavily motivated by the finite-$N$ Wilson loop expectation surface sum result \cite[Theorem 1.8]{cao2023random}, although we emphasize that the present paper is self-contained; in particular, it does not rely on any of the results or techniques from \cite{cao2023random}.
\end{rmk}

Next, we discuss the benefits of our new surface sum perspective. The perspective on $\phi$ provided by \cite{chatterjee_rigorous_2019, jafarov2016wilson} is algebraic in nature, in the sense that $\phi$ is characterized as the unique function satisfying certain algebraic relations. On the other hand, the perspective that we provide is geometric in nature, in the sense that we write $\phi$ as an explicit surface sum.\footnote{We point out that while surfaces are mentioned in \cite[Figure 13]{chatterjee_rigorous_2019}, the surfaces there arise by suitable interpretations of sequences of strings, whereas the surfaces in our paper are genuine {\it maps} (i.e.\ gluings of polygons). 
} 
Our main contention is that these two points of view complement each other, and in particular, our geometric perspective informs the algebraic perspective by providing additional algebraic relations that $\phi$ must satisfy -- see Point 1 in the discussion after Theorem \ref{thm:main}. These additional algebraic relations are the new tools that we previously referred to, which we crucially use in the companion work~\cite{bcsk2024area2d}.

Finally, we point out that our results can also be viewed as belonging to the area of planar maps, a field which by now has a vast literature -- see e.g. the lecture notes \cite{Curien2023} for more background and many references. From this point of view, our surface sum can be regarded as a new model of planar maps, with some significant twists: (1) the weight of a given map may be negative; (2) the maps are embedded in $\ZZ^d$ (as we will carefully explain in Section~\ref{sect:embedded-maps}). In order to prove our main results, we show that despite these twists, we are still able to understand several aspects of this model, particularly the surface cancellations. This is obtained via a new ``peeling exploration'' of our maps, introduced in Section~\ref{subsubsec: Splittings and Deformations on embedded maps}.

To summarize the rest of the paper, we begin by rigorously introducing $\unitary(N)$ lattice Yang--Mills theory in Section~\ref{subsec: lattice Yang--Mills}, then present (informally, at least) our main results in Section~\ref{sect:main-results} together with some further discussion on the main novelties of our work. 
Before formally stating our results in Section~\ref{subsec: large N limit surface sum formula}, we introduce various preliminary notions in Section~\ref{sec: Background}. Sections \ref{sect:fundamental-tools}-\ref{sec: cancellation lemma} contain proofs.

\subsection{Lattice Yang--Mills theory}\label{subsec: lattice Yang--Mills}

Let $\Lambda$ be some finite subgraph of $\ZZ^d$. We consider the set of oriented nearest neighbor edges in $\Lambda$ which we denote by $E_{\Lambda}$. We say that an edge $e\in E_{\Lambda}$ is positively oriented if the endpoint of the edge is greater than the initial point of the edge in lexicographical ordering. Let $E^+_{\Lambda}$ denote the set of positively oriented edges in $\Lambda$. For an edge $(u,v)=e\in E_{\Lambda}$ we let $e^{-1}=(v,u)$ denote the reverse direction. Whenever we consider a lattice edge, we always assume that such an edge is oriented, unless otherwise specified.

We will denote the group of $N\times N$ unitary matrices by $\unitary(N)$. The $\unitary(N)$-lattice Yang--Mills theory assigns a random matrix from $\unitary(N)$ to each oriented edge in $E_{\Lambda}$. We stipulate that this assignment must have \textbf{edge-reversal symmetry}. That is , if $Q_e$ is the matrix assigned to the oriented edge $e$ then $Q_{e^{-1}}=Q^{-1}_{e}$. 

We call an oriented cycle of edges $\ell$ a \textbf{loop}.\footnote{We stress that our definition of loop is different from the one in \cite{chatterjee_rigorous_2019}. Indeed, in the latter, work loops are defined so that they do not contain backtracks (see \eqref{eq:backt} for a definition).} We call the \textbf{null-loop} the loop with no edges, and denote it by $\emptyset$. For a loop $\ell=e_1e_2\hdots e_n$, we let $Q_{\ell} = Q_{e_1}Q_{e_2}\hdots Q_{e_n}$. We say that $\ell$ is a \textbf{simple loop} if the endpoints of all the $e_i$'s are distinct. We call a \textbf{string} (of cardinality $n$) any multiset of loops $\{\ell_1,\dots,\ell_n\}$ and denote it by $s$. We define $|\ell|$ to be the \textbf{length} of the loop $\ell$, and $|s|:= \sum_{i=1}^n|\ell_i|$ the \textbf{length} of the string $s$.

Define $\cP_{\Lambda}$ to be the collection of simple loops consisting of four edges (i.e.\ oriented squares). We call such loops \textbf{plaquettes}. We say that a plaquette $p\in \cP_{\Lambda}$ is positively oriented if the edge connecting the two smallest vertices of $p$ (with respect to lexicographical order) is positively oriented. In dimension two, this simply means that the leftmost edge of the square is oriented upwards. Let $\cP^+_{\Lambda}$ denote the set of positively oriented plaquettes in $\Lambda$. For $p\in \cP_{\Lambda}$, we let $p^{-1}$ denote the plaquette containing the same edges as $p$ but with opposite orientations. Whenever we consider a plaquette, we always assume that it is oriented, apart from when we explicitly say that we are considering its unoriented version.

We say that a loop $\ell$ has a \textbf{backtrack} if two consecutive edges of $\ell$ correspond to the same edge in opposite orientations. That is, $\ell$ has a backtrack if it is of the form 
\begin{equation}\label{eq:backt}
\ell=\pi_1 \, e  \, e^{-1}  \, \pi_2,
\end{equation} 
where $\pi_1$ and $\pi_2$ are two paths of edges and $e\in E_{\Lambda}$. Notice for such a loop, we can remove the $e\,e^{-1}$ backtrack and obtain a new loop $\pi_1 \, \pi_2$. We say that a loop $\ell$ is a \textbf{non-backtrack loop} if it has no backtracks. We say that a loop $\ell$ is a \textbf{trivial} loop if, after removing all backtracks from $\ell$, the resulting loop is the null-loop $\emptyset$. In particular, the null-loop is trivial, but the opposite is not true.

For an $N \times N$ matrix $Q$, define the \textbf{normalized trace} to be
\begin{equation*}
\tr(Q) := \frac{1}{N}\Tr(Q),
\end{equation*}
where $\Tr(Q)$ is the sum of the diagonal elements in $Q$, i.e.\ $\Tr(Q) = \sum_{i=1}^n Q_{i,i}$. Let $\cQ = (Q_e)_{e\in E^{+}_{\Lambda}}$ denote a \textbf{matrix configuration}, that is, an assignment of matrices to each positively oriented edge of $\Lambda$. The lattice Yang--Mills measure (with \emph{Wilson} action) is a probability measure for such matrix configurations. In particular, \begin{equation}\label{eq: lattice YM measure-2}
\hat{\mu}_{\Lambda,N,\beta} (\cQ):= \hat{Z}_{\Lambda,N, \beta}^{-1}\cdot\left(\prod_{p\in \cP_{\Lambda}}\exp\Big(\beta \cdot \,\Tr(Q_p)\Big)\right)\prod_{e\in E^+_{\Lambda}}dQ_{e},
\end{equation}
where $\hat{Z}_{\Lambda, N, \beta}$ is a normalizing constant\footnote{The normalizing constant $\hat{Z}_{\Lambda, N, \beta}$ is finite because $U(N)$ is a compact Lie group.} (to make $\hat{\mu}_{\Lambda,N,\beta}$ a probability measure), $\beta\in \RR$ is a parameter (often called the inverse temperature), and each $dQ_e$ denotes the Haar measure on $\unitary(N)$. We highlight that in \eqref{eq: lattice YM measure-2} we are considering both positively and negatively oriented plaquettes in $\cP_{\Lambda}$.

Since the goal of our paper is to consider the large-$N$ limit, we prefer to use the following more convenient rescaling, replacing $\beta$ by $\upbeta N$,
\begin{equation}\label{eq: lattice YM measure}
\mu_{\Lambda,N,\upbeta} (\cQ):= Z_{\Lambda,N, \upbeta}^{-1}\cdot\left(\prod_{p\in \cP_{\Lambda}}\exp\Big(\upbeta N\cdot\Tr(Q_p)\Big)\right)\prod_{e\in E^+_{\Lambda}}dQ_{e}.
\end{equation}

\begin{rmk}\label{rmk: beta factor of two}
In previous work, for instance \cite{chatterjee_rigorous_2019, chatterjee20161n,Basu:2016dnp}, $\beta$ in \eqref{eq: lattice YM measure-2} is simply replaced by $\beta N$ in \eqref{eq: lattice YM measure}, without any distinction between $\beta$ and $\upbeta$.
Moreover, the $\beta$ and $\upbeta$ terms appearing throughout this paper differ from those in previous work by a factor of $2$. That is, where we have $\beta$ or $\upbeta$, previous works would have $\beta/2$. This is because we are considering both positively and negatively oriented plaquettes. This will not be particularly important but should be noted and will be mentioned again when utilizing results from these previous works. 
\end{rmk}

The primary quantities of interest in lattice Yang--Mills are the Wilson loop observables. These observables are defined in terms of a matrix configuration $\cQ$ and a string $s$.  With this, we define \textbf{Wilson loop observables} as (note the normalized trace)
\begin{equation}\label{eq: Wilson loop observables}
W_s(\cQ) :=\prod_{\ell\in s}\tr(Q_{\ell}).
\end{equation}
Importantly, the Wilson loop observable for the null loop is defined to be $1$ for any matrix configuration $\cQ$, that is, $W_{\emptyset}(\cQ)=1$. Wilson loop observables are invariant up to adding or removing copies of the null-loop, i.e.\ $W_{\{\ell_1,\dots,\ell_n,\emptyset\}}(\cQ) = W_{\{\ell_1,\dots,\ell_n\}}(\cQ)$ for any matrix configurations $\cQ$ and any string $s=\{\ell_1,\dots,\ell_n\}$. Thus, throughout we will assume that all copies of the null loop are removed from all strings, unless the string is just the null loop, i.e.\ $s=\emptyset$. Moreover, Wilson loop observables are invariant under backtrack erasure, that is, $W_{\pi_1 \, e  \, e^{-1}  \, \pi_2}(\cQ)=W_{\pi_1 \pi_2}(\cQ)$. 

One of the fundamental questions of Yang--Mills theory is to understand the expectation of Wilson loop observables with respect to the lattice Yang--Mills measure; see~\cite{chatterjee2018yangmillsprobabilists} for further explanation. We denote this expectation by
\begin{align*}
\phi_{\Lambda, N, \upbeta}(s) := \EE_{\mu_{\Lambda,N,\upbeta}}[W_s(\cQ)].
\end{align*}
The rest of this paper is devoted to understanding the expectation of Wilson loop observables in the specific case when $N$ tends to infinity.

\begin{rmk}
As in \cite{cao2023random}, we define the Wilson loop observable with respect to the normalized trace. While this choice contrasts with some previous works, for instance \cite{chatterjee_rigorous_2019, chatterjee20161n,Basu:2016dnp}, this scaling will be natural in the large-$N$ limit.
\end{rmk}

\subsection{Informal statement of the main result}\label{sect:main-results}

We begin by informally stating the main result of our paper. It gives formulas for the large-$N$ limits of Wilson loop expectations in terms of weighted sums over \emph{planar} maps. The corresponding precise version is contained in Theorems \ref{thm: surface sum representation in 't hooft limit} and \ref{thm: fixed K 't Hooft master loop equation for surface sum}.

\begin{thm}[\textsc{$\unitary(\infty)$ Wilson loop expectations as surface sums -- informal statement}]\label{thm:main}
There exists a number $\upbeta_0(d)>0$, depending only on the dimension $d$, such that the following is true. Let $\Lambda_1\subseteq\Lambda_2\subseteq\hdots$ be any sequence of finite subsets of the lattice $\ZZ^d$ such that $\ZZ^d = \cup_{N=1}^{\infty}\Lambda_N$. If $|\upbeta|\leq \upbeta_0(d)$, then for any string $s=\{\ell_1,\dots,\ell_n\}$,
\begin{align}\label{eq:factor}
	\lim_{N\to\infty}\phi_{\Lambda_N,N,\upbeta}(s) =
	\prod_{i=1}^n\phi(\ell_i),
\end{align}
where
\begin{equation*}
	\phi(\ell)=\sum_{K:\cP_{\ZZ^d}\to\NN} \phi^K(\ell), \quad\text{with}\quad \phi^K(\ell)=\sum_{M\in\npe(\ell,K)}\upbeta^{\area(M)}w_{\infty}(M).
\end{equation*}
Here $\npe(\ell,K)$ is a finite set of connected planar maps with a single boundary component and embedded in the lattice $\mathbb Z^d$ (to be described in more detail in Sections \ref{sec: Background} and \ref{sec: Main Results})
and $w_{\infty}(M)$ is a simple product of signed Catalan numbers depending on the perimeter of certain faces of $M$. Moreover, the infinite sum $\phi(\ell)$ is absolutely convergent,\footnote{The sum $\phi^K(\ell)$ is a finite sum for all $K$ and $\ell$.} in the sense that
\[ \sum_{K : \cP_{\ZZ^d} \rightarrow \NN} |\phi^K(\ell)| < \infty. \]
Finally, we also establish in Theorem~\ref{thm: fixed K 't Hooft master loop equation for surface sum} and Corollary~\ref{cor: 't Hooft master loop equation surface sum} an important recursive relation for $\phi(\ell)$, that goes under the name of master loop equation.
\end{thm}

We emphasize here that in the large-$N$ limit, two important facts occur: (1) only planar maps appear; and (2) the weights of each map are essentially products of signed Catalan numbers, and thus are very simple and explicit (despite still being signed). This is what makes the surface sums appearing in Theorem \ref{thm:main} more tractable than the surface sums of \cite{cao2023random} for finite-$N$ Wilson loop expectations (which we review in Theorem \ref{thm: finite N surface sum}).

Surface sums which are in a sense dual to the ones appearing in Theorem \ref{thm:main} were previously proposed in the physics literature by Kostov \cite{KOSTOV1984445}. In the mathematics literature, the factorization in \eqref{eq:factor} was first established\footnote{See also~\cite{jafarov2016wilson} for the $\SU(N)$ case.} by Chatterjee~\cite{chatterjee_rigorous_2019} for the group $\SO(N)$ instead of $\unitary(N)$, but we point out that our surface sum
\begin{equation}\label{eq:ssum}
\phi(\ell)=\sum_{K:\cP_{\ZZ^d}\to\NN}\sum_{M\in\npe(\ell,K)}\upbeta^{\area(M)}w_{\infty}(M)
\end{equation}
is rather different from the ``sum over string trajectories'' $\sum_{X}w_{\upbeta}(X)$ appearing in \cite[Theorem 3.1]{chatterjee_rigorous_2019}. Indeed, the surface sum~\eqref{eq:ssum} is a refinement of the sum over string trajectories and there are a few advantages to considering the former sum compared to the latter:
\begin{enumerate}
\item It yields a stronger master loop equation (Theorem~\ref{thm: fixed K 't Hooft master loop equation for surface sum} and Corollary~\ref{cor: 't Hooft master loop equation surface sum});
\item It allows to explicitly compute Wilson loop expectations in two dimensions in a simplified way and for a larger class of loops; as shown in our companion paper \cite{bcsk2024area2d}.
\item It offers a new geometric perspective that aids in identifying surface cancellations; as shown for instance in Lemmas~\ref{lemma: single vertex pinching cancellations}~and~\ref{lemma: backtrack cancellations}, the Master surface cancellation lemma~\ref{lemma:master-cancellation}, and Theorem~\ref{thm:master-sum-blue-faces}.
\item It leads to a natural ``peeling exploration'' (Section~\ref{subsubsec: Splittings and Deformations on embedded maps}) of planar maps, a powerful tool in the study of random planar maps~\cite{Curien2023}, which we hope to use in future work for studying scaling limits of these maps.
\item It allows to explore both surfaces and trajectories in a very local way (Section~\ref{subsubsec: Splittings and Deformations on embedded maps}), without the need to erase the backtracks as in \cite{chatterjee_rigorous_2019}.
\item It possibly explains and clarifies the relation, found in \cite{Basu:2016dnp}, between large-$N$ Yang--Mills and non-crossing partitions. As mentioned in Theorem \ref{thm:main}, the weights of each surface involve a product of signed Catalan numbers, and the Catalan numbers also count non-crossing partitions. 
\end{enumerate}

Finally, while we do not state this as a theorem or proposition, we remark that our formula for $\phi(\ell)$ in Theorem \ref{thm:main} also gives the large-$N$ limit for Wilson loop expectations in $\SO(N)$ and $\SU(N)$ lattice Yang--Mills theories (for small $\upbeta$). This is because the limiting master loop equation for these other groups is the same as the limiting master loop equation for $\unitary(N)$ lattice Yang--Mills (compare \cite[Section 16]{jafarov2016wilson} with our Section \ref{sec:mlq-large-N}).

\paragraph{Acknowledgments.}~We thank Ron Nissim and Scott Sheffield for many helpful discussions. J.B.\ was partially supported by the NSF under Grant No.\ DMS-2441646. S.C.\ was partially supported by the NSF under Grant No.\ DMS-2303165.

\section{Background}\label{sec: Background}

Before formally stating our results in Section~\ref{subsec: large N limit surface sum formula}, we introduce various preliminary concepts. First, in Section~\ref{sect:embedded-maps}, we introduce embedded maps, which will constitute the surfaces used in the surface sums. In Section~\ref{subsec: sum surfaces lattice Yang--Mills}, we review the results from \cite{cao2023random}, explaining how to interpret Wilson loop expectations in terms of surface sums. Finally, Section~\ref{subsec: Finite N master loop equation} is devoted to reviewing a recursive relation for Wilson loop expectations, the so-called master loop equation.

\subsection{Embedded maps}\label{sect:embedded-maps}

In the recent work~\cite{cao2023random}, Cao, Park, and Sheffield showed how to interpret Wilson loop expectations for the $\unitary(N)$-lattice Yang--Mills measure (and other groups of $N\times N$ matrices) as surface sums. Our goal is now to review this result. 

The key objects for such interpretation are embedded maps, which we introduce in this section after recalling the more classical definitions of maps and maps with boundary. 

\subsubsection{Maps and maps with boundary}\label{sect:orientable-surfaces}

A \textbf{surface with boundary} is a non-empty Hausdorff topological space in which every point has an open neighborhood homeomorphic to
some open subset of the upper half-plane $\RR \times \RR_+$. Its \textbf{boundary} is the set of points having a neighborhood
homeomorphic to a neighborhood of the origin $(0, 0)$ in the upper half-plane. We consider orientable\footnote{The Wikipedia page on ``Orientability'' is a good reference to recall the notion of orientable surfaces}
compact connected surfaces with a (possibly empty) boundary. The classification theorem states that these surfaces are characterized (up to homeomorphisms) by two non-negative integers: The genus $g$ and the number $b$ of connected components of the boundary. The set of compact orientable surfaces of genus $g$ with $b$ boundary components can be obtained from the connected union of $g$ tori (or from the
sphere when $g = 0$) by removing $b$ disjoint open disks whose boundaries are pairwise disjoint circles. See the left-hand side of Figure~\ref{fig-map-with-boundary} for an example.

\medskip

A \textbf{map} is a proper embedding of a finite connected graph\footnote{Our graphs admit multiple
edges but will have no loops.}  into a compact connected orientable surface without boundary. Proper means that 
\begin{enumerate}
\item edges can intersect only at vertices;
\item faces (i.e., connected components of the complement of the edges) are homeomorphic to 2-dimensional open disks.
\end{enumerate}
Maps will always be considered up
to orientation-preserving homeomorphisms of the surface into which they are embedded.
The genus of a map is defined as the genus of the surface into which it is embedded. \textbf{Planar maps} are maps of genus zero.

\medskip 

A \textbf{map with boundary}\footnote{In the literature, maps with boundary are sometime also called maps with holes.} is a proper embedding of a finite connected graph with $b$ distinct distinguished faces, called \textbf{external faces}, into a compact connected orientable surface with boundary having $b$ connected components such that:

\begin{enumerate}
\item[3.] each connected component of the boundary of the surface is contained in one distinct external face.\footnote{Note that we are not imposing that the edges of each external face are properly embedded to the boundary of the surface. Indeed, this won't be possible in general, since the boundaries of the external faces are typically neither pairwise disjoint nor simple curves. See Figure~\ref{fig-map-with-boundary} for further explanations.}
\end{enumerate}
In this case, we say that the map has \textbf{a boundary with $b$ connected components}. Every face of a map with boundary that is not an external face is called an \textbf{internal face}. 
Maps with boundary will also always be considered up
to orientation-preserving homeomorphisms of the surface into which they are embedded.
The genus of a map with boundary is defined as the genus of the surface into which it is embedded.
See the right-hand side of Figure~\ref{fig-map-with-boundary} for an example.

\begin{figure}[ht!]
\begin{center}
	\includegraphics[width=.99\textwidth]{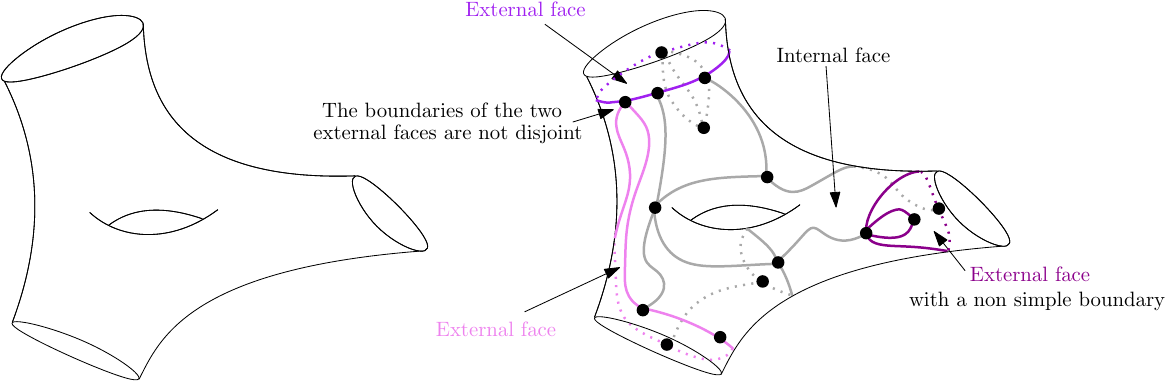}  
	\caption{\label{fig-map-with-boundary}\textbf{Left:} A compact orientable surfaces of genus $1$ with $3$ boundary components. \textbf{Right:} A map with boundary embedded in the surface on the left. The boundary of the map has $3$ connected components corresponding to the $3$ external faces, highlighted in purple, pink and magenta. Note that two external faces are not disjoint. The third one has a boundary that is not a simple curve. The map is properly embedded in the compact orientable surface. Indeed, each connected component of the boundary of the surface is contained in one distinct external face. This map has genus $1$.}
\end{center}
\vspace{-3ex}
\end{figure}

\medskip

From now on, given a map or a map with boundary, we will refer to its embedding as the \textbf{surface embedding}. We finally recall the definition of the \textbf{Euler characteristic} $\chi(m)$ of a map $m$ (with or without boundary):
\begin{equation}\label{eq:euler-ch}
\chi(m):=V(m)-E(m)+F(m),
\end{equation}
where $V(m)$, $E(m)$, and $F(m)$ are respectively the numbers of vertices, edges, and faces of the map $m$. We recall that $\chi(m)$ satisfies the following relation with the genus $g(m)$ and the number $b(m)$ of components of the boundary:
\begin{equation}\label{eq:chi-gen-relation}
\chi(m):=2-2g(m)-b(m).
\end{equation}
We also remark that $\chi(\cdot)$, $g(\cdot)$ and $b(\cdot)$ are all additive quantities for a collection of maps or maps with boundaries $\{m_1,\dots,m_n\}$, some possibly having boundary, that is 	
\begin{align*}
\chi(\{m_1,\dots,m_n\})=\sum_{i=1}^n\chi(m_i),
\end{align*}
and similarly, $g(\{m_1,\dots,m_n\})=\sum_{i=1}^n g(m_i)$ and $b(\{m_1,\dots,m_n\})=\sum_{i=1}^n b(m_i)$.

Finally, we introduce the notation $c(\{m_1,\dots,m_n\})$ to denote the number of components of $\{m_1,\dots,m_n\}$, that is, $c(\{m_1,\dots,m_n\})=n$.

\subsubsection{Embedded maps} 

We now introduce the key objects for the surface sum interpretation of Wilson loop expectations, that is, embedded maps. We immediately stress that the word ``embedded'' does not refer to the surface embedding from the previous section, but it refers to a new lattice embedding that we are soon going to define. 

\medskip

We recall that a \textbf{graph homomorphism} $\psi:G\to H$ is a mapping between the vertex sets of two graphs $G$ and $H$ such that if two vertices $u$ and $v$ of $G$ are adjacent (i.e., connected by an edge) in $G$ then $\psi(u)$ and $\psi(v)$ are adjacent in $H$. Note that every graph homomorphism can be naturally extended to a map of the set of edges of the two graphs. 

If the graphs $G$ and $H$ are maps or maps with boundary (so that the notion of face is well-defined), we say that $\psi$ \textbf{sends a face $g$ of $G$ to a face $h$ of $H$} (or to an edge $e$ of $H$) if $\psi$ sends the vertices on the boundary of $g$ to the vertices on the boundary of $h$ (or to the two vertices of $e$).

\medskip

Recall that $\Lambda$ is the lattice where the lattice Yang--Mills measure has been defined.

\begin{defn}\label{defn:embedded-map}
An \textbf{embedded map} is a pair $M = (m,\psi)$ where
\begin{enumerate}
	\item $m=\{m_1,\dots,m_r\}$ is a multiset of $r$ maps (with or without boundary and of any genus) with the following two properties (see the top of Figure~\ref{fig-embedded-maps-exemp}):
	\begin{enumerate}
		\item[(1a)] The dual graph of each component of $m$ is bipartite. The faces of $m$ in one partite class are called \textbf{blue faces} and those in the other class are called \textbf{yellow faces}.
		\item[(1b)] The external faces of each component of $m$ with boundary are yellow faces.
	\end{enumerate}
	We call a multiset of maps (with or without boundary and of any genus) with these two properties a \textbf{YB-bipartite maps family}.
	\item $\psi:m \to \Lambda$ is a graph homomorphism  with the following two properties (see the bottom of Figure~\ref{fig-embedded-maps-exemp}):
	\begin{enumerate}
		\item[(2a)] $\psi$ sends each internal yellow face of $m$ isomorphically to an unoriented plaquette of $\Lambda$,
		\item[(2b)] $\psi$ sends each blue face of $m$ to a single unoriented edge of $\Lambda$.
	\end{enumerate}
	We call a graph homomorphism $\psi:m \to \Lambda$ with these two properties a \textbf{lattice embedding} of $m$.
\end{enumerate}
\end{defn}

We remark that embedded maps were referred to as ``edge-plaquette embeddings'' in \cite{cao2023random}. When we consider an embedded map $M=(m,\psi)$, we often refer to the internal yellow faces as \textbf{plaquette faces} and to the blue faces as \textbf{edge faces}. 

\medskip

The two properties satisfied by the lattice embedding $\psi:m \to \Lambda$ in Definition~\ref{defn:embedded-map} impose two immediate constraints on the blue faces and the internal yellow faces (recall that $\Lambda\subset \ZZ^d$):
\begin{itemize}
\item Each blue face has an even degree, i.e.\ an even number of edges on its boundary.
\item Each internal yellow face has degree four, i.e., it is a quadrangle.
\end{itemize}
We highlight that blue faces might have (multiple) vertices that are not distinguished; for instance, the blue face $e_2$ in Figure~\ref{fig-embedded-maps-exemp} has degree four but it has only two vertices. We also make the following observation for future reference.

\begin{obs}\label{obs:bipartite}
Since the vertices of the lattice $\Lambda\subset\ZZ{^d}$ are naturally bipartite, the vertices of an embedded map $M$ are bipartite.
\end{obs}

\begin{figure}[ht!]
\begin{center}
	\includegraphics[width=.99\textwidth]{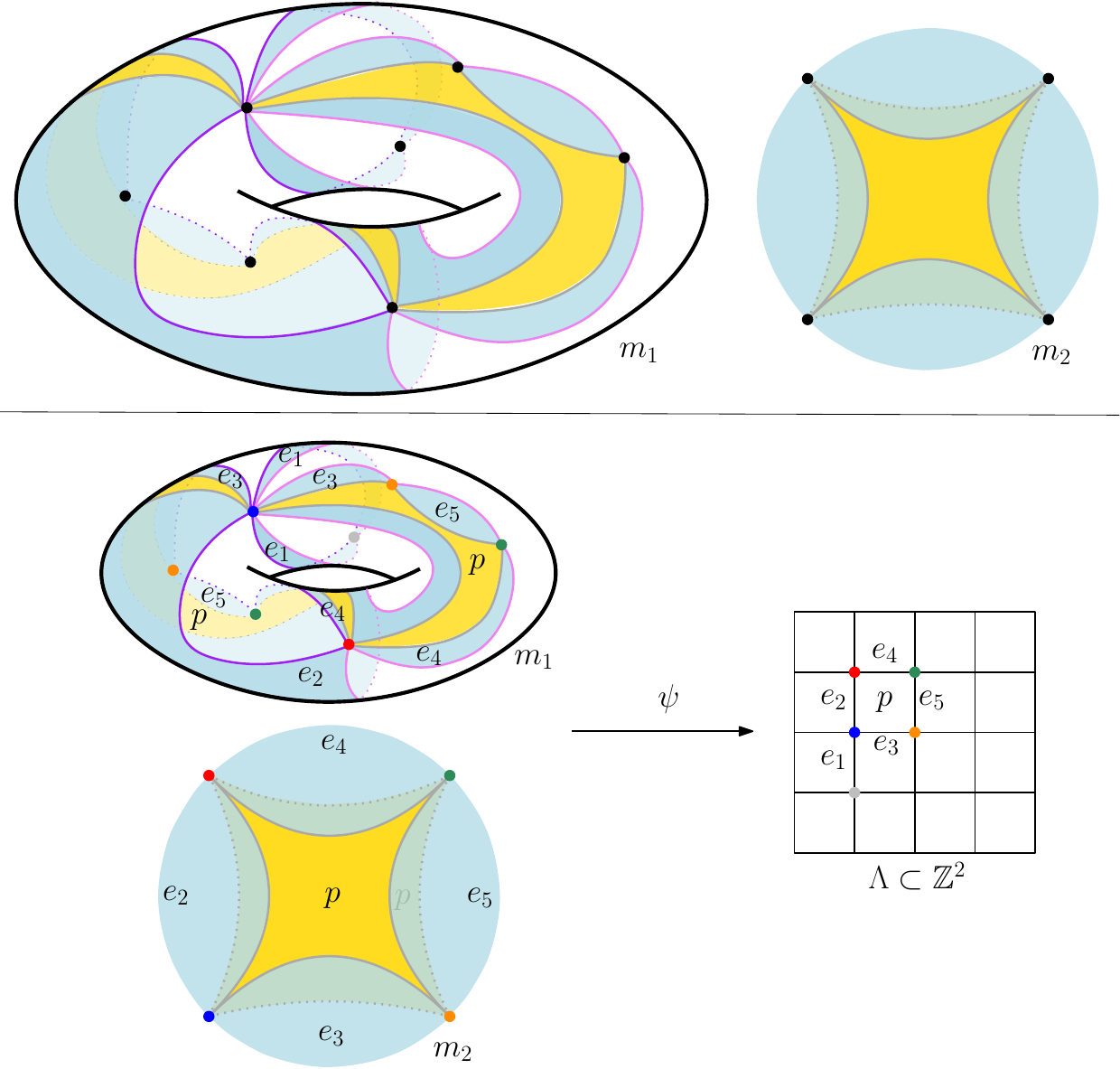}  
	\caption{\label{fig-embedded-maps-exemp}\textbf{Top:}  A YB-bipartite map family $m=\{m_1,m_2\}$  with one map $m_1$ with boundary of genus one and one map $m_2$ of genus zero. The boundary of $m_1$ has two components, one is colored in purple and one in pink. The dual graph of $m$ is bipartite. The blue faces of $m$ are colored in blue and the internal yellow faces of $m$ are colored in yellow.
		Note that the two external yellow faces of $m_1$ are not colored in the picture (but if they were colored in yellow this would still give a valid bipartion of the faces of the map, as requested in Definition~\ref{defn:embedded-map}). We always use this convention of not coloring the external yellow faces.
		\textbf{Bottom:} A lattice embedding $\psi$ of the map $m$ into the lattice $\Lambda$. Note that $\psi$ is defined by the color of the vertices: Vertices of one color in the map $m$ are sent to the vertex of $\Lambda$ of the same color. Note that $\psi$ sends the four internal yellow faces of $m$ (two in each component) isomorphically to the unoriented plaquette $p$. The internal yellow faces are labeled by the plaquette where they are sent by $\psi$. Moreover, $\psi$ sends each blue face of $m$ to a single unoriented edge of $\Lambda$. The blue faces are labeled by the edge where they are sent by $\psi$.
	}
\end{center}
\vspace{-3ex}
\end{figure}

\subsubsection{The orientation of embedded maps}

Recall from Definition~\ref{defn:embedded-map} that lattice embeddings send blue and yellow faces to \emph{unoriented} edges and plaquettes respectively. Further, recall that lattice Yang--Mills theory is defined on a lattice $\Lambda$ where all the plaquettes and the edges are oriented. Thus, to have these maps be sensible objects in lattice Yang--Mills they must ``inherit'' the orientation of the lattice. Indeed, the maps considered in \cite[Theorem 3.10]{cao2023random} (Theorem~\ref{thm: finite N surface sum} below) are such that edges in the map correspond to oriented edges of the lattice. Thus, in this section we will establish a convention for how we orient the plaquette faces and the edges of an embedded map.

Before doing this, we highlight a trivial but subtle aspect: to determine if a plaquette on the lattice is positively or negatively oriented, we considered a prior orientation of each plaquette (for instance, in two dimensions, this prior orientation is the natural clockwise orientation of each plaquette). We must also fix such a prior orientation on the embedded maps (recall their definition from Definition~\ref{defn:embedded-map}).

\medskip

\noindent\underline{\textbf{Prior orientation of an embedded map}}. The prior orientation of the edges on the boundary of each internal yellow face is defined to be their clockwise orientation.\footnote{Note that this is a natural definition since each internal yellow face is isomorphically mapped to a plaquette, which has a prior clockwise orientation (at least in dimension two).} Note that since the maps we are considering have bipartite faces, this imposes that all the edges on the boundary of the blue faces are counter-clockwise oriented; and this further imposes that the edges on the boundary of each external yellow face are also clockwise oriented. We stress that our maps or maps with boundary are always embedded in an orientable surface (recall Section \ref{sect:orientable-surfaces}) and so the clockwise/counter-clockwise orientation of the boundary of a face is well-defined. 

\medskip

\noindent\underline{\textbf{Edge orientation of an embedded map}}. We can now explain how to ``pull back'' the orientation of the edges to the embedded maps. 

Given an embedded map $M=(m,\psi)$, one can naturally ``pull back'' the orientation of the edges of $\Lambda$ to an orientation of the edges of $m$: Given two adjacent vertices $u$ and $v$ in $m$, we say that the edge between them is oriented from $u$ to $v$ if the edge between $\psi(u)$ and $\psi(v)$ in $\Lambda$ is oriented from $\psi(u)$ to $\psi(v)$.

Now, since $\psi$ is graph homeomorphism, we have that the edges along the boundary of each blue face must alternate between one direction and the opposite one. Fix an (oriented) lattice edge $e\in E^+_{\Lambda}$. Given a blue face of $m$ sent by $\psi$ to (the unoriented version of) $e$, we say that an edge of such blue face is \textbf{sent to $e$} if the orientation of this edge is consistent with the prior counter-clockwise orientation of the boundary of the blue face; otherwise, we say that the edge is sent \textbf{sent to $e^{-1}$}.  See Figure~\ref{fig-embedded-maps-exemp2} for an example.

\begin{figure}[ht!]
\begin{center}
	\includegraphics[width=.99\textwidth]{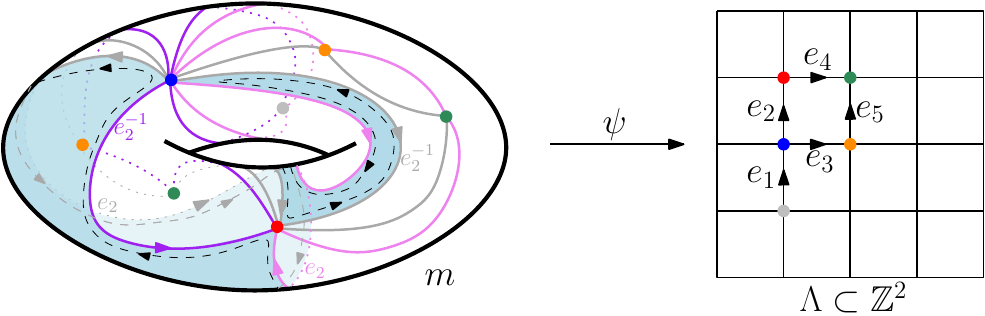}  
	\caption{\label{fig-embedded-maps-exemp2} An embedded map $M=(m,\psi)$ (obtained from the one in Figure~\ref{fig-embedded-maps-exemp} by forgetting the component $m_2$). We focus on the highlighted blue face of degree four sent to (the unoriented version of) $e_2$. The orientations of the four boundary edges of that face are ``pulled back'' from the orientation of the edge $e_2$ on the lattice $\Lambda$. We highlight -- with an oriented black dashed path --- the prior counter-clockwise orientation of the boundary of the blue face. The two edges of the blue face that have a consistent orientation with this prior counter-clockwise orientation are sent to $e_2$ by $\psi$ (and labeled by $e_2$ in the figure), while the two edges that have an inconsistent orientation with this prior counter-clockwise orientation are sent to $e_2^{-1}$ by $\psi$ (and labeled by $e_2^{-1}$ in the figure).}
\end{center}
\vspace{-3ex}
\end{figure}

From now on, given an (oriented) lattice edge $e\in E^+_{\Lambda}$, if we say that \textbf{a blue face of $m$ sent by $\psi$ to $e$}, we mean that the blue face of $m$ is sent by $\psi$ to (the unoriented version of) $e$ and the edges on the boundary of such blue face are sent to $e$ or $e^{-1}$ as explained above.

\medskip

\noindent\underline{\textbf{Plaquette face orientation of an embedded map}}. We finally explain how to ``pull back'' the orientation of plaquettes. Note that each plaquette $p$ in $\cP_\Lambda$  visits the four vertices on its boundary in a specific order and $p^{-1}$ visits these four vertices exactly in the opposite order. Given an embedded map $M=(m,\psi)$ and an internal yellow face sent by $\psi$ to (the unoriented version of) $p$, we say that such internal yellow face is \textbf{sent to} the (oriented) plaquette $p$ if the prior clockwise exploration of its boundary visits its four boundary vertices $u_1,u_2,u_3,u_4$ in the same order as $p$ visits $\psi(u_1),\psi(u_2),\psi(u_3),\psi(u_4)$; otherwise, we say that such internal yellow face is \textbf{sent to $p^{-1}$}. See Figure~\ref{fig-embedded-maps-exemp3} for an example. 

Note that this procedure gives an orientation to the internal yellow faces which is consistent with the orientation given in the previous paragraph to blue faces (actually one can easily check that one orientation determines the other).

\medskip

\begin{figure}[ht!]
\begin{center}
	\includegraphics[width=.99\textwidth]{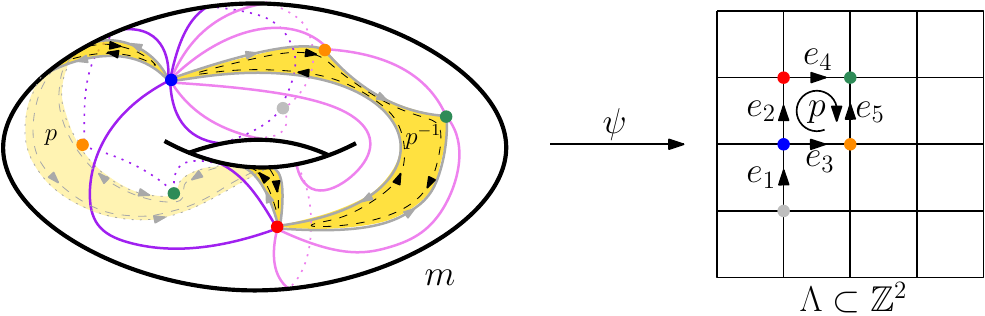}  
	\caption{\label{fig-embedded-maps-exemp3} The embedded map $M=(m,\psi)$ from Figure~\ref{fig-embedded-maps-exemp2}. We focus on the two internal yellow faces sent to (the unoriented version of) $p$. We highlight with two oriented black dashed path the prior clockwise orientation of the two boundaries of these yellow faces. The path of the yellow face on the left visits the four edges in this order: blue, red, green, and orange. Hence this face is sent to $p$ (and labeled by $p$ in the picture). The path of the yellow face on the right visits the four edges in this order: blue, orange, green, and red. Hence this face is sent to $p^{-1}$ (and labeled by $p^{-1}$ in the picture). }
\end{center}
\vspace{-3ex}
\end{figure}

\subsubsection{Embedded maps, plaquette assignments and strings}\label{subsubsec: Embedded maps, plaquette assignments, and strings}

Let $s=\{\ell_1,\dots,\ell_n\}$ be a string and consider a function $K:\cP_{\Lambda} \to \NN$, which we call a \textbf{plaquette assignment}. One should think of $K(p)$ as the number of copies of the plaquette $p$ we have available.

For such a pair $(s,K)$, we let $n_e(s,K)$ denote the number of times the oriented edge $e\in E_{\Lambda}$ appears in $(s,K)$ in the sense that 
\begin{align}\label{defn:ne}
n_e(s,K) := \text{\# copies of $e$ in $s$} + \sum_{p\in \cP_{\Lambda}(e)}K(p),
\end{align}
where $\cP_{\Lambda}(e)$ is the collection of plaquettes in $\cP_{\Lambda}$ containing $e$ as one of the four boundary edges (with the correct orientation). We say that the pair $(s,K)$ is \textbf{balanced}\footnote{The fact that we consider balanced pairs $(s,K)$ might look a bit mysterious to the reader. This is just a (non immediate) consequence of the fact that the expected trace of product of Haar-distributed matrices is zero if a matrix in the product does not appear the same number of times as its inverse.} if 
\begin{equation*}
n_e(s,K) = n_{e^{-1}}(s,K),\quad\text{for all $e\in E^+_{\Lambda}$}.
\end{equation*}

\begin{defn}\label{defn:embedded-maps-bound-plaq-ass}
For a balanced  pair $(s,K)$, we say that  $M=(m,\psi)$ is an \textbf{embedded map with boundary $s$ and plaquette assignment $K$} if $M$ is an embedded map (Definition~\ref{defn:embedded-map}) which satisfies the following three additional conditions:
\begin{enumerate}
	\item [3.] the number of internal yellow faces sent to $p$ by $\psi$ is equal to $K(p)$, for all $p\in \cP_{\Lambda}$.
	\item [4.] the total number of external yellow faces in $m$ is equal to the number of loops in $s$;
	\item [5.] the boundary of each external yellow faces of $m$ is isomorphically sent by $\psi$ to a distinct loop in $s$ preserving the orientation;\footnote{We remark that to determine if the boundary of an external yellow face of $m$ is sent to $\ell$ or $\ell^{-1}$ by $\psi$, one should use the same procedure used to determine if an internal yellow face is sent to $p$ or $p^{-1}$ by $\psi$. See also Figure~\ref{fig-embedded-maps-exemp4} for more explanations.}
\end{enumerate}
We denote by $\cM(s,K)$ the set of embedded maps with boundary $s$ and plaquette assignment $K$.
For a non-balanced pair $(s,K)$, we define $\cM(s,K)=\emptyset$. 
\end{defn}

\medskip

An embedded map in $\cM(s,K)$ can be thought of as an embedded map constructed from the plaquettes given by the plaquette assignment $K$ and with boundary components that are exactly equal to the loops in $s$. See Figure~\ref{fig-embedded-maps-exemp4} for an example.

\begin{figure}[ht!]
\begin{center}
	\includegraphics[width=.99\textwidth]{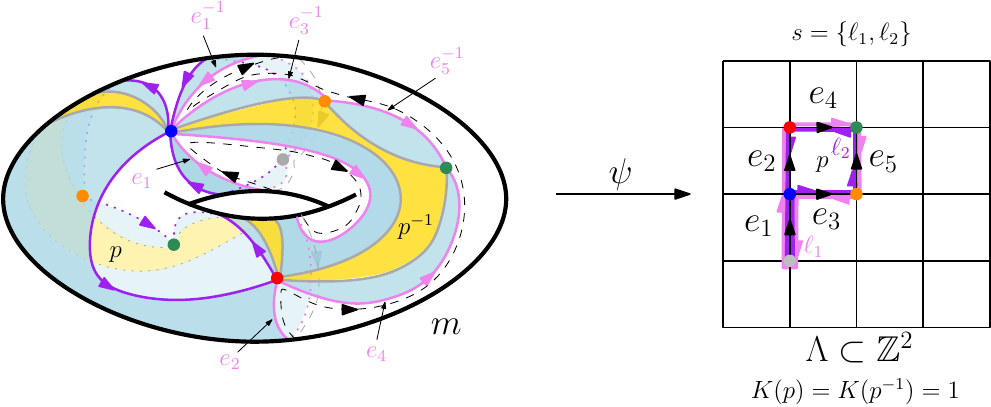}  
	\caption{\label{fig-embedded-maps-exemp4} The embedded map $M=(m,\psi)$ from Figure~\ref{fig-embedded-maps-exemp2}, which is an element of $\cM(s,K)$ with $s=\{\ell_1,\ell_2\}$ the string shown in the figure and $K$ the plaquette assignment such that $K(p)=K(p^{-1})=1$ and $K(q)=0$ for all $q\in\cP_{\Lambda}\setminus\{p,p^{-1}\}$. Note that $\ell_{1}=e_4e_5^{-1}e_3^{-1}e_1^{-1}e_1e_2$ and $\ell_{2}$ is the same loop but followed in reverse order. On the map $m$, we highlight the orientation of the edges on the two components of the boundary which are mapped to $\ell_1$ and $\ell_2$. Note that following the prior clockwise orientation of the edges in the boundary of each of these two yellow external faces (we highlight this prior orientation only for one external face), we recover exactly the two loops $\ell_1$ and $\ell_2$ (with the correct orientation).}
\end{center}
\vspace{-3ex}
\end{figure}

\subsection{Finite-\texorpdfstring{$N$}{N} 
Wilson loop expectations as surface sums} \label{subsec: sum surfaces lattice Yang--Mills}

Now that we properly defined the set of maps $\cM(s,K)$, we are almost ready to explain the main result of~\cite{cao2023random}, i.e.\ the interpretation of Wilson loop expectations as surface sums. We remark that this is not directly needed for any of the arguments in the present paper. Rather, it serves as motivation and a starting point for defining the surface sums which appear in the large-$N$ limit. We emphasize that the arguments of the present paper are self-contained; in particular, one does not need to read \cite{cao2023random} to understand our paper.

We first introduce a few remaining necessary definitions.

\medskip

The \textbf{area}, denoted by  $\area(M)$, of an embedded map $M$ is defined to be equal to the number of internal yellow faces. If $M\in \cM(s,K)$, then $\area(M) = \sum_{p\in \cP_{\Lambda}}K(p)$.

For a pair $(s,K)$, we denote by $E_{\Lambda}^+(s,K)$ the set of edges of $E_{\Lambda}^+$ that are contained in $(s,K)$. Fix $e\in E_{\Lambda}^+(s,K)$. If $M=(m,\psi)$ is an embedded map that contains $k$ blue faces that are sent to the edge $e$ by $\psi$, then we define the integer partition 
\begin{equation*}
\mu_e(M) = (\lambda_1^e,\dots,\lambda_k^e),
\end{equation*} 
where each $\lambda_i^e$ is equal to half the degree of the $i$-th largest blue face sent to the edge $e$.  
Finally, define the normalized Weingarten weight by
\begin{equation}\label{eq:wein-fct}
\nWg_N(\sigma) := N^{n+\|\sigma\|}\Wg_N(\sigma), \quad \text{for all $\sigma\in \symgrp_n$},
\end{equation}
where $\symgrp_n$ is the symmetric group on $n$ elements, $\|\sigma\| = n - \#\text{cycles}(\sigma)$, and $\Wg_N(\cdot)$ is the Weingarten function. For our purposes, the exact definition of the Weingarten function is not important, but we highlight that $\Wg_N(\cdot)$ only depends on $N$ and the cycle structure of $\sigma$. Even more importantly, we emphasize that the Weingarten function is signed—that is, it can take both positive and negative values. See \cite[Eq.\ (1.7)]{cao2023random} for an explicit formula for the Weingarten function.

We are now able to reinterpret Wilson loop expectations as surface sums.  

\begin{thm}[{\cite[Theorem 3.10]{cao2023random};} {\textsc{$\unitary(N)$ Wilson loop expectations as surface sums}}]\label{thm: finite N surface sum}
Recall that $\Lambda$ is a finite subgraph of $\mathbb{Z}^d$ for some $d \geq 2$, $N\geq 1$ and $\upbeta\in \mathbb{R}$.
Let $s=\{\ell_1,\dots,\ell_n\}$ be a string. Then
\begin{equation}\label{eq:finite-N-sum-over-surfaces}
	\phi_{\Lambda,N,\upbeta}(s) = Z_{\Lambda,N,\upbeta}^{-1}\sum_{K:\cP_{\Lambda}\to \NN}\sum_{M\in \cM(s,K)}\upbeta^{\area(M)}
	\cdot
	w_N(M)
	\cdot
	N^{\gec(M)},
\end{equation}
where $\cM(s,K)$ is introduced in \cref{defn:embedded-maps-bound-plaq-ass},
\begin{equation*}
	\gec(M):=2(c(M)-g(M)-b(M))=\chi(M) -b(M),
\end{equation*}
is the generalized Euler characteristic (recall \eqref{eq:chi-gen-relation} and the notation introduced above) and
\begin{equation}\label{eq:w-wei}
	w_N(M):=\prod_{e\in E_{\Lambda}^+(s,K)} \nWg_N(\mu_e(M)),
\end{equation}
with $\nWg_N(\mu_e(M))$ denoting the value of $\nWg_N(\sigma)$ for any permutation $\sigma$ with the same cycle structure as $\mu_e(M)$.
\end{thm}

Thus, to understand Wilson loop expectations, we need to understand the embedded maps in $\cM(s,K)$ and their weights appearing in Theorem~\ref{thm: finite N surface sum}. 

\begin{rmk}
The infinite sum in \eqref{eq:finite-N-sum-over-surfaces} converges in the following sense:
\[\sum_{K:\cP_{\Lambda}\to \NN}\left|\sum_{M\in \cM(s,K)}\upbeta^{\area(M)}
	\cdot
	w_N(M)
	\cdot
	N^{\gec(M)}\right|<\infty,\]
where we note that for any fixed plaquette assignment $K$, the sum over $M\in \cM(s,K)$ is a finite sum.
\end{rmk}

\begin{rmk}
Our definition of embedded maps with boundary $s$ and plaquette assignment $K$ (Definition~\ref{defn:embedded-maps-bound-plaq-ass}) does not distinguish plaquette faces that are sent to the same plaquette. That is, if $K(p)=n$ and $M=(m,\psi)\in \cM(s,K)$, there are $n$ internal yellow faces of $m$ that are sent by $\psi$ to \emph{the same} plaquette $p$.

In \cite{cao2023random}, the authors preferred to distinguish multiple copies of the same plaquette $p$, that is, if $K(p)=n$, they consider $n$ distinguishable copies $p_1,\dots,p_n$ of $p$ and send $n$ internal yellow face of $m$ to $p_1,\dots,p_n$. Note that there are $n!$ different possible ways to do this.

This explains why the formula in \eqref{eq:finite-N-sum-over-surfaces} differs from the one in \cite[Theorem 3.10]{cao2023random} by a $\prod_{p\in \cP_{\Lambda}}K(p)!$ factor.
\end{rmk}

\noindent\underline{\emph{Important note.}} \emph{In the following sections, we will often adopt a slightly different, yet very natural and equivalent, perspective on embedded maps. Specifically, we will often treat the embedding $\psi$ as a form of labeling for the edges and faces of the map $m$. Therefore, using the notation established earlier, when we refer to ``the edge $\mathbf{e}$ on the boundary of $m$,'' we are referring to the edge of the map $m$ that is mapped by $\psi$ to the lattice edge $e$, corresponding to the copy $\mathbf{e}$ of $\ell$. Similar interpretations will apply to blue and yellow faces.}

\emph{We also point out that we will use bold notation for edges and yellow faces on an embedded map or specific edges of a loop and normal notation to denote plaquettes and edges in the lattice $\Lambda$.}

\subsection{Finite-\texorpdfstring{$N$}{N} master loop equation}\label{subsec: Finite N master loop equation}

One powerful approach to computing Wilson loop expectations is utilizing the fact that they solve a recursive relation, variously called the Makeenko-Migdal/master loop/Schwinger-Dyson equation. We will mainly use the terminology ``master loop equation'' in this paper. In this section, we present the necessary background to state the master loop equation established in \cite{cao2023random}. This recursive relation will be paramount for our later analysis. 

\subsubsection{Loop operations}\label{sect:loop-oper}

We define three operations on loops. Let $s=\{\ell_1,\dots,\ell_n\}$ be a string and fix $\mathbf{e}$ to be one specific copy of the (oriented) edge $e$ appearing in $s$. Assume that $\mathbf{e}$ is contained in the loop $\ell_i$.

\medskip

\noindent\underline{\textbf{Splittings}}. First we define splittings at $\mathbf{e}$, an operation that splits a loop into two loops. Let $\mathbf{e}'$ denote another copy of $e$ in $\ell_i$. Suppose that $\ell_i$ has the form $\pi_1\, \mathbf{e} \,\pi_2 \,\mathbf{e}' \,\pi_3$, where $\pi_i$ is a path of edges, then we say that $\mathsf{S}_{\mathbf{e},\mathbf{e}'}(\ell_i)=\{\pi_1 \, \mathbf{e} \, \pi_3 \, , \, \pi_2 \, \mathbf{e}'\}$ is a \textbf{positive splitting} of $\ell_i$ at $\mathbf{e}$. We denote these two new loops by
\begin{align*}
\mathsf{S}_{\mathbf{e},\mathbf{e}'}^1(\ell_i) = \pi_1 \, \mathbf{e} \, \pi_3\qquad\text{and}\qquad\mathsf{S}_{\mathbf{e},\mathbf{e}'}^2(\ell_i) = \pi_2 \, \mathbf{e}'.
\end{align*}

We let $\SS_+(\mathbf{e},s)$ denote the multiset of strings that can be obtained from $s$ by a positive splitting of $s$ at $\mathbf{e}$. See the top of Figure~\ref{fig-operations} for an example.

Similarly, let $\mathbf{e}^{-1}$ be any specific copy of the edge $e^{-1}$ in $s$. Suppose $\ell_i$ has the form $\pi_1 \, \mathbf{e} \, \pi_2 \, \mathbf{e}^{-1} \, \pi_3$, then we say that $\mathsf{S}_{\mathbf{e},\mathbf{e}^{-1}}(\ell_i)=\{\pi_1 \, \pi_3  \, ,  \, \pi_2\}$ is a \textbf{negative splitting} of $\ell_i$ at $\mathbf{e}$. Similarly, we denote these two new loops by
\begin{align*}
\mathsf{S}_{\mathbf{e},\mathbf{e}^{-1}}^1(\ell_i) = \pi_1 \, \pi_3\quad\text{and}\quad\mathsf{S}_{\mathbf{e},\mathbf{e}^{-1}}^2(\ell_i) = \pi_2.
\end{align*}

We let $\SS_-(\mathbf{e},s)$ denote the multiset of strings that can be obtained from $s$ by a negative splitting of $s$ at $\mathbf{e}$. See the top of Figure~\ref{fig-operations} for an example.

\medskip

\noindent\underline{\textbf{Mergers}}. Next we define mergers at $\mathbf{e}$, where two loops are combined to create one loop. Let $\mathbf{e}'$ be any copy of $e$ in $s$ but on a different loop than $\ell_i$. Assume that $\mathbf{e}'$ is on the loop $\ell_h$, where $h\neq i$. Suppose that $\ell_i = \pi_1 \, \mathbf{e} \, \pi_2$ and $\ell_h = \pi_3 \, \mathbf{e}' \, \pi_4$. Then we define the \textbf{positive merger} of $\ell_i$ with $\ell_h$ at $\mathbf{e}$ and $\mathbf{e}'$ to be \begin{equation*}
\ell_i\oplus_{\mathbf{e},\mathbf{e}'}\ell_h = \pi_1 \, \mathbf{e} \, \pi_4 \, \pi_3 \,\mathbf{e}' \, \pi_2.
\end{equation*}
We let $\MM_+(\mathbf{e},s)$ denote the set of strings that can be obtained from $s$ by positively merging the loop $\ell_i$ with another loop in $s$ at $\mathbf{e}$. See the bottom-left of Figure~\ref{fig-operations} for an example.

Similarly, let $\mathbf{e}^{-1}$ be any copy of $e^{-1}$ on a different loop than $\ell_i$. Assume that $\mathbf{e}^{-1}$ is on $\ell_h$, where $h\neq i$. Suppose that $\ell_i = \pi_1 \, \mathbf{e} \, \pi_2$ and $\ell_h = \pi_3 \, \mathbf{e}^{-1} \pi_4$ . Then we define the \textbf{negative merger} of $\ell_i$ with $\ell_h$ at $\mathbf{e}$ and $\mathbf{e}^{-1}$ to be \[\ell_i\ominus_{\mathbf{e},\mathbf{e}^{-1}}\ell_h = \pi_1 \, \pi_4 \, \pi_3 \, \pi_2.\] 
We let $\MM_-(\mathbf{e},s)$ denote the set of strings that can be obtained from $s$ by negatively merging the loop $\ell_i$ with another loop in $s$ at $\mathbf{e}$. See the bottom-left of Figure~\ref{fig-operations} for an example.

\medskip

\noindent\underline{\textbf{Deformations}}. Lastly, we define a \textbf{positive deformation} at $\mathbf{e}$ to be a positive merger of $\ell_i$ at $\mathbf{e}$ with some plaquette $p$ that contains the edge $e$. Similarly, we define a \textbf{negative deformation} at $\mathbf{e}$ to be a positive merger of $\ell_i$ at $\mathbf{e}$ with some plaquette $q$ that contains the edge $e^{-1}$. Notice since a plaquette can only contain the edge $e$ (or the edge $e^{-1}$) at most once, as long as the edge $\mathbf{e}$ is specified there is no ambiguity as to where the merger needs to be performed. Thus, we can simplify the notation of mergers with plaquettes as follows:
\begin{equation*}
\ell_i\oplus_{\mathbf{e},\mathbf{e}'}p = \ell_i\oplus_{\mathbf{e}}p\quad\text{and}\quad \ell_i\ominus_{\mathbf{e},\mathbf{e}^{-1}}q = \ell_i\ominus_{\mathbf{e}}q.
\end{equation*}
We define the sets of positive and negative deformations at $\mathbf{e}$ by 
\begin{align*}
&\DD_+(\mathbf{e},s) = \{\{\ell_1,\dots,\ell_{i-1},\ell_i\oplus_{\mathbf{e}}p,\ell_{i+1},\dots,\ell_n\}:p\in \cP_{\Lambda}(e)\},\\
&\DD_-(\mathbf{e},s) = \{\{\ell_1,\dots,\ell_{i-1},\ell_i\oplus_{\mathbf{e}}q, \ell_{i+1},\dots,\ell_n\}:q\in \cP_{\Lambda}(e^{-1})\},
\end{align*}
where we recall that $\cP_{\Lambda}(e)$ is the collection of plaquettes in $\cP_{\Lambda}$ containing $e$ as one of the four boundary edges (with the correct orientation). See the bottom-right of Figure~\ref{fig-operations} for an example.

\begin{figure}[ht!]
\begin{center}
	\includegraphics[width=.49\textwidth]{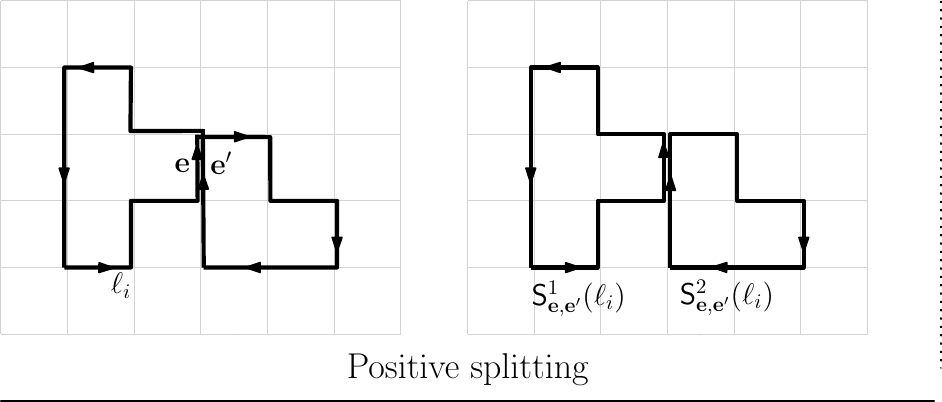}  
	\includegraphics[width=.49\textwidth]{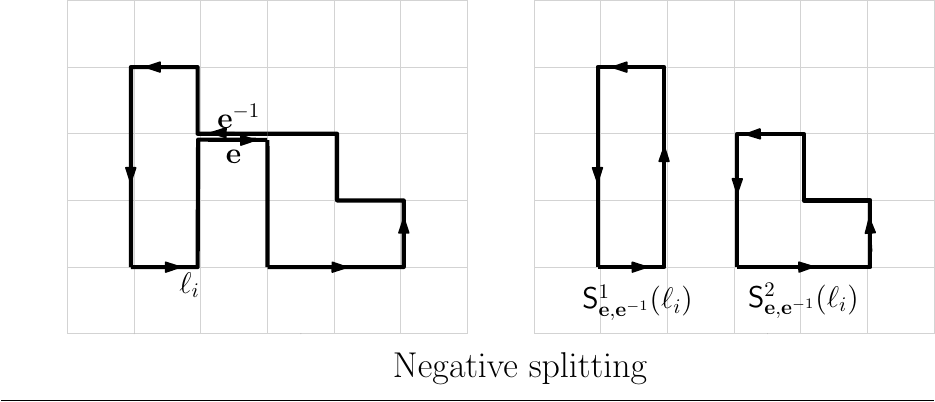}
	\includegraphics[width=.495\textwidth]{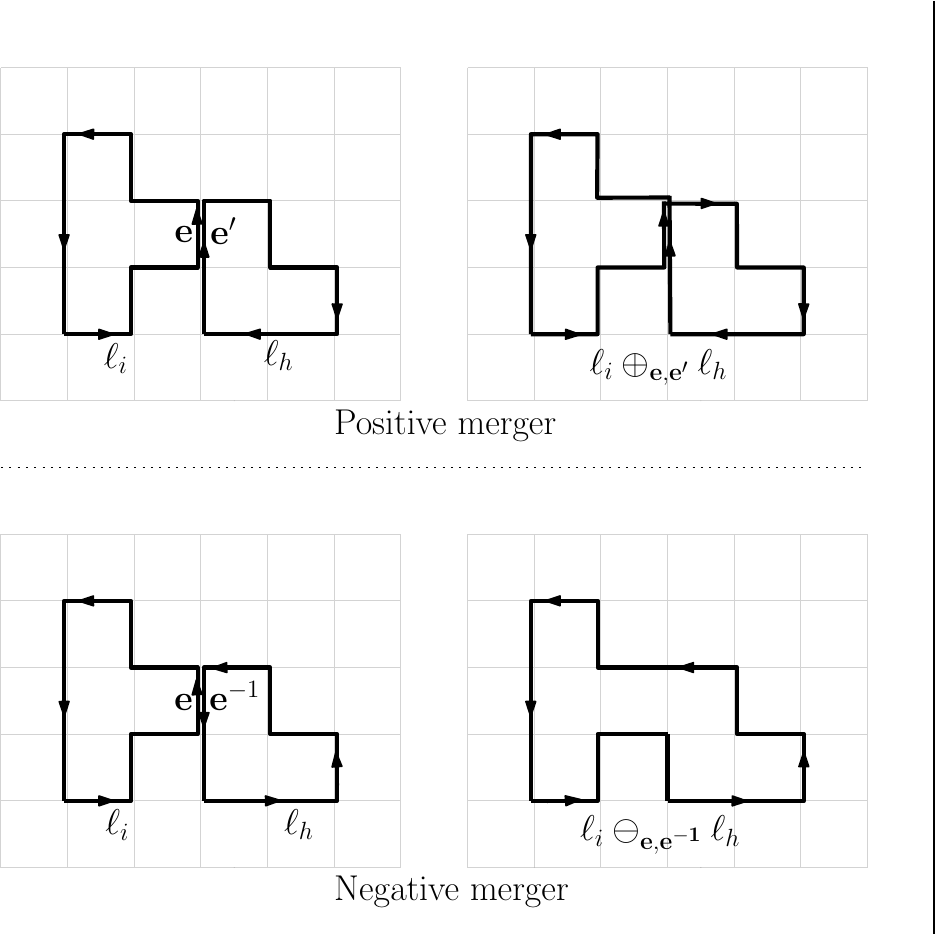} 
	\includegraphics[width=.49\textwidth]{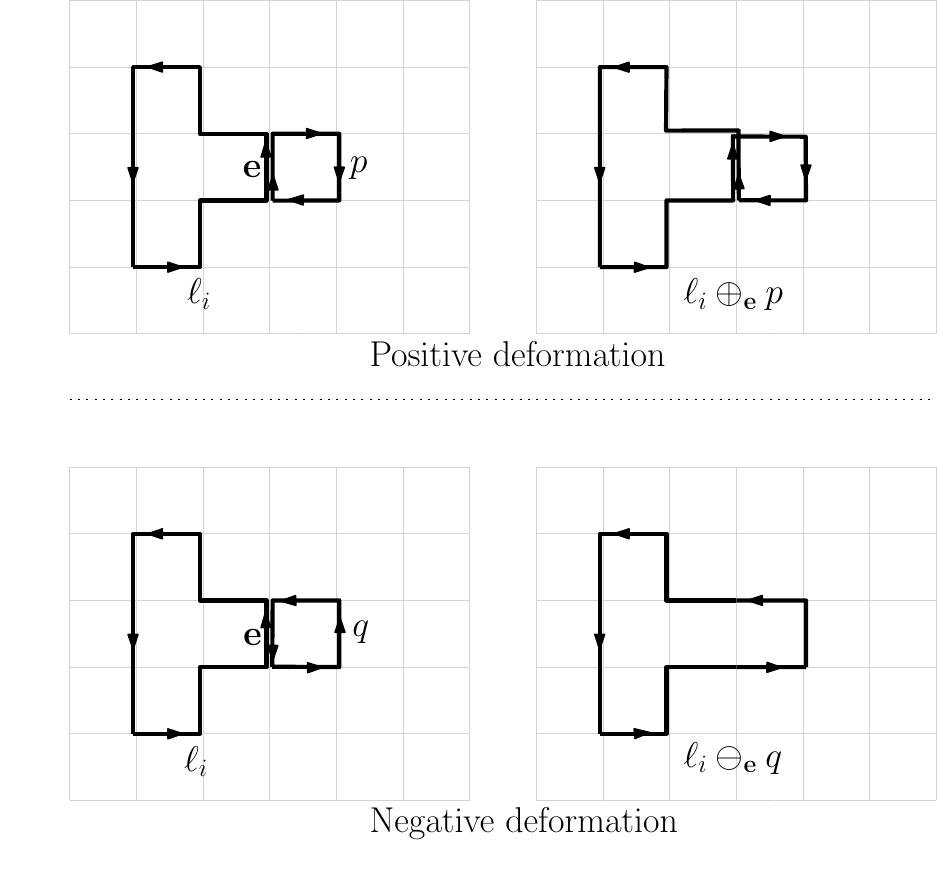} 
	\caption{\label{fig-operations} \textbf{Top:} An example of a positive and negative splitting.
		\textbf{Bottom-left:} An example of a positive and negative merger.
		\textbf{Bottom-right:} An example of a positive and negative deformation.}
\end{center}
\vspace{-3ex}
\end{figure}

\subsubsection{Statement of the master loop equation}

Now we are able to state the master loop equation for $\unitary(N)$-lattice Yang--Mills theory.

\begin{thm}[{\cite[Theorem 5.7]{cao2023random};} \textsc{$\unitary(N)$ master loop equation for finite $N$}]\label{thm: finite N master loop equation}
Let $s = \{\ell_1,\dots,\ell_n\}$ be a string. Fix a specific edge $\mathbf{e}$ on $s$. Then \begin{align*}
	\phi_{\Lambda,N, \upbeta}(s) =&\sum_{s'\in \SS_-(\mathbf{e},s)}\phi_{\Lambda,N, \upbeta}(s') - \sum_{s'\in \SS_+(\mathbf{e},s)}\phi_{\Lambda,N, \upbeta}(s') \\
	&+ \upbeta\sum_{s'\in \DD_-(\mathbf{e},s)}\phi_{\Lambda,N, \upbeta}(s') - \upbeta \sum_{s'\in \DD_+(\mathbf{e},s)}\phi_{\Lambda, N, \upbeta}(s')\\
	&+ \frac{1}{N^2}\sum_{s'\in \MM_-(\mathbf{e},s)}\phi_{\Lambda,N, \upbeta}(s') - \frac{1}{N^2}\sum_{s'\in \MM_+(\mathbf{e},s)}\phi_{\Lambda,N, \upbeta}(s').
\end{align*}
\end{thm}

\begin{rmk}
As remarked in \cite{cao2023random} we should briefly note that the above master loop equation and string operations are slightly different from those that appear in \cite{chatterjee_rigorous_2019} and \cite{jafarov2016wilson} as the operations are defined for a specific location instead of all locations of an edge. This leads to a more general master loop equation.
\end{rmk}

\section{Main results and open problems}\label{sec: Main Results}

We now formally state our main results. In Section~\ref{subsec: large N limit surface sum formula}, we provide a surface sum formula for the large-$N$ limit of Wilson loop expectations (Theorem~\ref{thm: surface sum representation in 't hooft limit}). Then, in Section~\ref{sec:mlq-large-N}, we detail the master loop equation it satisfies (Theorem~\ref{thm: fixed K 't Hooft master loop equation for surface sum}). We conclude by discussing a list of open problems in Section~\ref{sect:open-problems}.

\subsection{Wilson loop expectations in the large-\texorpdfstring{$N$}{N} limit as sums over planar surfaces}\label{subsec: large N limit surface sum formula}

It turns out the collection of embedded maps that are considered in the surface sum for Wilson loop expectations (Theorem~\ref{thm: finite N surface sum}) greatly simplifies in the large-$N$ limit. That is, the sum will only consider a much simpler subset of $\cM(s,K)$ (introduced in Definition~\ref{defn:embedded-maps-bound-plaq-ass}) that we now define.

\begin{defn}\label{defn:planar-embedded}
Fix a loop $\ell$ and a plaquette assignment $K$ such that the pair $(\ell,K)$ is balanced. We say that an  embedded map $M=(m,\psi)$ with boundary $\ell$ and plaquette assignment $K$ (Definition~\ref{defn:embedded-maps-bound-plaq-ass}) is \textbf{planar} if $M$ satisfies the two following additional conditions:
\begin{enumerate}
	\item[6.] $m$ has a unique connected component;
	\item[7.] $m$ is planar (and so, as a consequence, it is a disk with boundary sent by $\psi$ to the loop $\ell$);
\end{enumerate}
We denote by $\cP\cM(\ell,K)$ the set of planar embedded maps with boundary $\ell$ and plaquette assignment $K$.
For a non-balanced pair $(\ell,K)$, we define $\cP\cM(\ell,K)=\emptyset$.
\end{defn}

We need one final restriction on the embedded maps that we have to consider in the large-$N$ limit. Given a planar embedded map $M=(m,\psi)$, the \textbf{dual graph} of $m$ is the standard dual graph of $m$ where we also include one vertex (with its corresponding edges) for each external yellow face of $m$ (see the two dual graphs in the middle of Figure~\ref{fig-embedded-maps-exemp5} for an example). A family of blue faces of $m$ is said to \textbf{disconnect} the dual graph of $m$ if removing all the vertices corresponding to this collection (along with the edges incident to them) from the dual graph causes it to become disconnected.

\begin{defn}\label{defn:non-separableplanar-embedded}
Fix a loop $\ell$ and a plaquette assignment $K$ such that the pair $(\ell,K)$ is balanced.
We say that a planar embedded map $M=(m,\psi)$ with boundary $\ell$ and plaquette assignment $K$ is \textbf{non-separable} if for every lattice edge $e\in E^+_{\Lambda}$,
\begin{itemize}
	\item[8.] removing all blue faces sent to the edge $e$ by $\psi$ does not disconnect the dual graph of $m$.
\end{itemize} 

We denote by $\npe(\ell,K)$ the set of non-separable planar embedded maps with boundary $\ell$ and plaquette assignment $K$.
For a non-balanced pair $(\ell,K)$, we define $\npe(\ell,K)=\emptyset$.
\end{defn}

Note that we have the trivial inclusions
\begin{equation*}
\npe(\ell,K)\subset \cP\cM(\ell,K) \subset \cM(\ell,K).
\end{equation*}
If a planar embedded map $M=(m,\psi)\in \cP\cM(\ell,K)$ violates Condition 8.\ above, we say that the map is \textbf{separable}. 
We refer to a \emph{minimal} family of blue faces, all sent to the same edge $e$ by $\psi$, which disconnects the dual graph of $m$, as an \textbf{enclosure loop}.

\begin{obs}\label{obs:enclosure loops}
We note that every enclosure loop can be reduced to a simple loop (by shrinking the interior of the blue faces forming the enclosure loop) surrounding at least one vertex of the dual graph. See the red simple loops in Figure~\ref{fig-embedded-maps-exemp5}.
\end{obs}

\begin{figure}[ht!]
\begin{center}
	\includegraphics[width=.99\textwidth]{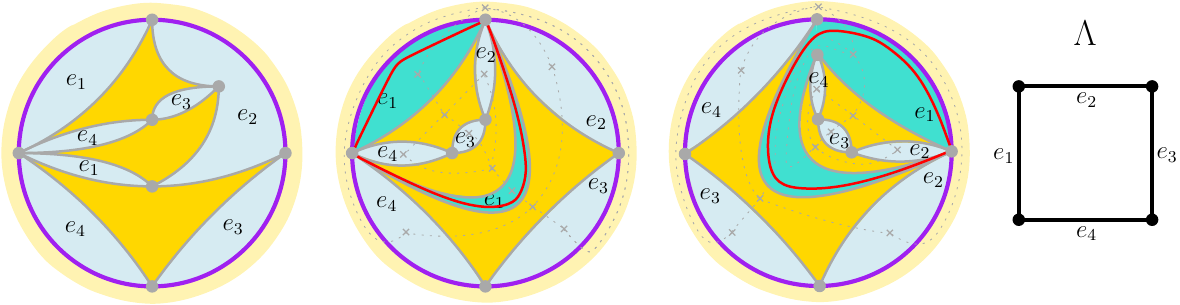}  
	\caption{\label{fig-embedded-maps-exemp5}  Three planar embedded maps with the topology of the disk (the boundary of each map is drawn in purple and the corresponding yellow external face is highlighted in light yellow) sent by the embedding to the lattice $\Lambda$ on the right: We only displayed where the blue faces are sent by the embedding. The first map is non-separable. The second and the third ones are separable: We highlight in turquoise two enclosure loops (one on each map) corresponding to blue faces sent to $e_1$ (in red we highlighted the simple loops mentioned in Observation~\ref{obs:enclosure loops}). Note that the first enclosure loop is formed by two blue faces, while the second enclosure loop is formed by a single blue face. We have also drawn the dual graph of the two separable maps: Note that removing the vertices of the dual graph (and the edges incident to them) corresponding to the faces highlighted in turquoise would disconnect the dual graph.}
\end{center}
\vspace{-3ex}
\end{figure}

In the large-$N$ limit, two simplifications occur: (1) only planar non-separable maps appear (2) the weights of each map are essentially products of signed Catalan numbers, and thus are very explicit (despite still being signed).

Given an embedded map $M\in \cM(\ell,K)$, let $\BF(M)$ denote the set of blue faces of $M$. For a blue face $f\in \BF(M)$, let $\deg(f)$ denote the degree of the blue face, i.e.\ the number of edges in the boundary of the face $f$. Recall that each blue face of $M$ has an even degree. With this, we define the new weights
\begin{equation}\label{eq:w-wei-inf}
w_{\infty}(M):= \prod_{f\in \BF(M)}w_{\deg(f)/2},\quad\text{with}\quad \text{$w_i = (-1)^{i-1}\cat(i-1)$},
\end{equation}
where $\cat(k):= \frac{(2k)!}{k!(k+1)!}$ is the $k$-th Catalan number. 

\medskip

Now we can state our surface sum formula for the large-$N$ limit of $\unitary(N)$ Wilson loop expectations.

\begin{thm}[\textsc{$\unitary(\infty)$ Wilson loop expectations as surface sums}]\label{thm: surface sum representation in 't hooft limit}
There exists a number $\upbeta_0(d)>0$, depending only on the dimension $d$, such that the following is true. Let $\Lambda_1\subseteq\Lambda_2\subseteq\hdots$ be any sequence of finite subsets of the lattice $\ZZ^d$ such that $\ZZ^d = \cup_{N=1}^{\infty}\Lambda_N$. If $|\upbeta|\leq \upbeta_0(d)$, then for any string $s=\{\ell_1,\dots,\ell_n\}$,
\begin{align*}
	\lim_{N\to\infty}\phi_{\Lambda_N,N,\upbeta}(s) =
	\prod_{i=1}^n\phi(\ell_i),
\end{align*}
where
\begin{equation*}
	\phi(\ell)=\sum_{K:\cP_{\ZZ^d}\to\NN} \phi^K(\ell), \quad\text{with}\quad \phi^K(\ell)=\sum_{M\in\npe(\ell,K)}\upbeta^{\area(M)}w_{\infty}(M),
\end{equation*}
and the infinite sum $\phi(\ell)$ is only over finite plaquette assignments $K:\cP_{\ZZ^d}\to\NN$, i.e.\ plaquette assignments such that $\sum_{p\in\cP_{\ZZ^d}}K(p)<\infty$. Moreover, the infinite sum $\phi(\ell)$ is absolutely convergent.
\end{thm}

\begin{rmk}
One might naturally wonder whether the non-separability condition in the definition of $\npe(\ell,K)$ is essential. We will show in  Appendix~\ref{sect:counter-exemp-sep} that in dimension two, even for a plaquette $p$, summing over   $\mathcal{P}\mathcal{M}(\ell,K)$ instead of $\npe(\ell,K)$ would yield an incorrect value for $\phi(\ell)$ when $\ell$ is a single plaquette $p$. 
\end{rmk}

We need to clarify how we define $\phi^K(\ell)$ when $\ell=\emptyset$. Note that $\phi(\emptyset)$ must be equal to $1$. Indeed, we defined the Wilson loop observable $W_{\emptyset}(\cQ)$ to be equal to $1$ for all matrix configurations $\cQ$ (see the discussion below \eqref{eq: Wilson loop observables}), and so $\phi_{\Lambda_N,N,\upbeta}(\emptyset)=1$.
We define 
\begin{equation}\label{eq:when-is-one}
\phi^K(\emptyset):=0, \quad\text{for all } K\neq 0, \qquad\text{and}\qquad \phi^K(\emptyset):=1,  \quad\text{when } K=0.
\end{equation}
Here $K\neq 0$ means that $K:\cP_{\ZZ^d}\to \NN$ is such that there exists $p\in \cP_{\ZZ^d}$ with $K(p)\geq 1$, and $K=0$ means that $K(p) = 0$ for all $p\in \cP_{\ZZ^d}$. 
While this definition may seem somewhat ad-hoc, as there are many ways to define $\phi^K(\emptyset)$ such that $\phi(\emptyset)=1$, it is not too hard to prove that our definition is the only one that also ensures that $\phi(s)$ satisfies the master loop equation in Theorem \ref{thm: fixed K 't Hooft master loop equation for surface sum}.

Lastly, note that, for a non-trivial loop $\ell$, we have $\phi^K(\ell):=0$ whenever $K=0$, and for any loop $\ell$, we have $\phi^K(\ell):=0$ whenever $(\ell,K)$ is not balanced. Indeed, $\npe(\ell,K)=\emptyset$ if $K=0$ or $(\ell,K)$ is not balanced. 

\subsection{The master loop equation in the large-\texorpdfstring{$N$}{N} limit}\label{sec:mlq-large-N}

We can now state the master loop equation for the large-$N$ limit of $\unitary(N)$ Wilson loop expectations.

\begin{thm}[\textsc{$\unitary(\infty)$ master loop equation for fixed plaquette assignment and location}]\label{thm: fixed K 't Hooft master loop equation for surface sum} 
Fix a non-null loop $\ell$ and a plaquette assignment  $K:\cP_{\ZZ^d}\to\NN$. Let $\mathbf{e}$ be a specific edge of the loop $\ell$. Suppose that the  edge $\mathbf{e}$ is a copy of the lattice edge $e\in E_{\ZZ^d}$. Then
\begin{align*}
	\phi^K(\ell) =&\sum_{\{\ell_1,\ell_2\}\in \SS_-(\mathbf{e},\ell)}
	\sum_{K_1+K_2=K}
	\phi^{K_1}(\ell_1)\phi^{K_2}(\ell_2) - \sum_{\{\ell_1,\ell_2\}\in \SS_+(\mathbf{e},\ell)}
	\sum_{K_1+K_2=K}\phi^{K_1}(\ell_1)\phi^{K_2}(\ell_2) \\
	&+ \upbeta
	\sum_{p\in \cP_{\ZZ^d}(e^{-1},K)}
	\phi^{K\sm p}(\ell \ominus_{\mathbf{e}}p)
	-\upbeta\sum_{q\in \cP_{\ZZ^d}(e,K)}
	\phi^{K\sm q}(\ell \oplus_{\mathbf{e}}q),
\end{align*}
where $\cP_{\ZZ^d}(e,K)$ denotes the collection of plaquettes in $\cP_{\ZZ^d}$ containing $e$ as one of the four boundary edges (with the correct orientation) and such that $K(p)\geq 1$.

We also used the (more compact) notation $\sum_{K_1+K_2=K}$ to indicate the sum $\sum_{\substack{K_1,K_2:\\K_1+K_2=K}}$, where $K_1+K_2=K$ means that $K_1(p) + K_2(p)=K(p)$ for all $p\in \cP_{\ZZ^d}$.
\end{thm}

\begin{rmk}
Note as $\phi^K(\ell)=0$ if $(\ell,K)$ is not balanced, the sum $\sum_{K_1+K_2=K}\phi^{K_1}(\ell_1)\phi^{K_2}(\ell_2)$ only considers decompositions of $K$ such that $(\ell_1,K_1)$ and $(\ell_2,K_2)$ are both balanced.
\end{rmk}

The above master loop equation will be our fundamental tool to prove Theorem~\ref{thm: surface sum representation in 't hooft limit} and explicitly compute Wilson loop expectations in dimension two in~\cite{bcsk2024area2d}.

We point out that the master loop equation introduced in Theorem \ref{thm: fixed K 't Hooft master loop equation for surface sum} is more general than the ones previously presented in the literature, as it is both for a fixed location in the loop and for a fixed plaquette assignment $K$. This seemingly subtle difference is crucial for the proofs of the results in \cite{bcsk2024area2d}. Moreover, one can easily recover the more classical (but weaker) form of the master loop equation presented in the literature, as stated in the next corollary.

\begin{cor}[\textsc{$\unitary(\infty)$ master loop equation for fixed location}]\label{cor: 't Hooft master loop equation surface sum}
Fix a non-null loop $\ell$ and a specific edge  $\mathbf{e}$ of $\ell$. Let $\upbeta_0(d)$ be as in the statement of Theorem~\ref{thm: surface sum representation in 't hooft limit}. Then for all $|\upbeta|\leq\upbeta_0(d)$,
\begin{align*}
	\phi(\ell) =&\sum_{\{\ell_1,\ell_2\}\in \SS_+(\mathbf{e},\ell)}
	\phi(\ell_1)\phi(\ell_2) - \sum_{\{\ell_1,\ell_2\}\in \SS_-(\mathbf{e},\ell)}
	\phi(\ell_1)\phi(\ell_2)\\
	&+ \upbeta\sum_{\ell'\in \DD_-(\mathbf{e},\ell)}\phi(\ell') - \upbeta\sum_{\ell'\in \DD_+(\mathbf{e},\ell)}\phi(\ell').
\end{align*}
\end{cor}

\begin{rmk}
We stress that the master loop equation in Corollary~\ref{cor: 't Hooft master loop equation surface sum} requires the additional assumption that $|\upbeta| \leq\upbeta_0(d)$, while the master loop equation in Theorem~\ref{thm: fixed K 't Hooft master loop equation for surface sum} holds for all $\upbeta\in \RR$, as for a fixed $K$, $\phi^K(\ell)$ is a finite sum and thus is defined for all $\upbeta$.
\end{rmk}

\subsection{Open problems}\label{sect:open-problems}

We give here a list of open problems that we think might be interesting to explore:

\begin{enumerate}
\item Theorem~\ref{thm: surface sum representation in 't hooft limit} is shown for $\upbeta$ small.  It would be interesting to understand the value $\upbeta^*$ at which the conclusion of Theorem \ref{thm: surface sum representation in 't hooft limit} first fails to hold. In our companion paper~\cite[Remark 1.19]{bcsk2024area2d}, we explain why we suspect that $\upbeta^* = \frac{1}{2}$. It would be nice to confirm whether this is true. 

Moreover, it would be interesting to discover the right limiting expression for $\phi_{\Lambda_N,N,\upbeta}(s)$ when $\upbeta \geq \upbeta^*$.

\item Theorem~\ref{thm: surface sum representation in 't hooft limit} determines the large-$N$ behavior of the Wilson loop expectations $\phi_{\Lambda_N,N,\upbeta}(s)$. It would be interesting to establish a $1/N$-expansion of these quantities, that is, to show that (at least  for $\upbeta$ small enough) for any string $s=\{\ell_1,\dots,\ell_n\}$, 
\begin{equation*}
	\lim_{N\to \infty} N^{2k}\left(\phi_{\Lambda_N,N,\upbeta}(s)-\phi_0(s)-\frac{1}{N^2}\phi_2(s)-\dots-\frac{1}{N^{2k}}\phi_{2k}(s)\right)=0,
\end{equation*}
where for all $h\in\NN$,
\begin{equation*}
	\phi_{2h}(s):=\sum_{K:\cP_{\ZZ^d}\to\NN} \phi^K_{2h}(s), \quad\text{with}\quad \phi^K_{2h}(s)=\sum_{M\in\cM_{2h}(s,K)}\upbeta^{\area(M)}w_{\infty}(M).
\end{equation*}
As for the proof of  Theorem~\ref{thm: surface sum representation in 't hooft limit}, the first step would be to guess the set  $\cM_{2h}(s,K)$ of maps and the type of weight $w_{\infty}(\cdot)$ one should consider in the sums
$$\phi^K_{2h}(s)=\sum_{M\in\cM_{2h}(s,K)}\upbeta^{\area(M)}w_{\infty}(M).$$ 
One natural guess (see also Appendix~\ref{subsec: large N surface sum guess} for more explanations) would be to consider (non-separable) maps with generalized Euler characteristic equal to $2h$ and the same signed Catalan weights as in the planar case. Unfortunately, this guess seems to be incorrect and various similar adaptations all turn out to be incorrect. We plan to revisit this problem in the future.

\item Theorem~\ref{thm: surface sum representation in 't hooft limit} give us a simple sum over weighted planar surfaces. This can be naturally interpreted as a finite (recall that the sum is absolutely convergent) signed measure on the set of non-separable planar embedded maps. Is it possible to define a notion of scaling limit for this maps? Whould the limiting measure be a signed measure on the space of Liouville quantum gravity surfaces? The latter are the scaling limit of many natural models or random planar maps~\cite{sheffield2022random}.

\item Our ``peeling exploration'' from Section~\ref{subsubsec: Splittings and Deformations on embedded maps} can be naturally interpreted as a signed-Markov chain on the space of non-separable planar embedded maps. Would this Markov chain help in answering the question in the previous item? Can one compute the transition signed measures of this chain?

\item In this paper we considered lattice Yang--Mills with the \emph{Wilson} action, but it seems very natural to also consider the large-$N$ limit of lattice Yang--Mills with the \emph{Villain/heat-kernel} action, in the same spirit as~\cite{park2023wilson}. This should yield a different surface sum formula that might work in a larger regime of the $\upbeta$ parameter. We plan to explore this question in future work.
\end{enumerate}

The rest of the paper is organized as follows: Section~\ref{sect:fundamental-tools} introduces some useful preliminary tools that will be used in the consecutive sections. In Section~\ref{sec: Large N Limit}, we prove our main results, i.e.\ Theorems~\ref{thm: surface sum representation in 't hooft limit}~and~\ref{thm: fixed K 't Hooft master loop equation for surface sum}, under the assumption that the Master surface cancellation lemma~\ref{lemma:master-cancellation} holds. In Section~\ref{sect:pinch-canc}, we explore certain surface cancellation results (Theorem~\ref{thm:master-sum-blue-faces}) arising from a procedure that we call ``pinching''. These surface cancellations will be fundamental to prove later in Section~\ref{sec: cancellation lemma} the Master surface cancellation lemma~\ref{lemma:master-cancellation}.

The proof of Theorem~\ref{thm: surface sum representation in 't hooft limit} in Section~\ref{sec: Large N Limit} does not offer intuition for arriving at the expression for the surface sum $\phi(s)$ presented in the statement of the theorem. Thus, in Appendix~\ref{subsec: large N surface sum guess}, we provide insights and intuition that support why this expression serves as the correct ansatz.

\section{Two fundamental tools: pinchings and backtrack cancellations}\label{sect:fundamental-tools}

In this section we introduce two preliminary tools that will be used later to establish Theorems~\ref{thm: surface sum representation in 't hooft limit}~and~\ref{thm: fixed K 't Hooft master loop equation for surface sum}.

\subsection{Pinchings and surface cancellations}\label{sect:pinchings}

In this short section we introduce an important operation on the blue faces of embedded maps, called the  pinching operation. Throughout the paper, it will be repeatedly used to get certain surface cancellations (recall Definition~\ref{def:surface-cancellation}).  

We start with a few definitions. We say that a face of an embedded map has \textbf{disjoint vertices} if none of its vertices are identified; see the left-hand side of Figure~\ref{fig-pinching} for an example.

For a balanced pair $(s,K)$, consider an embedded map $M\in \cM(s,K)$, as introduced in Definition~\ref{defn:embedded-maps-bound-plaq-ass}. Fix a blue face $B$ of $M$ with disjoint vertices and let $v$ be a vertex of $B$.
Recall from Observation~\ref{obs:bipartite} that the vertices of $M$ are bipartite. Given any vertex $u\neq v$ on $B$ in the same partite class as $v$, let $\pp_{u,v}$ denote the \textbf{pinching operation} that pinches the vertices $u$ and $v$ together inside $B$. More precisely, one draws a line from $u$ to $v$ through the interior of $B$ and contracts this line to a point splitting $B$ into two blue faces $\pp_{u,v}(B)$ that share a vertex (the vertex corresponding to $u$ and $v$). See the right-hand side of Figure~\ref{fig-pinching} for an example. Note we define pinchings for vertices in the same partite class because this ensures that the resulting two new blue faces have both even degree.

We denote by $\AF(B,v)$ the set of all pairs of blue faces obtained from pinching the vertex $v$ with another vertex of the blue face $B$ in the same partite class. Similarly, we let $\AM(M, B, v)$ denote the set of all maps that can be obtained from $M$ by the same type of pinchings.

\begin{figure}[ht!]
\begin{center}
	\includegraphics[width=.89\textwidth]{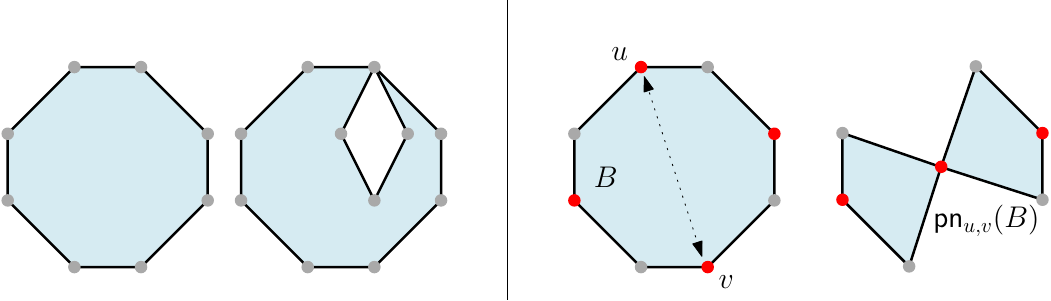}  
	\caption{\label{fig-pinching}
		\textbf{Left:} An example of a blue face with disjoint vertices and of another blue face with non disjoint vertices.
		\textbf{Right:} An example of the pinching operation $\pp_{u,v}(B)$ for a blue face $B$ having 8 vertices.}
\end{center}
\vspace{-3ex}
\end{figure}

We have the following surface cancellation result. 

\begin{lem}[\textsc{single vertex pinching cancellations}]\label{lemma: single vertex pinching cancellations}
Fix a blue face $B$ with disjoint vertices and a vertex $v$ of $B$. Then 
\begin{equation}\label{eq:ifeibewoufbw}
	w_{\infty}(B) + \sum_{B'\in \AF(B,v)}w_{\infty}(B') = 0,
\end{equation}
where if $B'$ is the pair of faces $\{B'_1,B'_2\}$, then $w_{\infty}(B'):=w_{\infty}(B'_1)\cdot w_{\infty}(B'_2)$.

As a consequence, for an embedded map $M\in \cM(s,K)$ and a vertex $v$ on a blue face $B$ of $M$ with disjoint vertices, we have that
\begin{align}\label{eq: single pinching cancellation equation form}
	w_{\infty}(M) + \sum_{M'\in \AM(M,B,v)}w_{\infty}(M') = 0.
\end{align} 
\end{lem}

\begin{proof}
Let $C = \prod_{f\in \BF(M)\sm B}w_{\deg(f)/2}$ where $\BF(M)$ denotes the set of blue faces of $M$. Note that since the embedded maps in $\AM(M,B,v)$ are obtained from $M$ by only modifying $B$, we have from \eqref{eq:w-wei-inf} that
\begin{equation*}
	w_{\infty}(M) = C\cdot w_{\infty}(B)
\end{equation*}
and 
\begin{equation*}
	\sum_{M'\in \AM(M,B,v)}w_{\infty}(M') = C \cdot \sum_{B'\in \AF(B,v)}w_{\infty}(B').
\end{equation*}
Hence we only need to prove \eqref{eq:ifeibewoufbw} to complete the proof of the lemma.
Assume that $B$ has degree $2k$. Then we can write \begin{align*}
	w_{\infty}(B) =  w_{\deg(B)/2} =   (-1)^{k-1}\cat(k - 1).
\end{align*}
Now recall the following recursive property of Catalan numbers
\begin{equation}\label{eq:catalan number recurstion}
	\cat(k)  = \sum_{i=1}^k\cat(i-1)\cat(k-i).
\end{equation}
This lets us write
\begin{align*}\
	w_{\infty}(B) &=  (-1)^{k-1}\cat(k - 1)\nonumber\\&=   (-1)^{k-1}\sum_{i=1}^{k-1}\cat(i-1)\cat(k-1-i)\nonumber\\&= - (-1)^{k-2}\sum_{j=0}^{k-2}\cat(j)\cat(k-j) \nonumber\\&= -\sum_{h=1}^{k-1}(-1)^{h-1}\cat(h-1)(-1)^{k-h-1}\cat(k-h-1)\\
	&=-\sum_{B'\in \AF(B,v)}w_{\infty}(B'_1)w_{\infty}(B'_2).
\end{align*}
This is enough to complete the proof.
\end{proof}

\subsection{Backtrack cancellations}

We turn to our second tool. One fundamental property in lattice Yang--Mills theory is the following one: if a matrix configuration assigns to an oriented edge the matrix $Q$, then it assigns the matrix $Q^{-1}$ to the opposite orientation of the same edge. This has the immediate consequence that if $\ell$ has a backtrack, i.e.\
\begin{equation*}
\ell=\pi_1 \, \mathbf{e}  \, \mathbf{e}^{-1}  \, \pi_2,
\end{equation*}
where $\pi_1$ and $\pi_2$ are two paths of edges, $\mathbf{e}$ corresponds to the oriented edge $e\in E_{\Lambda}$, and $\mathbf{e}^{-1}$ corresponds to $e^{-1}$,
then 
\begin{equation}\label{eq:requested-coeff2}
\phi_{\Lambda,N,\upbeta}(\pi_1 \, \mathbf{e} \, \mathbf{e}^{-1} \, \pi_2) = \phi_{\Lambda,N,\upbeta}(\pi_1 \, \pi_2).
\end{equation}
Obviously, the same property must be satisfied by the limit $\phi(\ell)=\sum_{K:\cP_{\ZZ^d}\to\NN} \phi^K(\ell)$ in Theorem~\ref{thm: surface sum representation in 't hooft limit}. 

In order to prove that $\phi_{\Lambda,N,\upbeta}$ converges to $\phi$ in the large-$N$ limit, we will need to show that $\phi$ has the property in \eqref{eq:requested-coeff2} before establishing Theorem~\ref{thm: surface sum representation in 't hooft limit}; see the proof of  Lemma~\ref{lemma:phi-bound} in Section~\ref{Asec: proof of uniqness to limiting master loop equation} for further details. Hence we establish this result here. In fact, we will actually show a stronger result, that this property holds for $\phi^K$. 

\begin{lem}[\textsc{backtrack cancellations}]\label{lemma: backtrack cancellations}
Suppose that $\ell=\pi_1 \, \mathbf{e} \, \mathbf{e}^{-1} \, \pi_2$ is a loop with backtrack $\mathbf{e} \, \mathbf{e}^{-1}$ and $(\ell,K)$ is a balanced pair. Then 
\begin{equation*}
	\phi^K(\pi_1 \, \mathbf{e} \, \mathbf{e}^{-1} \, \pi_2) = \phi^K(\pi_1 \, \pi_2),
\end{equation*}
where $\phi^K$ has been introduced in the statement of Theorem~\ref{thm: surface sum representation in 't hooft limit}.
\end{lem}

\begin{proof}
Let $\cM_{\mathbf{e} \, \mathbf{e}^{-1}}$ denote the set of embedded maps in $\npe(\ell,K)$ such that $\mathbf{e}$ and $\mathbf{e}^{-1}$ are together in a blue 2-gon  and $\overline{\cM}_{\mathbf{e} \, \mathbf{e}^{-1}}$ denote the complement of $\cM_{\mathbf{e} \, \mathbf{e}^{-1}}$ in $\npe(\ell,K)$. Further, let $\overline{\cM}^{\text{same}}_{\mathbf{e} \, \mathbf{e}^{-1}}$ denote the maps in $\overline{\cM}_{\mathbf{e} \, \mathbf{e}^{-1}}$ where $\mathbf{e}$ and $\mathbf{e}^{-1}$ are in the same blue face and $\overline{\cM}^{\text{split}}_{\mathbf{e} \, \mathbf{e}^{-1}}$ denote the maps where $\mathbf{e}$ and $\mathbf{e}^{-1}$ are in two separate blue faces. With this decomposition of $\npe(\ell,K)$ we have that \begin{align*}
	\phi^K(\ell) &= \upbeta^{\sum_{p\in \cP}K(p)}\left(\sum_{M\in \cM_{\mathbf{e} \, \mathbf{e}^{-1}}}w_{\infty}(M)+ \sum_{M\in \overline{\cM}^{\text{same}}_{\mathbf{e} \, \mathbf{e}^{-1}}}w_{\infty}(M)+ \sum_{M\in \overline{\cM}^{\text{split}}_{\mathbf{e} \, \mathbf{e}^{-1}}}w_{\infty}(M)\right) \\&= \phi^K(\pi_1 \, \pi_2) + \upbeta^{\sum_{p\in \cP}K(p)}\left(\sum_{M\in \overline{\cM}^{\text{same}}_{\mathbf{e} \, \mathbf{e}^{-1}}}w_{\infty}(M)+ \sum_{M\in \overline{\cM}^{\text{split}}_{\mathbf{e} \, \mathbf{e}^{-1}}}w_{\infty}(M)\right),
\end{align*}
where to get the last equality we used that each map $M\in \cM_{\mathbf{e} \, \mathbf{e}^{-1}}$ corresponds bijectively to the map in $\npe(\pi_1 \, \pi_2,K)$ obtained by  removing the  blue 2-gon containing $\mathbf{e}$ and $\mathbf{e}^{-1}$, and moreover that these two maps have the same weight, since blue 2-gons have weight 1. Thus to finish the proof we need to show that \begin{equation}\label{eq: backtrack cancellation proof eq to show}
	\sum_{M\in \overline{\cM}^{\text{same}}_{\mathbf{e} \, \mathbf{e}^{-1}}}w_{\infty}(M) + \sum_{M\in \overline{\cM}^{\text{split}}_{\mathbf{e} \, \mathbf{e}^{-1}}}w_{\infty}(M')=0.
\end{equation}
Fix $M\in \overline{\cM}^{\text{same}}_{\mathbf{e} \, \mathbf{e}^{-1}}$. Let $v$ be the vertex shared by $\mathbf{e}$ and $\mathbf{e}^{-1}$, and $B$ denote the blue face incident to $\mathbf{e}$ and $\mathbf{e}^{-1}$.  Now we rewrite the sum in the left-hand side of \eqref{eq: backtrack cancellation proof eq to show} as
\begin{align*}
	\sum_{M\in \overline{\cM}^{\text{same}}_{\mathbf{e} \, \mathbf{e}^{-1}}}\left(w_{\infty}(M) + \sum_{M'\in \AM(M,B,v)}w_{\infty}(M')\right).
\end{align*}
Notice applying Lemma~\ref{lemma: single vertex pinching cancellations} (recall each blue face in a non-separable map has disjoint vertices) we get that $w_{\infty}(M) + \sum_{M'\in \AM(M,B,v)}w_{\infty}(M')=0$,
which gives us \eqref{eq: backtrack cancellation proof eq to show} as desired.
\end{proof}

\section{Large-\texorpdfstring{$N$}{N} limit for Wilson loop expectations as sums over disks and the master loop equation}\label{sec: Large N Limit}

In this section we prove our two main results, i.e.\ Theorems~\ref{thm: surface sum representation in 't hooft limit}~and~\ref{thm: fixed K 't Hooft master loop equation for surface sum}. 
The proof of Theorem~\ref{thm: surface sum representation in 't hooft limit} will not provide any intuition for guessing the expression for $\phi(s)$ given in the statement of Theorem~\ref{thm: surface sum representation in 't hooft limit}. Instead, the proof will simply demonstrate that $\phi(s)$ is the correct limit. Therefore, in Appendix~\ref{subsec: large N surface sum guess}, we offer insight and intuition as to why this expression should be the correct ansatz. 

This section is organized as follows: in Section~\ref{subsec: large N Master loop equation in large-$N$ limit}, we prove Theorem~\ref{thm: surface sum representation in 't hooft limit} assuming Theorem~\ref{thm: fixed K 't Hooft master loop equation for surface sum}, then we give the proof of the latter result in Section~\ref{subsec: surface sum Satisfies Master Loop Equation}, after introducing the fundamental ``peeling exploration'' in Section~\ref{subsubsec: Splittings and Deformations on embedded maps}.

\medskip

For later convenience, we rewrite the master loop equation from Theorem~\ref{thm: fixed K 't Hooft master loop equation for surface sum} in a more compact form and establish a simple consequence (Corollary~\ref{cor:mle-goal}). Fix a non-null loop $\ell$ and a plaquette assignment  $K:\cP_{\ZZ^d}\to\NN$. Let $\mathbf{e}$ be a specific edge of the loop $\ell$. Suppose that the  edge $\mathbf{e}$ is a copy of the lattice edge $e\in E_{\ZZ^d}$. Then Theorem~\ref{thm: fixed K 't Hooft master loop equation for surface sum} states that
\begin{align*}
\phi^K(\ell) =&\sum_{\{\ell_1,\ell_2\}\in \SS_{-}(\mathbf{e},\ell)}
\sum_{K_1+K_2=K}
\phi^{K_1}(\ell_1)\phi^{K_2}(\ell_2) - \sum_{\{\ell_1,\ell_2\}\in \SS_{+}(\mathbf{e},\ell)}
\sum_{K_1+K_2=K}\phi^{K_1}(\ell_1)\phi^{K_2}(\ell_2) \\
&+ \upbeta
\sum_{p\in \cP_{\ZZ^d}(e^{-1},K)}
\phi^{K\sm p}(\ell \ominus_{\mathbf{e}}p)
-\upbeta\sum_{q\in \cP_{\ZZ^d}(e,K)}
\phi^{K\sm q}(\ell \oplus_{\mathbf{e}}q).
\end{align*}
For a string $s=\{\ell_1,\dots,\ell_n\}$, we set 
\begin{equation}\label{eq:gen-phi-string}
\phi^{K}(s):=\sum_{M\in\npe(s,K)}\upbeta^{\area(M)}w_{\infty}(M), 
\end{equation}
where
\begin{equation}\label{eq:string-set-maps}
\npe(\{\ell_1,\dots,\ell_n\},K):=\bigsqcup_{K_1+\dots+K_n=K}\npe(\ell_1,K_1)\times\dots\times\npe(\ell_n,K_n),
\end{equation}
that is, an element of $\npe(\{\ell_1,\dots,\ell_n\},K)$ is a collection of $n$ planar non-separable disks, each of them having boundary $\ell_i$.
With these new definitions,\footnote{Here we only need the $n=2$ case of \eqref{eq:gen-phi-string}, but later in Corollary~\ref{cor:mle-goal}, we will need the version for general $n$.} we get
\begin{equation}\label{eq:ewfjbewbfopwejnf}
\sum_{\{\ell_1,\ell_2\}\in \SS_\pm(\mathbf{e},\ell)}
\sum_{K_1+K_2=K}
\phi^{K_1}(\ell_1)\phi^{K_2}(\ell_2)=\sum_{s\in \SS_{\pm}(\mathbf{e},\ell)} \phi^{K}(s),
\end{equation}
and so, we can rewrite the master loop equation in Theorem~\ref{thm: fixed K 't Hooft master loop equation for surface sum} in the following equivalent compact form:\footnote{We did not write the master loop equation in Theorem~\ref{thm: fixed K 't Hooft master loop equation for surface sum} immediately in this compact form because we wanted to stress that the $\unitary(\infty)$ Yang--Mills theory is a theory where one can only look at single loops instead of strings. A similar comment applies to Corollary~\ref{cor:mle-goal}.}
\begin{align}\label{eq:new-form}
\phi^K(\ell) =&
\sum_{s\in \SS_{{-}}(\mathbf{e},\ell)} \phi^{K}(s)
- \sum_{s\in \SS_{{+}}(\mathbf{e},\ell)} \phi^{K}(s)\notag\\
&+ \upbeta
\sum_{p\in \cP_{\ZZ^d}(e^{-1},K)}
\phi^{K\sm p}(\ell \ominus_{\mathbf{e}}p)
-\upbeta\sum_{q\in \cP_{\ZZ^d}(e,K)}
\phi^{K\sm q}(\ell \oplus_{\mathbf{e}}q).
\end{align}

We also need the following extension of Theorem~\ref{thm: fixed K 't Hooft master loop equation for surface sum} for general strings $s=\{\ell_1,\dots,\ell_n\}$.

\begin{cor}[\textsc{$\unitary(\infty)$ master loop equation for fixed plaquette assignment and location and for general strings}]\label{cor:mle-goal}
For all strings $s=\{\ell_1,\dots,\ell_n\}$ and all plaquette assignments  $K:\cP_{\ZZ^d}\to\NN$ the following master loop equation at every fixed edge $\mathbf{e}$ in $s$ holds:
\begin{equation}
	\begin{aligned}\label{eq:limiting_master_loop_eq-2222}
		\phi^K(s) &= \sum_{s'\in \SS_-(\mathbf{e},s)}\phi^K(s') - \sum_{s'\in \SS_+(\mathbf{e},s)}\phi^K(s') 
		\\&+ \upbeta\sum_{s'\in \DD_-(\mathbf{e},s)}\phi^K(s') - \upbeta\sum_{s'\in \DD_+(\mathbf{e},s)}\phi^K(s'),
	\end{aligned}
\end{equation}
where we recall that $\phi^K(s)$ has been introduced in \eqref{eq:gen-phi-string}.
\end{cor}

\begin{proof}[Proof of Corollary~\ref{cor:mle-goal} assuming Theorem~\ref{thm: fixed K 't Hooft master loop equation for surface sum}]
Fix a string $s=\{\ell_1,\dots,\ell_n\}$ and a plaquette assignment $K:\cP_{\ZZ^d}\to\NN$. Further, fix an edge $\mathbf{e}$ in $s$ that we assume to be an edge of $\ell_1$. From \eqref{eq:gen-phi-string} and \eqref{eq:string-set-maps}, and the definition of our weights in \eqref{eq:w-wei-inf}, we have that 
\begin{equation}\label{eq:eiwuvfuowevo}
	\phi^{K}(s)= \sum_{K_1+\dots+K_n=K}\prod_{i=1}^n\phi^{K_i}(\ell_i).
\end{equation}
From Theorem~\ref{thm: fixed K 't Hooft master loop equation for surface sum}, we have that (recall that we assumed that $\mathbf{e}\in\ell_1$)
\begin{align*}
	\phi^{K_1}(\ell_{1}) =&\sum_{\{\ell_{1,2},\ell_{1,2}\}\in \SS_{-}(\mathbf{e},\ell_1)}
	\sum_{K_{1,1}+K_{1,2}=K_1}
	\phi^{K_{1,1}}(\ell_{1,1})\phi^{K_{1,2}}(\ell_{1,2})\\
	&- \sum_{\{\ell_{1,1},\ell_{1,2}\}\in \SS_{+}(\mathbf{e},\ell_1)}
	\sum_{K_{1,1}+K_{1,2}=K_1}\phi^{K_{1,1}}(\ell_{1,1})\phi^{K_{1,2}}(\ell_{1,2}) \\
	&+ \upbeta
	\sum_{p\in \cP_{\ZZ^d}(e^{-1},K)}
	\phi^{K_1\sm p}(\ell_1 \ominus_{\mathbf{e}}p)
	-\upbeta\sum_{q\in \cP_{\ZZ^d}(e,K)}
	\phi^{K_1\sm q}(\ell_1 \oplus_{\mathbf{e}}q).
\end{align*}
Substituting this equation in \eqref{eq:eiwuvfuowevo}, and rearranging the order of the sums and products, we get that 
\begin{align*}
	\phi^{K}(s)=&
	\sum_{K_{1,1}+K_{1,2}+K_2+\dots+K_n=K}\sum_{\{\ell_{1,2},\ell_{1_2}\}\in \SS_{{-}}(\mathbf{e},\ell_1)}
	\left(\phi^{K_{1,1}}(\ell_{1,1})\phi^{K_{1,2}}(\ell_{1,2})\prod_{i=2}^n\phi^{K_i}(\ell_i)\right)\\
	&- 
	\sum_{K_{1,1}+K_{1,2}+K_2+\dots+K_n=K}\sum_{\{\ell_{1,1},\ell_{1,2}\}\in \SS_{{+}}(\mathbf{e},\ell_1)}\left(\phi^{K_{1,1}}(\ell_{1,1})\phi^{K_{1,2}}(\ell_{1,2})\prod_{i=2}^n\phi^{K_i}(\ell_i)\right) \\
	&+ \upbeta \sum_{K_1+K_2+\dots+K_n=K}
	\sum_{p\in \cP_{\ZZ^d}(e^{-1},K)}
	\left(\phi^{K_1\sm p}(\ell_1 \ominus_{\mathbf{e}}p)\prod_{i=2}^n\phi^{K_i}(\ell_i)\right)\\
	&-\upbeta \sum_{K_1+K_2+\dots+K_n=K}\sum_{q\in \cP_{\ZZ^d}(e,K)}
	\left(\phi^{K_1\sm q}(\ell_1 \oplus_{\mathbf{e}}q)\prod_{i=2}^n\phi^{K_i}(\ell_i)\right).
\end{align*}
Using once again \eqref{eq:gen-phi-string}, \eqref{eq:string-set-maps} and the definition of our weights in \eqref{eq:w-wei-inf}, we get \eqref{eq:limiting_master_loop_eq-2222}.
\end{proof}

\subsection{Wilson loop expectations in the large-\texorpdfstring{$N$}{N} limit}\label{subsec: large N Master loop equation in large-$N$ limit}

In this section, we assume that Theorem~\ref{thm: fixed K 't Hooft master loop equation for surface sum} (and thus Corollary~\ref{cor:mle-goal}) holds.

To prove Theorem~\ref{thm: surface sum representation in 't hooft limit}, we show that for a fixed increasing sequence $\Lambda_N$ converging to $\ZZ^d$, there exists $\upbeta_0(d)>0$ such that whenever $|\upbeta|\leq \upbeta_0(d)$,
\begin{equation*}
\phi_{\Lambda_N,N,\upbeta}(s)\xrightarrow[N\to\infty]{}\phi(s):=\prod_{i=1}^n\phi(\ell_i), \quad \text{for all $s=\{\ell_1,\dots,\ell_n\}$},
\end{equation*}
where we recall that $\phi(s)$ is precisely defined in the statement of Theorem~\ref{thm: surface sum representation in 't hooft limit}.
To do this, we show that:
\begin{enumerate}	
\item any sub-sequential limit of $\phi_{\Lambda_N,N,\upbeta}(s)$ satisfies the master loop equation \eqref{eq:limiting_master_loop_eq} (Proposition~\ref{prop: single-location master loop equation in the 't Hooft limit});
\item for sufficiently small $\upbeta$ the master loop equation \eqref{eq:limiting_master_loop_eq} has a unique solution (Proposition~\ref{prop: uniqueness of solution for single-location master loop equation in the 't Hooft limit});
\item $\phi(s)$ satisfies the master loop equation \eqref{eq:limiting_master_loop_eq}. (This will follow from the assumption that Theorem~\ref{thm: fixed K 't Hooft master loop equation for surface sum} holds.)
\end{enumerate}

\subsubsection{Proof of the large-\texorpdfstring{$N$}{N} limit for Wilson loop expectations}

First we need to understand the form of the master loop equation in the large-$N$ limit. That is, what recursive relation do subsequential limits of $\phi_{\Lambda_N,N,\upbeta}(s)$ satisfy. The next theorem gives us this relation. 

We recall, from Section~\ref{subsec: lattice Yang--Mills}, that a string of cardinality $n$ has been defined as a \emph{multiset} of loops $\{\ell_1,\dots,\ell_n\}$.

\begin{defn}
Let $\cS$ denote the set of all possible strings of any finite cardinality.
\end{defn}

\begin{prop}[\textsc{single-location master loop equation in the large-n limit}]\label{prop: single-location master loop equation in the 't Hooft limit}
Let $\Lambda_N$ be an increasing sequence converging to $\ZZ^d$. Suppose that there exists a subsequence of $N$ such that $\phi_{\Lambda_N,N,\upbeta}(s)$ converges to $\phi_{\infty}(s)$ for all strings $s\in\cS$.  Then, for every fixed non-empty string $s$ and every fixed edge $\mathbf{e}$ in $s$,
\begin{align}\label{eq:limiting_master_loop_eq}
	\phi_{\infty}(s) &= \sum_{s'\in \SS_-(\mathbf{e},s)}\phi_{\infty}(s') - \sum_{s'\in \SS_+(\mathbf{e},s)}\phi_{\infty}(s') \notag
	\\&+ \upbeta\sum_{s'\in \DD_-(\mathbf{e},s)}\phi_{\infty}(s') - \upbeta\sum_{s'\in \DD_+(\mathbf{e},s)}\phi_{\infty}(s').
\end{align}
\end{prop}

\begin{proof}
If we take the limit of the equation in Theorem~\ref{thm: finite N master loop equation} (satisfied by $\phi_{\Lambda_N,N, \upbeta}(s)$) along the subsequence in our theorem statement, we get the desired result as long as we can show that the merger terms vanish (note we can exchange the order of the limit and sum as there are only a finite number of splitting and deformations). To see that the merger terms vanish, notice that $|\phi_{\Lambda_N,N,\upbeta}(s)|\leq 1$ for any string $s$ contained in $\Lambda_N$ as the eigenvalues of unitary matrices have modulus one and our Wilson loop observable are defined in terms of the normalized trace (recall \eqref{eq: Wilson loop observables}). Thus, since $\Lambda_N$ increases to $\ZZ^d$, it will eventually contain any fixed string $s$. Since for a fixed string $s$ there are only a finite number of possible mergers, this bound on $|\phi_{\Lambda_N,N,\upbeta}(s)|$ gives us that the factors $\frac{1}{N^2}$ will make these terms vanish.
\end{proof}

The next proposition (which immediately follows from \cite[Theorem 9.2]{chatterjee_rigorous_2019}) shows that the master loop equation in the large-$N$ limit has a unique solution for sufficiently small $\upbeta$.

\begin{prop}[\textsc{uniqueness of solution for single-location master loop equation in the large-n limit}]\label{prop: uniqueness of solution for single-location master loop equation in the 't Hooft limit}
Given any $L\geq 1$, there exists $\upbeta_*(L,d)>0$ such that if $|\upbeta|\leq \upbeta_*(L,d)$, then there is a unique function $\phi_{\infty}:\cS\to \RR$ such that 
\begin{enumerate}
	\item $\phi_{\infty}(\emptyset) = 1$;
	\item $\phi_{\infty}$ is invariant under backtrack erasures, i.e.\ for any loop $\ell = \pi_1ee^{-1} \pi_2$, we have that $\phi_{\infty}(\ell) = \phi_{\infty}(\pi_1\pi_2)$;
	\item $|\phi_{\infty}(s)|\leq L^{|s|}$, for all $s\in \cS$;
	\item $\phi_{\infty}(s)$ satisfies the master loop equation \eqref{eq:limiting_master_loop_eq},  for all non-empty $s\in \cS$.
\end{enumerate}
\end{prop}

\begin{proof}
Suppose that there are two functions $\phi_{\infty}(s)$ and $\gamma_{\infty}(s)$ that satisfy the conditions of the theorem. Fix a string $s\in\mathcal{S}$ and  a loop $\ell$ in $s$. Let $e\in \ZZ^d$ be an unoriented lattice edge such that $\ell$ contains at least one copy of this edge (in one of the two possible orientations). Suppose that there are $m$ different edges in $\ell$ that correspond to the edge $e$ (in one of the two possible orientations).  Applying \eqref{eq:limiting_master_loop_eq} at every edge $\mathbf{e}$ in $\ell$ that corresponds to $e$, we get that $\phi_{\infty}(s)$ and $\gamma_{\infty}(s)$ both satisfy the conditions\footnote{Recall Remark~\ref{rmk: beta factor of two} when transfering result from our paper to the ones in \cite{chatterjee_rigorous_2019}.} of \cite[Theorem 9.2]{chatterjee_rigorous_2019} giving us that $\phi_{\infty}(s) = \gamma_{\infty}(s)$. Here, we remark that our assumption that $\phi_{\infty}$ is invariant under backtrack erasures is needed because the Master loop equation in \cite{chatterjee_rigorous_2019} is stated for loop operations with all backtracks removed, whereas in our master loop equation \eqref{eq:limiting_master_loop_eq}, we do not necessarily remove backtracks.
\end{proof}

We have the following desired consequence.

\begin{cor}\label{cor:almostdone}
Let $\Lambda_N$ be an increasing sequence converging to $\ZZ^d$.
There exists a number $\upbeta_1(d)>0$, depending only on the dimension $d$, such that the following is true. If $|\upbeta|\leq \upbeta_1(d)$, then, for any string $s\in\cS$,
\begin{equation*}
	\text{$\phi_{\Lambda_N,N,\upbeta}(s)$ converges to a limit $\phi_{\infty}(s)$ as $N\to\infty$.}
\end{equation*} 
Moreover, $\phi_{\infty}(\emptyset)=1$, $\phi_{\infty}$ is invariant under backtrack erasures, $|\phi_{\infty}(s)|\leq 1$ for all $s\in\cS$, and $\phi_{\infty}(s)$ is the unique solution to the master loop equation \eqref{eq:limiting_master_loop_eq} for all non-empty $s\in \cS$.
\end{cor}

\begin{proof}
We use the compact notation $\phi_{N}(s):=\phi_{\Lambda_N,N,\upbeta}(s)$. We will show that the statement of the corollary holds by setting $\upbeta_1(d):=\upbeta_*(1,d)$ from  Proposition~\ref{prop: uniqueness of solution for single-location master loop equation in the 't Hooft limit}.
To do this, we prove that every subsequence $N_k$ of $N$ has a further subsequence $N_{k_j}$ such that if $|\upbeta|\leq \upbeta_1(d)$ then for all $s\in\cS$,  
\begin{equation*}
	\text{$\phi_{N_{k_j}}(s)$ converges to $\phi_{\infty}(s)$},
\end{equation*}
where $\phi_{\infty}(s)$ is as in the lemma statement. This would be enough to conclude by standard arguments.

Fix a subsequence $N_k$ of $N$. Recall that $|\phi_{N_k}(s)|\leq 1$ for any string $s$ contained in $\Lambda_{N_k}$. Since $\Lambda_{N_k}$ increases to $\ZZ^d$, it will eventually contain any $s$. Therefore, a standard diagonal argument, gives the existence of a subsequence  $N_{k_j}$ along
which the limit of $\phi_{N_{k_j}}(s)$ exists for all $s\in\cS$. Call this limit $\hat{\phi}_{\infty}(s)$. It remains to prove that $\hat{\phi}_{\infty}(s)=\phi_{\infty}(s)$. Note that
\begin{enumerate}
	\item by the discussion below \eqref{eq: Wilson loop observables}, we have that by definition $|\phi_{N}(\emptyset)|= 1$ for all  $N\geq1$, and so we must also have that $\hat{\phi}_{\infty}(\emptyset)=1$;
	\item for any loop $\ell = \pi_1ee^{-1} \pi_2$, by defintion of lattice Yang--Mills theory, we have that $\phi_{N}(\ell) = \phi_{N}(\pi_1\pi_2)$, and so we must also have that $\hat{\phi}_{\infty}$ is invariant under backtrack erasures;
	\item by construction, $|\hat{\phi}_{\infty}(s)|\leq 1$ for all $s\in\cS$;
	\item by Proposition~\ref{prop: single-location master loop equation in the 't Hooft limit}, $\hat{\phi}_{\infty}(s)$ solves the master loop equation \eqref{eq:limiting_master_loop_eq} for all non-empty $s\in\cS$.
\end{enumerate}
Hence, by the uniqueness in Proposition~\ref{prop: uniqueness of solution for single-location master loop equation in the 't Hooft limit} with $L=1$, we get that  $\hat{\phi}_{\infty}(s)=\phi_{\infty}(s)$ for all $|\upbeta|\leq \upbeta_1(d)$ and all strings $s\in\cS$.
\end{proof}

Note that in the above proof we only used Proposition~\ref{prop: uniqueness of solution for single-location master loop equation in the 't Hooft limit} with $L=1$. The general $L$ case is used to conclude that $\phi_{\infty}$ from Corollary~\ref{cor:almostdone} is actually equal to the explicit $\phi$ we introduced in Theorem~\ref{thm: surface sum representation in 't hooft limit}.

We now show that $\phi_{\infty}=\phi$ by showing that $\phi$ satisfies the conditions of Proposition~\ref{prop: uniqueness of solution for single-location master loop equation in the 't Hooft limit}.
Recall by the comments immediately after Theorem~\ref{thm: surface sum representation in 't hooft limit} that $\phi(s)$ is defined such that
\begin{equation}\label{eq:condition-1}
\phi(\emptyset)=1,
\end{equation}
so the first condition holds. The second condition concerning backtrack erasure invariance follows by Lemma \ref{lemma: backtrack cancellations} and the next estimate, which also shows the third condition. The proof is postponed to Section~\ref{Asec: proof of uniqness to limiting master loop equation}.
\begin{lem}\label{lemma:phi-bound}
There exists  $L\geq 1$ large enough and $\upbeta_{2}(d)>0$ such that if $|\upbeta| \leq \upbeta_2(d)$ then $\phi(s)$ is absolutely convergent and
\begin{equation}\label{eq:condition-2}
	|\phi(s)| \leq L^{|s|},\quad\text{ for all $s\in \cS$}.
\end{equation}
\end{lem}

Finally, since we know from the above lemma that if $|\upbeta| \leq \upbeta_2(d)$ then $\phi(s)$ is absolutely convergent, we can immediately deduce from Corollary~\ref{cor:mle-goal} (which holds under our assumption) that, whenever $|\upbeta| \leq \upbeta_2(d)$, $\phi(s)$ satisfies the master loop equation \eqref{eq:limiting_master_loop_eq} for all non-empty $s\in \cS$.
Combining this with~\eqref{eq:condition-1}~and~\eqref{eq:condition-2}, we can conclude from the general $L$ case of Proposition~\ref{prop: uniqueness of solution for single-location master loop equation in the 't Hooft limit} that $\phi_{\infty}=\phi$ for all $|\upbeta|\leq \upbeta_0(d)$, where
\begin{equation}\label{eq:beta_0-def}
\upbeta_0(d):=\min\{\upbeta_1(d),\upbeta_2(d)\}.
\end{equation}

This completes the proof of Theorem~\ref{thm: surface sum representation in 't hooft limit}, assuming that Theorem~\ref{thm: fixed K 't Hooft master loop equation for surface sum} holds.
In the next subsection we give the missing proof of Lemma~\ref{lemma:phi-bound}. Theorem~\ref{thm: fixed K 't Hooft master loop equation for surface sum} will be proved in Section~\ref{subsec: surface sum Satisfies Master Loop Equation}, after  introducing our  ``peeling exploration'' in Section~\ref{subsubsec: Splittings and Deformations on embedded maps}.

\subsubsection{Absolute summability}\label{Asec: proof of uniqness to limiting master loop equation}

In this section, we prove Lemma \ref{lemma:phi-bound} concerning the absolute summability of 
\[\phi(\cdot)=\sum_{K:\cP_{\ZZ^d}\to\NN} \phi^K(\cdot),\] 
where we recall that the infinite sum $\phi(\cdot)$ is only over finite plaquette assignments $K$, i.e.\ plaquette assignments such that $\sum_{p\in\cP_{\ZZ^d}}K(p)<\infty$.
We remark at the beginning that while absolute summability is a crucial result, the arguments are mostly of a technical nature, and are quite distinct from the main conceptual arguments of the paper. Thus, we recommend the reader to skip Section \ref{Asec: proof of uniqness to limiting master loop equation} on a first reading.

The primary step is to obtain a bound like 
$$|\phi^K(s)|  \leq C^{|s|} (C \upbeta d)^{\area(K)}$$ 
for some constant $C$, where
$\area(K):=\sum_{p\in\cP_{\ZZ^d}}K(p).$ This goal will be achieved in Proposition~\ref{prop:-technical-bound-on-phi-K}.
The proof of this estimate will be an adaptation of the fixed point argument in \cite{chatterjee_rigorous_2019}. The main idea is to view the master loop equation (from Corollary \ref{cor:mle-goal}) as a fixed point equation, and then to derive appropriate estimates for the associated fixed point map. First, we introduce some preliminary notation and a notion of norm that will be used in the following; in Remark~\ref{rmk:expl-norm} we will explain the motivations behind our choice of norm.

It will be convenient in this section to view strings as ordered collections of rooted loops. With this in mind, we make the following definition.

\begin{defn}
Let $\cS_o$ be the set of finite ordered tuples of rooted loops $s = (\ell_1, \ldots, \ell_n)$. Rooted loops are loops with a distinguished starting point.
\end{defn}	

In this subsection, when we refer to ``loops'', by default we refer to rooted loops, and when we refer to ``strings'', we refer to ordered tuples of rooted loops. We prefer to introduce this slight abuse of notation, rather than always saying ``ordered string of rooted loops''. We remark that it is often convenient in combinatorics to order unlabeled collections of objects so as to avoid having to consider combinatorial factors. The ensuing arguments are an example of this.

To begin the discussion, let $\Delta$ be the set of all finite sequences of strictly positive integers, plus the null sequence. Given $\delta = (\delta_1,\dots,\delta_n)\in \Delta$, define\begin{align*}
|\delta| ~:= \sum_{i=1}^n\delta_i, \hspace{0.5cm} \#\delta := n, \hspace{0.5cm} \iota(\delta) = |\delta|-\#\delta.
\end{align*}
All these quantities are defined to be $0$ for the null sequence. Given two non-null elements $\delta=(\delta_1,\dots,\delta_n)\in \delta$ and $\delta' = (\delta'_1,\dots,\delta'_m)\in \delta$ we say that $\delta\leq \delta'$ if $m=n$ and $\delta_i\leq \delta'_i$ for each $i$. Given a string $s = (s_1, \ldots, s_n) \in \cS_o$, define $\delta(s) := (|s_1|, \ldots, |s_n|)$ (if $s$ is the null string, then $\delta(s)$ is defined to be the null sequence).

Let $\Delta^+$ be the subset of $\Delta$ consisting of non-null elements whose components are all at least $2$. Since we are working with the usual integer lattice $\ZZ^d$, if $s$ is a non-null string then $\delta(s)\in \Delta^+$. For $\lambda < 1/2$, we have that
\begin{equation}\label{eq:sum-delta}
\sum_{\delta \in \Delta^+} \lambda^{\iota(\delta)} = \sum_{n=1}^\infty \sum_{\delta_1, \ldots, \delta_n \geq 2} \lambda^{(\delta_1 + \cdots + \delta_n) - n} = \sum_{n=1}^\infty \lambda^{-n} \bigg(\frac{\lambda^2}{1-\lambda}\bigg)^n = \frac{\lambda}{1 - 2\lambda}.
\end{equation}
We set and recall some convenient notation to be used in the following:
\begin{itemize}
\item Given a plaquette assignment $K : \cP_{\ZZ^d} \rightarrow \NN$, we defined $\area(K) := \sum_{p \in \cP_{\ZZ^d}} K(p)$.
\item For two plaquette assignments $K, K' : \cP_{\ZZ^d} \rightarrow \NN$, we say that $K \leq K'$ if $K(p) \leq K'(p)$ for all $p \in \cP_{\ZZ^d}$.
\end{itemize}

\begin{defn}
Given a plaquette assignment $K_* : \cP_{\ZZ^d} \rightarrow \NN$ and an integer $M_* \geq 1$, define the index set
\[
I_{K_*, M_*} ~:= \{(s, K) \,:\, s \in \cS_o,\, K: \cP_{\ZZ^d} \rightarrow \NN,\, K \leq K_*,\, |s| + 4\area(K) \leq M_*\},
\]
and define the space
\begin{equation*}
	\Theta_{K_*, M_*} := \left\{f = (f(s, K))_{(s, K) \in I_{K_*, M_*}} \,:\, f(\emptyset, K) = \mathds{1}_{\{K = 0\}}\right\}.
\end{equation*} 
\end{defn}

Note that $\Theta_{K_*, M_*}$ is simply a subset (even more, an affine subspace) of $\RR^{I_{K_*, M_*}}$ and the function $f_\phi(s, K) := \phi^K(s)$ is an element of $\Theta_{K_*, M_*}$, where $\phi^K$ is as in \eqref{eq:gen-phi-string}.
Next, we define a carefully chosen norm on $\RR^{I_{K_*, M_*}}$. We remark that we will only ever apply the norm to elements of $\Theta_{K_*, M_*} \subseteq \RR^{I_{K_*, M_*}}$. However, our terminology of ``norm'' is only accurate when viewed as a function defined on $\RR^{I_{K_*, M_*}}$, since this is a vector space while $\Theta_{K_*, M_*}$ is only an affine space.

For $\gamma \in (0, \infty)$, let
\begin{equation}\label{eq:defn-Dgf}
(D_\gamma f)(\delta) ~:= \sup_{\substack{s \in \cS_o \\ \delta(s) \leq \delta}} \sum_{K:(s,K)\in I_{K_*, M_*}} \gamma^{\area(K)} |f(s, K)|, \qquad\forall f \in \RR^{I_{K_*, M_*}}, \delta \in \Delta_+.
\end{equation}
For $\lambda \in (0, 1/2)$, $\gamma \in (0, \infty)$, define the norm 
\begin{equation}\label{eq:lambda-gamma-norm}
\|f\|_{\lambda, \gamma} ~:= \sum_{\delta \in \Delta_+} \lambda^{\iota(\delta)} (D_\gamma f)(\delta), \qquad\forall f \in \RR^{I_{K_*, M_*}}.
\end{equation}
Here, even if $\|\cdot\|_{\lambda, \gamma}$ does not satisfy all the mathematical properties of a norm (it may be infinite), we will still refer to it as one, because this is really how we think about it. 

Note that the norm $\|\cdot\|_{\lambda, \gamma}$ also depends on $K_*, M_*$, but we keep this dependence implicit. In the end, the motivation for this norm is that it satisfies the estimate in Lemma \ref{lemma:M-contraction-map}, which we will get to later.

\begin{rmk}[\textsc{comments on the norm}]\label{rmk:expl-norm}
One can think of the norm \eqref{eq:lambda-gamma-norm} as a mixed $\ell^1$ and $\ell^\infty$ norm. Intuitively, one should think of $\lambda$ as a small parameter and $\gamma$ as a large parameter greater than 1 (the latter will only be true when $\upbeta$ is small enough -- see the proof of Proposition \ref{prop:-technical-bound-on-phi-K}), so that $$\|f\|_{\lambda, \gamma} \leq C\quad\text{ implies that}\quad|f(s, K)| \leq C \lambda^{-\iota(s)} \gamma^{-\area(K)},$$ 
i.e.\ $f$ decays exponentially in $\area(K)$ and grows at most exponentially in $\iota(s)$ (a quantity which should be thought of as a proxy for the length of $s$).

Perhaps a more natural definition of the norm would have been to use $\lambda^{|\delta|}$ instead of $\lambda^{\iota(\delta)}$ in \eqref{eq:lambda-gamma-norm}, i.e.\ to just weight by the total length. However, this norm turns out to be too weak to close the ensuing contraction mapping argument. The problem is that if $s'$ is a positive splitting of $s$, then it could be the case that $|s'| = |s|$, while it is always the case that $\iota(\delta(s')) \leq \iota(\delta(s)) - 1$ as we will show in Lemma~\ref{lemma:delta-map}. The latter estimate allows us to gain a crucial factor of $\lambda$ (which recall we are thinking of as small).
\end{rmk}

The reason we take finite parameters $K_*, M_*$ is so that we have the following soft estimate.

\begin{lem}\label{lemma:phi-finite-norm}
Let $K_* : \cP_{\ZZ^d} \rightarrow \NN$ be a plaquette assignment and $M_* \geq 1$ an integer. Let $f_\phi(s, K) := \phi^K(s)$, where $\phi^K(s)$ is as in \eqref{eq:gen-phi-string}. For any $\lambda \in (0, 1/2)$ and $\gamma \in (0, \infty)$, we have that
\[ \|f_\phi\|_{\lambda, \gamma} < \infty,\]
where we recall that the norm $\|\cdot\|_{\lambda, \gamma}$ is defined using the implicit parameters $K_*$ and $M_*$, which we fixed.
\end{lem}
\begin{proof}
We claim that there are only finitely many pairs $(s,K)\in I_{K_*, M_*}$ such that $\phi^K(s)$ is nonzero.

To see this, first we observe that if $s$ is a string such that no edge of $s$ is contained in a plaquette $p$ for which $K_*(p) > 0$, then $\phi^K(s) = 0$ for all $K \leq K_*$. This follows because in this case, $\phi^K$ is an empty sum, because there are no maps in $\npe(s,K)$. Thus, if $\phi^K(s)$ is nonzero, then one of its edges must be contained in a plaquette $p$ for which $K_*(p) > 0$. Since $K_*$ is finite, there are only finitely many strings $s \in \cS_0$ for which this holds and such that 
$|s| + 4\area(K) \leq M_*.$ This concludes the proof of our claim.

Thus, it follows that there is some constant $C$ (depending on $\lambda, \gamma, K_*, M_*, d)$ such that
\[ \sum_{K:(s,K)\in I_{K_*, M_*}} \gamma^{area(K)} |f_{\phi}(s, K)|  \leq C, \qquad \forall s \in \cS_o. \]
Using this, we may bound the norm
\[ \|f_\phi\|_{\lambda, \gamma} \leq C \sum_{\delta \in \Delta_+} \lambda^{\iota(\delta)} < \infty, \]
where we used that $\lambda < 1/2$ and \eqref{eq:sum-delta}.
\end{proof}

Next, we define a relevant mapping on the space $\Theta_{K_*, M_*}$. To relate back to our discussion in Section \ref{sec: intro}, one should think of this as the mapping $\eta \mapsto G_\beta \eta + F$ discussed in Remark \ref{remark:comparison-with-chatterjee-vst}. For each non-null string $s = (\ell_1, \ldots, \ell_n) \in \cS_o$, let $\mathbf{e}_s$ be the first edge of $\ell_1$. 
Let $e_s$ denote the lattice edge which $\mathbf{e}_s$ is mapped to. Define the mapping 
\[M : \Theta_{K_*, M_*} \rightarrow \Theta_{K_*, M_*},\]
which sets, for all $(s,K)\in I_{K_*, M_*}$ with $s$ a non-null string, $(Mf)(s, K)$  to be the right-hand side of the fixed-$K$ master loop equation (Theorem~\ref{thm: fixed K 't Hooft master loop equation for surface sum}; recall also the equivalent version in \eqref{eq:new-form}) at the edge $\mathbf{e}_s$, i.e.\ 

\begin{equation}\label{eq:fixed-point-map-M-def}
\begin{aligned}
	(Mf)(s, K) := \mp \sum_{s' \in \SS_{\pm}(\mathbf{e}_s, s)} f(s', K) &+ \upbeta \sum_{p \in \cP_{\ZZ^d}(e_s^{-1}, K)} f(s \ominus_{\mathbf{e}_s} p, K \sm p)\\
	&-\upbeta \sum_{q \in \cP_{\ZZ^d}(e_s, K)} f(s \oplus_{\mathbf{e}_s} q, K \sm q).
\end{aligned} 
\end{equation}
When $s = \emptyset$, we set $(Mf)(s, K) := \mathds{1}_{\{K = 0\}}$, as required in the definition of $\Theta_{K_*, M_*}$. Observe that this map is well-defined, because if $(s, K)$ is such that $|s| + 4\area(K) \leq M_*$, then the same is true for $(s', K)$ for any (positive or negative) splitting $s'$ of $s$, as well as for $(s', K')$ for any (positive or negative) deformation $s'$ using a plaquette $q$ or $p$. In our notation, we omit the dependence of $M$ on $\upbeta$.

Here, we specify that by default, we do not erase backtracks, i.e.\ the strings $s'$, obtained by performing a loop operation to $s$, may still have backtracks. We also specify that $s' \in \SS_{\pm}(\mathbf{e}_s, s)$ is an ordered string of the form $(\ell'_{1, 1}, \ell'_{1, 2}, \ell_2, \ldots, \ell_n)$ or $(\ell_{1, 1}', \ell_2, \ldots, \ell_n)$ -- the latter case may happen in a negative splitting. That is, we always take the ordering of $s'$ so that the first string of $s$ is split into the first one or two strings of $s'$, while preserving the order of all the remaining strings.

The following lemma gives the key estimate for the mapping $M$.

\begin{lem}\label{lemma:M-contraction-map}
Let $K_* : \cP_{\ZZ^d} \rightarrow \NN$ be a plaquette assignment and $M_* \geq 1$ an integer.
For any $\lambda \in (0, 1/2)$ and $\gamma \in (0, \infty)$, we have that
\begin{align}
	\|M f\|_{\lambda, \gamma} &\leq \frac{\lambda}{1-2\lambda} + (4\lambda + 4d \upbeta \gamma \lambda^{-4}) \|f\|_{\lambda, \gamma},\qquad\forall f \in \Theta_{K_*, M_*}, \label{eq:M-self-map}
\end{align}
where we recall that the norm $\|\cdot\|_{\lambda, \gamma}$ is defined using the parameters $K_*, M_*$, which we fixed.
\end{lem}

\begin{rmk}
At first glance, it may be a bit surprising that the estimate \eqref{eq:M-self-map} is uniform in $K_*, M_*$. Conceptually, one may think of the following close analogy. We have a map $M$ defined on a Banach space $V$, and we expect to be able to prove that $M$ is a contraction on this space (a form of this is shown in \cite{chatterjee_rigorous_2019}). Morally, the role of the parameters $K_*, M_*$ is to give an increasing sequence of finite-dimensional subspaces $V_1 \subseteq V_2 \subseteq \cdots$ which increase to $V$, such that for each $n \geq 1$, $M$ maps $V_n$ to $V_n$. In this analogy, Lemma \ref{lemma:M-contraction-map} amounts to proving a contraction estimate for each $V_n$. Since $M$ is supposed to be a contraction on the entire space $V$, we certainly expect to have estimates that are uniform in $n$.
\end{rmk}

Before we prove Lemma \ref{lemma:M-contraction-map}, we first show one preliminary lemma.

\begin{lem}\label{lemma:delta-map}
Let $s = (\ell_1, \ldots, \ell_n) \in \cS_o$ be a non-null string, and let $\delta = \delta(s)$. For each positive splitting $s' \in \SS_+(\mathbf{e}_s, s)$, we have that there is some $1 \leq h_{s'} \leq \delta_1 - 1$ such that 
$$\delta(s') \leq (\delta_1 - h_{s'}, h_{s'}, \delta_2, \ldots, \delta_n).$$ 
Moreover, the map $s' \mapsto h_{s'}$ on the domain $\SS_+(\mathbf{e}_s, s)$ can be chosen to be injective. 

Similarly, for each negative splitting $s' \in \SS_-(\mathbf{e}_s, s)$, either 
$$\delta(s') \leq (\delta_1 - 2, \delta_2, \ldots, \delta_n),\qquad \Big(\delta(s') \leq (\delta_2, \ldots, \delta_n)\quad  \text{if}\quad \delta_1 = 2\Big),$$ 
or there is some $1 \leq g_{s'} \leq \delta_1 - 1$ such that  
$$\delta(s') \leq (\delta_1 - g_{s'}, g_{s'}, \delta_2, \ldots, \delta_n).$$ 
Moreover, the map $s' \mapsto g_{s'}$ (whose domain is a suitable subset of $\SS_-(\mathbf{e}_s, s)$) can be taken as injective.
\end{lem}
\begin{proof}
Any positive splitting $s' \in \SS_+(\mathbf{e}_s, s)$ is obtained by splitting the first loop $\ell_1$ into two loops $\ell_{1, 1}', \ell_{1,2}'$ whose lengths $|\ell_{1, 1}'| + |\ell_{1, 2}'| ~= |\ell_1| = \delta_1$. Thus we may take $h_{s'} = |\ell_{1, 2}'|$. The injectivity of $s \mapsto h_{s'}$ follows since different splittings $s'$ result in different lengths $|\ell'_{1, 2}|$.

Similarly, any negative splitting may be obtained by either taking the loop $\ell_1$ and erasing a backtrack $ee^{-1}$ to obtain $\ell_{1, 1}'$ (which may be the null loop), or by splitting the first loop $s_1$ into two loops $\ell_{1, 1}', \ell_{1, 2}'$ whose lengths $|\ell_{1, 1}'| + |\ell_{1, 2}'|
= |\ell_1| - 2$. In the former case, we have that $\delta(s') \leq (\delta_1 - 2, \delta_2, \ldots, \delta_n)$ (or $\delta(s') \leq (\delta_2, \ldots, \delta_n)$), while in the latter case, we may take $g_{s'} = |\ell_{1, 2}'|$.
\end{proof}

\begin{proof}[Proof of Lemma \ref{lemma:M-contraction-map}]
We may assume that $\|f\|_{\lambda, \gamma} < \infty$, as otherwise the inequality is trivial. It is convenient to introduce (recall from \eqref{eq:fixed-point-map-M-def} the definition of the mapping $M$) for all $(s,K)\in I_{K_*, M_*}$ with $s$ a non-null string,
\begin{align*}
	&g_1(s,K):= \sum_{s' \in \SS_{\pm}(\mathbf{e}_s, s)} f(s', K), \quad g_2(s,K):=- \sum_{s' \in \SS_{\pm}(\mathbf{e}_s, s)} f(s', K)\\ 
	& 
	g_3(s,K):= \upbeta \sum_{p \in \cP_{\ZZ^d}(e_s^{-1}, K)} f(s \ominus_{\mathbf{e}_s} p, K \sm p), \quad
	g_4(s,K):=-\upbeta \sum_{q \in \cP_{\ZZ^d}(e_s, K)} f(s \oplus_{\mathbf{e}_s} q, K \sm q).
\end{align*}
With this notation, by definition of our norm, 
\begin{equation}\label{eq:firstbound}
	\begin{aligned}
		\|Mf\|_{\lambda, \gamma} &= \sum_{\delta \in \Delta^+} \lambda^{\iota(\delta)} (D_\gamma M f)(\delta) \\
		&\leq \sum_{\delta \in \Delta^+} \lambda^{\iota(\delta)} + \sum_{\delta \in \Delta^+} \lambda^{\iota(\delta)} \sup_{\substack{s \in \cS \\ s \neq \emptyset \\ \delta(s) \leq \delta}} \sum_{K : (s, K) \in I_{K_*, M_*}} \gamma^{\area(K)} |(Mf)(s, K)| \\
		&\leq \frac{\lambda}{1-2\lambda} + \sum_{i=1}^4\Bigg(\sum_{\delta \in \Delta^+}  \sup_{\substack{s \in \cS \\ s \neq \emptyset \\ \delta(s) \leq \delta}} \lambda^{\iota(\delta)} \sum_{K : (s, K) \in I_{K_*, M_*}} \gamma^{\area(K)} |g_i(s, K)|\Bigg).
	\end{aligned}
\end{equation}
where we used \eqref{eq:sum-delta} in the final inequality (which is where the assumption $\lambda < 1/2$ comes in). We are now going to bound each of the four terms in the final sum separately.

\medskip

\noindent\underline{\emph{Estimate for the positive splitting term $g_1$}}: Fix $\delta \in \Delta_+$ and $s \in \cS_o$ such that $\delta(s) \leq \delta$. Assume further that $s \neq \emptyset$. By the triangle inequality and exchanging the order of summation, we get that 
\begin{align}\label{eq:positive-splitting-intermediate-M-fixed-point}
	&\lambda^{\iota(\delta)} \sum_{\substack{K : (s, K) \in I_{K_*, M_*}}} \gamma^{\area(K)} \bigg|\sum_{s' \in \SS_+(\mathbf{e}_s, s)}  f(s', K)\bigg|\\
	&\leq \lambda^{\iota(\delta)} \sum_{s' \in \SS_+(\mathbf{e}_s, s)} \sum_{\substack{K : (s, K) \in I_{K_*, M_*}}} \gamma^{\area(K)} |f(s', K)|\notag\\
	&\leq \lambda^{\iota(\delta)} \sum_{s' \in \SS_+(\mathbf{e}_s, s)} \sum_{\substack{K : (s', K) \in I_{K_*, M_*}}} \gamma^{\area(K)} |f(s', K)|.\notag
\end{align}
where in the last inequality we used that $(s, K) \in I_{K_*, M_*}$ implies that $(s', K) \in I_{K_*, M_*}$ for all $s' \in \SS_+$ because $|s'| \leq |s|$.
Since $\delta(s') \leq (\delta_1 - h_{s'}, h_{s'}, \delta_2, \ldots, \delta_n)$ by Lemma~\ref{lemma:delta-map}, we have that 
\begin{equation*}
	\sum_{\substack{K : (s', K) \in I_{K_*, M_*}}} \gamma^{\area(K)} |f(s', K)|\leq
	\sup_{\substack{s' \in \cS, s' \neq \emptyset \\ \delta(s') \leq (\delta_1 - h_{s'}, h_{s'}, \delta_2, \ldots, \delta_n)}}\sum_{\substack{K : (s', K) \in I_{K_*, M_*}}} \gamma^{\area(K)} |f(s', K)|,
\end{equation*}
and, recalling \eqref{eq:defn-Dgf}, the right-hand side of the above equation is $(D_\gamma f)(\delta_1 - h, h, \delta_2, \ldots, \delta_n)$.
Moreover, $\iota(\delta)=1+\iota(\delta_1 - h_{s'}, h_{s'}, \delta_2, \ldots, \delta_n)$ by definition of $\iota$, and so it follows that
\begin{align*}
	\eqref{eq:positive-splitting-intermediate-M-fixed-point}&\leq \lambda \sum_{s' \in \SS_+(\mathbf{e}_s, s)} \lambda^{\iota(\delta_1 - h_{s'}, h_{s'}, \delta_2, \ldots, \delta_n)} (D_\gamma f)(\delta_1 - h_{s'}, h_{s'}, \delta_2, \ldots, \delta_n) \\
	&\leq  \lambda \sum_{h=1}^{\delta_1 - 1} \lambda^{\iota(\delta_1 - h, h, \delta_2, \ldots, \delta_n)} (D_\gamma f)(\delta_1 - h, h, \delta_2, \ldots, \delta_n), \notag
\end{align*}
where for the last inequality, we used that the map $s' \mapsto h_{s'}$ is injective, again by Lemma \ref{lemma:delta-map}.
Upon taking sup over $s \neq \emptyset$ with $\delta(s) \leq \delta$ and then summing in $\delta \in \Delta^+$, we may thus obtain
\begin{equation}\label{eq:M-fixed-point-splitting-term-intermediate-bound}
	\sum_{\delta \in \delta^+}  \sup_{\substack{s \in \cS_o \\ s \neq \emptyset \\ \delta(s) \leq  \delta}}\eqref{eq:positive-splitting-intermediate-M-fixed-point}  \leq  \lambda \|f\|_{\lambda, \gamma},
\end{equation}

\medskip

\noindent\underline{\emph{Estimate for the negative splitting term $g_2$}}: It is handled in a similar way, but we prefer to spell out the details. Recall the second part of Lemma \ref{lemma:delta-map}. We bound
\begin{align}
	\lambda^{\iota(\delta)} \sum_{\substack{K : (s, K) \in I_{K_*, M_*}}} &\gamma^{\area(K)} \bigg|\sum_{s' \in \SS_-(\mathbf{e}_s, s)} f(s', K)\bigg| \label{eq:negative-splitting-intermediate-M-fixed-point}\\
	&\leq \lambda^{\iota(\delta)} \sum_{s' \in \SS_-(\mathbf{e}_s, s)} \sum_{\substack{K : (s', K) \in I_{K_*, M_*}}} \gamma^{\area(K)} |f(s', K)| \\
	&\leq I_1 + I_2 + I_3, \notag
\end{align}
where
\begin{align}
	I_1 &\leq \lambda \sum_{h=1}^{\delta_1 - 1} \lambda^{\iota(\delta_1 - h, h, \delta_2, \ldots, \delta_n)} (D_\gamma f)(\delta_1 - h, h, \delta_2, \ldots, \delta_n), \notag\\
	I_2 &\leq  \mathbbm{1}(\delta_1 > 2) \lambda\times \lambda^{\iota(\delta_1- 2, \delta_2, \ldots, \delta_n)} (D_\gamma f)(\delta_1 - 2, \delta_2, \ldots, \delta_n), \notag \\
	I_3 & \leq \mathbbm{1}(\delta_1 = 2) \lambda \times \lambda^{\iota(\delta_2, \ldots, \delta_n)} (D_\gamma f)(\delta_2, \ldots, \delta_n). \notag
\end{align}
Taking sup over $s \neq \emptyset$ with $\delta(s) \leq \delta$ and then summing in $\delta \in \Delta^+$, we thus obtain
\begin{equation}\label{eq:M-fixed-point-negative-splitting-term-intermediate-boun}
	\sum_{\delta \in\Delta^+} \sup_{\substack{s \in \cS_o \\ s \neq \emptyset \\ \delta(s) \leq \delta}} \eqref{eq:negative-splitting-intermediate-M-fixed-point} \leq 3 \lambda \|f\|_{\lambda, \gamma}.
\end{equation}

\medskip

\noindent\underline{\emph{Estimate for the negative deformation term $g_3$}}:  For any $s' \in \DD_-(\mathbf{e}_s, s)$, we have that $\delta(s') \leq (\delta_1 + 2, \delta_2, \ldots, \delta_n) \leq (\delta_1 + 4, \delta_2, \ldots, \delta_n)$. Noting that $|s \ominus_{\bf{e}_s} p| + 4\area(K\sm p) \leq |s| + 4\area(K)$, we thus have that 
\begin{align}\label{eq:negative-deformation-intermediate-M-fixed-point}
	\lambda^{\iota(\delta)} &\sum_{\substack{K : (s, K) \in I_{K_*, M_*}}} \gamma^{\area(K)} \bigg|\sum_{p \in \cP_{\ZZ^d}(e_s^{-1}, K)} f(s \ominus_{\mathbf{e}_s} p, K \sm p)\bigg|  \notag\\
	&\leq \gamma \lambda^{-4} \sum_{p \in \cP_{\ZZ^d}(e_s^{-1})} \lambda^{\iota(\delta_1 + 4, \delta_2, \ldots, \delta_n)} \sum_{\substack{K:K \leq K_*, K(p) \geq 1 \\ (s \ominus_{\bf{e}_s} p, 4\area(K)) \in I_{K_*, M_*}}} \gamma^{\area(K \sm p)} |f(s \ominus_{\mathbf{e}_s} p, K \sm p)|\notag\\
	&\leq \gamma \lambda^{-4} \sum_{p \in \cP_{\ZZ^d}(e_s^{-1})} \lambda^{\iota(\delta_1 + 4, \delta_2, \ldots, \delta_n)} (D_\gamma f)(\delta_1 + 4, \delta_2, \ldots, \delta_n)  \notag\\
	&\leq 2d \gamma \lambda^{-4} \lambda^{\iota(\delta_1 + 4, \delta_2, \ldots, \delta_n)} (D_\gamma f)(\delta_1 + 4, \delta_2, \ldots, \delta_n).
\end{align}
Here, the $2d$ factor arises because the number of oriented plaquettes containing any given oriented edge $e$ is upper bounded by $2d$. Taking sup over $s \neq \emptyset$ with $\delta(s) \leq \delta$ with and then summing in $\delta \in \Delta^+$, we thus obtain
\begin{equation}\label{eq:M-fixed-point-merger-term-intermediate-bound}
	\sum_{\delta \in \Delta^+} \sup_{\substack{s \in \cS_o \\ s \neq \emptyset \\ \delta(s) \leq \delta}} \eqref{eq:negative-deformation-intermediate-M-fixed-point} \leq 2d \gamma \lambda^{-4} \|f\|_{\lambda, \gamma}.
\end{equation}
\medskip

\noindent\underline{\emph{Estimate for the positive deformation term $g_4$}}: This term may be handled similarly, resulting in the same bound as above.

\medskip

Combining the four estimates above and recalling \eqref{eq:firstbound}, we may thus obtain
\begin{equation*}
	\|Mf\|_{\lambda, \gamma} \leq \frac{\lambda}{1-2\lambda} + (4\lambda + 4d \upbeta \gamma \lambda^{-4}) \|f\|_{\lambda, \gamma},
\end{equation*}
as desired.
\end{proof}

The main outcome of Lemma \ref{lemma:M-contraction-map} is the following proposition, which is the key step towards the proof of Lemma \ref{lemma:phi-bound}.

\begin{prop}\label{prop:-technical-bound-on-phi-K}
Let $\phi^K(\cdot)$ be as defined in \eqref{eq:gen-phi-string}. There is a constant $C = C_d$ depending only on $d$ such that
\[
|\phi^K(s)|  \leq C^{|s|} (C\upbeta)^{\area(K)},
\]
for all strings $s$ and finite plaquette assignments  $K:\cP_{\ZZ^d}\to\NN$.
\end{prop}
\begin{proof}
Fix $K_*, M_*$. For ordered strings $s \in \cS_o$, define $f_\phi(s, K) := \phi^K(s)$, where we abuse notation and take the $s$ appearing in $\phi^K(s)$ to also denote the multiset of (unrooted) loops corresponding to $s$.

By Theorem \ref{thm: fixed K 't Hooft master loop equation for surface sum}, we have that $M f_\phi = f_\phi$. Applying Lemma \ref{lemma:M-contraction-map} and taking $\lambda_0 = 10^{-2}$, $\gamma_0 = 10^{-10} / (\upbeta d)$, we have that 
\[
\|f_\phi\|_{\lambda_0, \gamma_0} \leq 2 + \frac{1}{2} \|f_\phi\|_{\lambda_0, \gamma_0}, \text{ and thus } \|f_\phi\|_{\lambda_0, \gamma_0} \leq 4.
\]
In the above, we used that $\|f_\phi\|_{\lambda_0, \gamma_0} < \infty$ (by Lemma \ref{lemma:phi-finite-norm}). Now for any $(s, K)$, we may take $K_*, M_*$ so that $(s, K) \in I_{K_*, M_*}$, and then applying the above estimate, we obtain
\[ \lambda_0^{\iota(s)} \gamma_0^{\area(K)} |\phi^K(s)|  \leq 4, \text{ and thus } |\phi^K(s)|  \leq 4\lambda_0^{-\iota(s)} \gamma_0^{-\area(K)}, \]
which is the desired estimate (with the constant $C = 10^{10}$, say). 
\end{proof}

Before we get to the proof of Lemma \ref{lemma:phi-bound} via an application of Proposition \ref{prop:-technical-bound-on-phi-K}, we need some preliminary results regarding the enumeration of plaquette assignments.

\begin{defn}\label{defn:connected}
Let $K : \cP_\Lambda \rightarrow \NN$ be a plaquette assignment. We say that $K$ is {\bf connected} if its support $\mathrm{supp}(K) := \{p \in \cP_\Lambda^+ : K(p) \text{ or } K(p^{-1}) \neq 0\}$ is connected, in the sense that any two plaquettes $p, p'$ in $\mathrm{supp}(K)$ are connected by a sequence $p = p_0, \ldots, p_n = p'$ such that $p_{j-1}$ and $p_j$ share an edge for all $j \in [n]$.
\end{defn}

\begin{lem}\label{lemma:connected-plaquette-count-bound}
There is a constant $C = C_d$ depending only on $d$ such that for any $A \geq 1$, and any plaquette $p$, the number of connected plaquette assignments $K$ for which $K(p) \neq 0$ with area $A$ is at most $C^A$. 
\end{lem}
\begin{proof}
Consider the graph $\mathcal{G}$ whose vertices are the oriented plaquettes of $\ZZ^d$, and two vertices $p, p'$ are connected if $p, p'$ share an edge. By standard results (see e.g. \cite[(4.24)]{GrimmettPercolation}), for any $j \geq 1$, the number of connected subgraphs of $\mcl{G}$ of size $j$ containing a given vertex $p$ is at most $C^j$. Now, a connected plaquette assignment $K$ of area $A$ is specified by a connected subgraph $G = (V, E)$ of $\mathcal{G}$ with $|V|  \leq A$, along with a vertex labeling $(k_v, v \in V)$ such that $k_v \geq 1$ for all $v$, and $\sum_{v \in V} k_v = A$. By a standard stars-and-bars argument, for fixed $j \geq 1$, for $|V| = j$, the number of such vertex labelings is at most $\dbinom{A-1}{j-1}$. Using the elementary inequalities
\begin{equation}\label{eq:binomial-bound}
	\dbinom{A-1}{j-1} \leq \bigg(\frac{(A-1) e}{j-1}\bigg)^{j - 1}, ~~\text{and}~~ \sup_{x > 0} (m/x)^x \leq e^{m/e},
\end{equation}
we obtain that the number of connected plaquette assignments $K$ containing $p$ is at most 
\[ 
\sum_{j=1}^A C^j \dbinom{A-1}{j-1} \leq \sum_{j=1}^A C^j e^{A/e} \leq C^A,
\]
where the constant $C$ may change within the line.
\end{proof}

\begin{defn}
Given a loop $\ell$, we say that a plaquette assignment $K : \cP_{\ZZ^d} \ra \NN$ is {\bf $\ell$-connected} if every connected component of its support is connected to at least one edge of $\ell$.
\end{defn}

\begin{rmk}\label{remark:loop-K-connected}
The motivation for the definition of $\ell$-connectedness is the following.
For loops $\ell$, the surface sum $\phi^K(\ell) = 0$ if $K$ is not $\ell$-connected. This is because $\npe(\ell,K)$ only contains maps with a unique connected component.
\end{rmk}

\begin{lem}\label{lemma:ell-connected-plaquette-bound}
Let $\ell$ be a loop. There is a constant $C = C_d$ depending only on $d$ such that for any $A \geq 1$, the number of $\ell$-connected plaquette assignments with area $A$ is at most $C^{|\ell|} C^A$.
\end{lem}
\begin{proof}
Every $\ell$-connected plaquette assignment $K$ decomposes into $K = K_1 + \cdots + K_m$, where the $K_j$ are connected plaquette assignments with disjoint supports which each contain some edge of $\ell$. 
Let $N(\ell, A)$ be the number of connected plaquette assignments with area $A$ which contain some edge of $\ell$. Then the number of $\ell$-connected plaquette assignments with area $A$ is at most
\[ \sum_{m=1}^A \frac{1}{m!}\sum_{\substack{A_1 + \cdots + A_m = A \\ A_j \geq 1 \\ 1 \leq j \leq m}} \prod_{j=1}^m N(\ell, A_j) .\]
Here, the $\frac{1}{m!}$ factor arises because we may permute the labels $K_1, \ldots, K_m$ of the decomposition. By Lemma \ref{lemma:connected-plaquette-count-bound}, we have that $N(\ell, A_j) \leq (2d|\ell|) C^{A_j}$, where the $2d|\ell|$ factor arises because any connected plaquette assignment which contains some edge of $\ell$ must contain one of the plaquettes touching $\ell$, and there are only at most $2d|\ell|$ such plaquettes. Inserting this estimate, we obtain the further upper bound
\[ \sum_{m=1}^A \frac{1}{m!} (2d|\ell|)^m C^A \dbinom{A-1}{m-1},\]
where $\dbinom{A-1}{m-1}$ counts the number of partitions of $A$ into parts $A_1, \ldots, A_m \geq 1$. To finish, apply the estimate \eqref{eq:binomial-bound} to obtain that the above is bounded by
\[ C^A e^{A/e} \sum_{m=1}^\infty \frac{(2d|\ell|)^m}{m!} \leq e^{2d|\ell|} (C e^{1/e})^A, \]
as desired.
\end{proof}

We may now finally prove Lemma \ref{lemma:phi-bound}.

\begin{proof}[Proof of Lemma \ref{lemma:phi-bound}]
Fix a string $s = \{\ell_1, \ldots, \ell_n\}$ with $n$ loops. We may write
\[ \sum_{K : \cP_{\ZZ^d} \rightarrow \NN} |\phi^K(s)| ~= \sum_{A=n}^\infty \sum_{\substack{K : \cP_{\ZZ^d} \rightarrow \NN \\ \area(K) = A}} |\phi^K(s)|.\]
Note that the sum starts at $A = n$, because each of the $n$ loops must be in its own component, and the smallest area a component can have is one. Next, recalling the definition of $\phi^K(s)$ from \eqref{eq:gen-phi-string} (see also \eqref{eq:eiwuvfuowevo}), we may further write the right-hand side as
\[ \sum_{A=n}^\infty \sum_{\substack{K : \cP_{\ZZ^d} \rightarrow \NN \\ \area(K) = A}} |\phi^K(s)| ~= \sum_{A=n}^\infty \sum_{\substack{A_1 + \cdots + A_n = A \\ A_1, \ldots, A_n \geq 1}} \sum_{\substack{K_j : \cP_{\ZZ^d} \rightarrow \NN \\ \area(K_j) = A_j\\ j = 1, \ldots, n }} |\phi^{K_1}(\ell_1)| \cdots |\phi^{K_n}(\ell_n)|.  \]
Now by Remark \ref{remark:loop-K-connected}, Lemma \ref{lemma:ell-connected-plaquette-bound} and Proposition \ref{prop:-technical-bound-on-phi-K}, the right-hand side above may be bounded by (using \eqref{eq:binomial-bound} in going from the first line to the second line, and also allowing the constant $C$ to change from line to line)
\begin{align*}
	\sum_{A=n}^\infty \sum_{\substack{A_1 + \cdots + A_n = A \\ A_1, \ldots, A_n \geq 1}}  \prod_{j=1}^n C^{|\ell_j|} (C_d \beta)^{A_j}  &\leq C^{|s|} \sum_{A=n}^\infty \dbinom{A-1}{n-1}  (C_d\beta)^A \\
	&\leq C^{|s|} \sum_{A=n}^\infty (C_d\beta)^A e^{A/e} < \infty,  
\end{align*}
as long as $\beta$ is small enough.
\end{proof}

\subsection{A pinching-peeling-separating exploration process for planar embedded maps}\label{subsubsec: Splittings and Deformations on embedded maps}

In order to derive the master loop equation later in Section~\ref{subsec: surface sum Satisfies Master Loop Equation}, we introduce here a ``peeling exploration'' -- called the pinching-peeling-separating process -- designed for non-separable planar embedded maps. This process systematically explores these maps face by face and further performs certain specific operations on the embedded maps, similar to splittings and deformations of loops. That is, when we  pinch-peel-separate an embedded map with boundary $\ell$, we obtain an embedded map with a new boundary $\ell'$, where $\ell'$ is a deformation or splitting of $\ell$. Our analysis will also account for how the weights of the original embedded map and the new embedded map are modified.

\medskip

Throughout this section, we assume that:
\begin{itemize}
\item $\ell$ is a fixed loop of the form $\ell = \mathbf{e} \, \pi$, where $\mathbf{e}$ is a copy of the oriented edge $e\in E_{\ZZ^d}$,
\item $K:\cP_{\ZZ^d}\to \NN$ is a fixed non-zero plaquette assignment such that the pair $(\ell,K)$ is balanced.
\item $\npe(\ell,K)\neq \emptyset$ and $M=(m,\psi) \in \npe(\ell,K)$ is a non-separable planar embedded map with boundary $\ell$ and plaquette assignment $K$ (recall Definition~\ref{defn:non-separableplanar-embedded}).
\end{itemize} 

\begin{rmk}\label{rmk:set-can-be-empty}
We explain why we assumed that $\npe(\ell,K)\neq \emptyset$. One subtle but important fact is that $\npe(\ell,K)$ can be empty even if $(\ell,K)$ is balanced, $\ell\neq \emptyset$ and $K\neq 0$. For instance, it can be checked that for two adjacent positively oriented plaquettes $p$ and $q$, when $\ell=p$, and $K$ is such that $K(p^{-1})=1$, $K(q)=1$, $K(q^{-1})=1$, and $K$ is zero for all other plaquettes, then $\npe(\ell,K)=\emptyset$ because all planar maps in $\cP\cM(\ell,K)$ turn out to be separable. This fact will play a role later in Section~\ref{subsec: surface sum Satisfies Master Loop Equation} (see in particular Section~\ref{sec: empty surface sum satisfies the master loop equation}).
\end{rmk}

We are now ready to introduce our \textbf{pinching-peeling-separating (PPS) process} for non-separable planar embedded maps.

\medskip

\noindent\underline{\textbf{Step 1:}} \emph{\textbf{Exploring the blue face incident to 
	$\mathbf{e}$}}. The first step of the PPS process prescribes how to explore the blue face $B_{\mathbf{e}}(M)$ incident to $\mathbf{e}$. Let $\delta_{\mathbf{e}}(M)$ denote half of the degree of $B_{\mathbf{e}}(M)$. There are two different possible cases: 
	
	\begin{enumerate}
\item either $\delta_{\mathbf{e}}(M)=1$, i.e.\ we are exploring a blue 2-gon. See the left-hand side of Figure~\ref{fig-PPS-process-1} for an example;
\item or $\delta_{\mathbf{e}}(M)\in [2,n_e(\ell,K)]$, i.e.\ we are exploring a large blue face (we recall that the definition of $n_e(\ell,K)$ was given in \eqref{defn:ne}). See the left-hand side of Figure~\ref{fig-PPS-process-2} for an example.
\end{enumerate}
We first describe the process in Case 1. Note that if  $\delta_{\mathbf{e}}(M)=1$, then $\mathbf{e}$ is in a $2$-gon with another edge $\mathbf{e}^{-1}$ of $M$ which is a copy of the lattice edge $e^{-1}\in E_{\ZZ^d}$. Now there are two possible sub-cases: 
\begin{enumerate}
\item[1a.] $\mathbf{e}^{-1}$ is incident to a plaquette face $\mathbf{p}$ with edges $\mathbf{e}^{-1}\,\mathbf{e}_a\,\mathbf{e}_b\,\mathbf{e}_c$ in clockwise order, where $\mathbf{e}_x$ is a copy of the oriented lattice edge $e_x\in E_{\ZZ^d}$ for $x=a,b,c$ and $p=e^{-1}e_ae_be_c$ is the corresponding lattice plaquette.

\item[1b.] $\mathbf{e}^{-1}$ is part of the boundary of $M$, i.e.\ $\ell = \mathbf{e} \,\ell_1\, \mathbf{e}^{-1}\,\ell_2$ where $\ell_1\,\mathbf{e}^{-1}\,\ell_2 = \pi$. Note that $\ell_1$ and $\ell_2$ are non-empty loops as we assumed that $\ell$ has no backtracks and moreover $\{\ell_1,\ell_2\}\in\SS_-(\mathbf{e},\ell)$.
\end{enumerate}

Next we describe the PPS process for the above two sub-cases:

\medskip

\noindent\textbf{\underline{Step 2 (Case 1a):}} \textbf{\emph{Peeling \& combining (negative deformation)}}. The PPS process peels from $M$ the blue 2-gon $B_{\mathbf{e}}(M)$ and combines the plaquette face $\mathbf{p}$ with the yellow external face of $M$, obtaining a new embedded map $M'$ with boundary $\mathbf{e}_a\,\mathbf{e}_b\,\mathbf{e}_c\, \pi$ corresponding to the new lattice loop $\ell'=e_a\,e_b\,e_c\, \pi = \ell \ominus_{\mathbf{e}} p$. Notice since $M$ is a non-separable planar embedded map so is $M'$.  Hence
\begin{equation*}
M'\in\npe(\ell \ominus_{\mathbf{e}} p , K\sm \{p\}) \quad\text{ and }\quad p\in\cP_{\ZZ^d}(e^{-1},K).
\end{equation*}
We call the operation that associates $M$ with the map $M'$ the \textbf{negative deformation operation on embedded maps}. We finally note that 
\begin{equation}\label{eq: neg def total weight relation}
\upbeta^{\area(M)}w_{\infty}(M) = \upbeta\cdot\upbeta^{\area(M')}w_{\infty}(M').
\end{equation}
The top of Figure~\ref{fig-PPS-process-1} shows an example of this step of the PPS process. We conclude with an observation useful for future reference.

\begin{figure}[ht!]
\begin{center}
	\includegraphics[width=.99\textwidth]{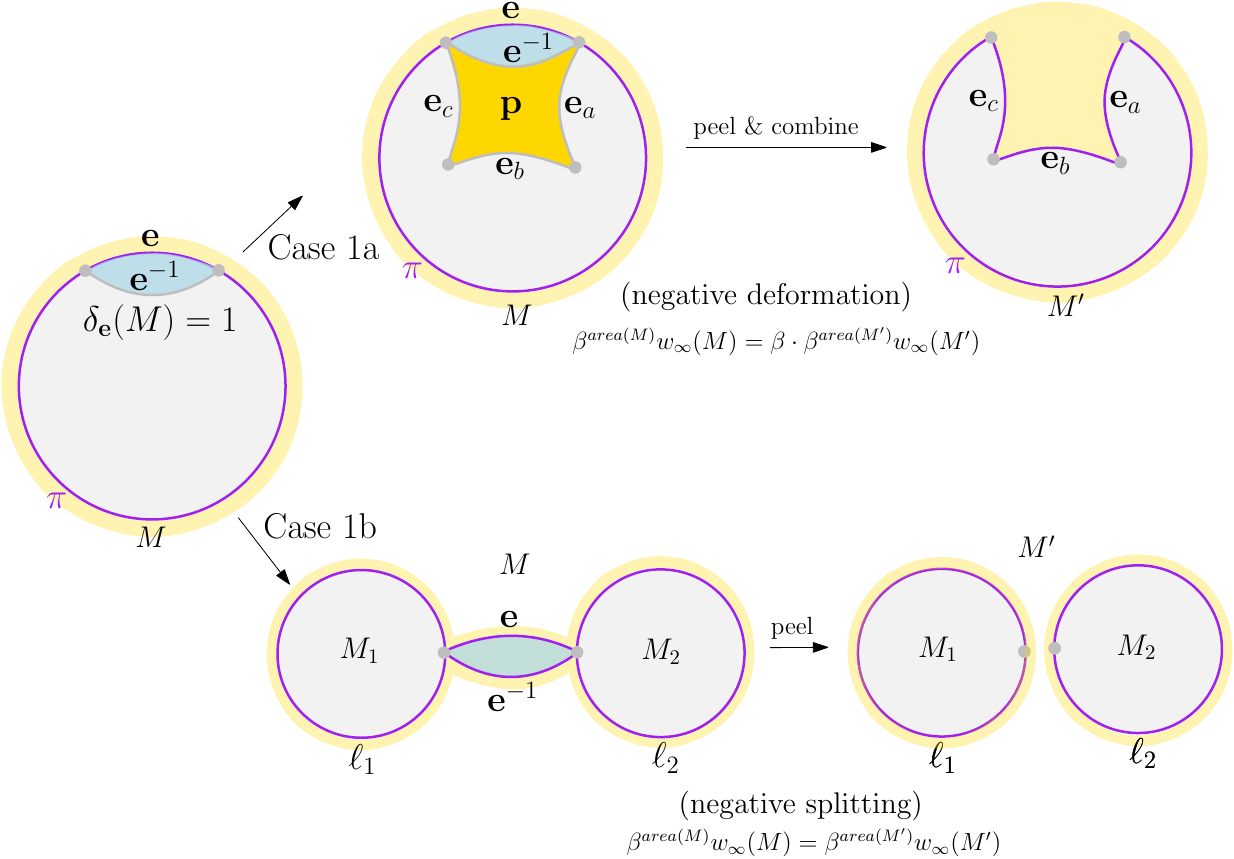}  
	\caption{\label{fig-PPS-process-1} A schema for the Step 2 (Cases 1a and 1b) of the PPS process.}
\end{center}
\vspace{-3ex}
\end{figure}

\begin{obs}\label{obs:almost-bij}
The negative deformation operation on embedded maps is an injective map from the set 
\begin{equation}\label{eq:crazyset2}
	\Big\{M\in\npe(\ell,K)\,:\,\delta_{\mathbf{e}}(M)=1,\,\mathbf{e}^{-1}\notin \ell \Big\} 
\end{equation}
to the set $\bigsqcup_{p\in\cP_{\ZZ^d}(e^{-1},K)}\npe(\ell \ominus_{\mathbf{e}} p , K\sm \{p\})$ but is not surjective, as we are going to explain. 

Recall that $\ell$ has the form $\mathbf{e} \, \pi$ and $\mathbf{e}$ is a copy of the lattice edge $e$.  
Fix $p\in\cP_{\ZZ^d}(e^{-1},K)$. All the embedded maps in $\npe(\ell \ominus_{\mathbf{e}} p , K\sm \{p\})$ that cannot be obtained from a negative deformation operation of a map in the set \eqref{eq:crazyset2} are the ones having a connected\footnote{Two faces are connected if they share at least one vertex.} sequence of  blue faces sent to the lattice edge $e$ and connecting the starting and final vertex of $\pi$. We denote this set by $\npe^{\text{bad}}( \ell\ominus_{\mathbf{e}} p, K \sm\{p\})$. See Figure~\ref{fig-PPS-process-3} for an example.
\end{obs}

\begin{figure}[ht]
\begin{center}
	\includegraphics[width=.89\textwidth]{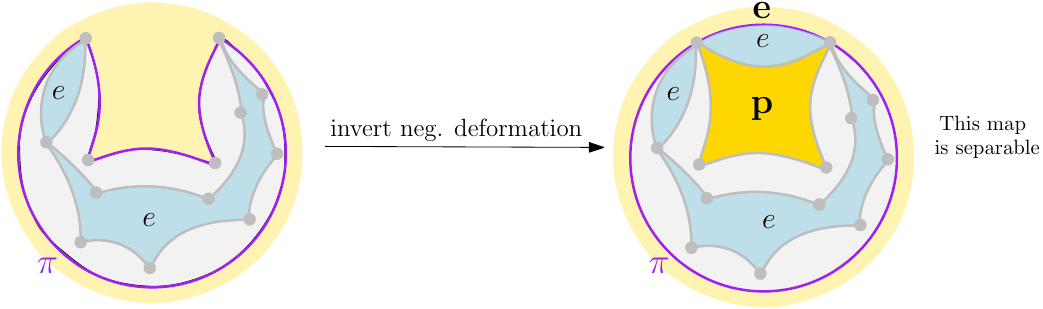}  
	\caption{\label{fig-PPS-process-3} On the left we have an example of an embedded map in $\npe^{\text{bad}}( \ell\ominus_{\mathbf{e}} p, K \sm\{p\})$. Indeed, there is  a connected sequence of blue faces sent to the lattice edge $e$ that connect the starting and final vertex of $\pi$. Note that if we invert the negative deformation operation, we get on the right an embedded map which contains an enclosure loop and so it is separable.}
\end{center}
\vspace{-3ex}
\end{figure}

\medskip

\noindent\textbf{\underline{Step 2 (Case 1b):}} \textbf{\emph{Peeling (negative splitting)}}. The PPS process peels $B_{\mathbf{e}}(M)$ from $M$ (i.e.\ removes the blue 2-gon $B_{\mathbf{e}}(M)$). Doing this, we obtain two new embedded maps $M_1$ (with boundary $\ell_1$) and $M_2$ (with boundary $\ell_2$) with plaquette assignments $K_1$ and $K_2$ such that 
\begin{itemize}
\item $\{\ell_1,\ell_2\}\in\SS_-(\mathbf{e},\ell)$;
\item $K_1(p)+K_2(p) = K(p)$, for all $ p\in \cP_{\ZZ^d}$;
\item $M_1$ and $M_2$ are both non-separable planar embedded maps (since $M$ is a non-separable planar embedded map);
\item the pairs $(\ell_1,K_1)$ and $(\ell_2,K_2)$ are balanced.
\end{itemize}
The last item is a simple consequence of the fact that each blue face sent to a specific lattice edge contains on its boundary the same number of copies of the two possible orientations of that edge, together with the fact that every edge of the map $M$ is contained in exactly one blue face. 

As a consequence of the four items above, $M_1\in \npe(\ell_1,K_1)$ and $M_2\in \npe(\ell_2,K_2)$. Thus, setting $M'=(M_1,M_2)$ and $s=\{\ell_1,\ell_2\}$, we conclude that (recall the form of $\npe(s,K)$ from \eqref{eq:string-set-maps})
\begin{equation*}
M'\in\npe(s,K)\quad\text{ and }\quad s\in \SS_-(\mathbf{e},\ell).
\end{equation*}
We call the operation that associates $M$ with the map $M'=(M_1,M_2)$ the \textbf{negative splitting operation on embedded maps}.
We finally note that \begin{equation}\label{eq: neg split total weight relation}
\upbeta^{\area(M)}w_{\infty}(M)
=
\upbeta^{\area(M')}w_{\infty}(M').
\end{equation}
The bottom part of Figure~\ref{fig-PPS-process-1} shows an example of this step of the PPS process. We conclude with an observation useful for future reference.

\begin{obs}\label{obs:it-is-a-bij}
The negative splitting operation on embedded maps is a bijection from the set 
$$\Big\{M\in\npe(\ell,K)\,:\,\delta_{\mathbf{e}}(M)=1,\,\mathbf{e}^{-1}\in \ell \Big\}$$ 
to the set $\bigsqcup_{s\in \SS_-(\mathbf{e},\ell)}\npe(s,K)$.
\end{obs}

\medskip

We now move to the description of the PPS process in Case 2, i.e.\ when $\delta_{\mathbf{e}}(M)\geq 2$. Recall that this means that $B_{\mathbf{e}}(M)$ is a 2$\delta_{\mathbf{e}}(M)$-gon.

\medskip

\noindent\textbf{\underline{Step 2 (Case 2):}} \textbf{\emph{Pinching}}. The PPS process pinches the starting vertex of $\mathbf{e}$ with another vertex of $B_{\mathbf{e}}(M)$ in the same partite class. Note there are $\delta_{\mathbf{e}}(M)-1$ possible different vertices to pinch with. Let $M_i$ denote the new embedded map obtained after one of these pinching operations, where $i$ is equal to half of the degree of the new blue face containing $\mathbf{e}$. An example is given on the left-hand side of Figure~\ref{fig-PPS-process-2}.
Let $\mathbf{e}_i$ denote the edge on the boundary of $B_{\mathbf{e}}(M)$ starting at the vertex used in the pinching operation (note that $\mathbf{e}_i$ is embedded into $e\in E_{\ZZ^d}$, the same (oriented) edge of the lattice that $\mathbf{e}$ is embedded into).
Now there are two possible sub-cases: 
\begin{enumerate}
\item[2a.] $\mathbf{e}_i$ is incident to a plaquette face $\mathbf{q}$ with edges $\mathbf{e}_i\,\mathbf{e}_d\,\mathbf{e}_f\,\mathbf{e}_g$ in clockwise order, where $\mathbf{e}_x$ is a copy of the oriented lattice edge $e_x\in E_{\ZZ^d}$ for $x=d,f,g$ and $q=e \, e_d  \, e_f \, e_g$ is the corresponding lattice plaquette.

\item[2b.] $\mathbf{e}_i$ is part of the boundary of $M_i$, i.e.\ $\ell = \mathbf{e} \,\ell_{i,1}\, \mathbf{e}_i\, \ell_{i,2}$ where $\ell_{i,1}\,\mathbf{e}_i\,\ell_{i,2} = \pi$. Note that $\ell_{i,1}$ and $\ell_{i,2}$ are non-empty loops and $(e\,\ell_{i,1},e\,\ell_{i,2})\in\SS_
+(\mathbf{e},\ell)$.
\end{enumerate}

Next we describe the PPS process for the above two sub-cases:

\medskip

\noindent\textbf{\underline{Step 3 (Case 2a):}} \textbf{\emph{Separating and combining (positive deformation)}}. The PPS process separates the vertex shared by $\mathbf{e}$ and $\mathbf{e}_i$ (duplicating this vertex) by opening the boundary of the plaquette face $\mathbf{q}$ and then combines the interior of the plaquette face $\mathbf{q}$ with the interior of the yellow external face of $M_i$. Doing this, we obtain a new embedded map $M'_i$ with boundary $\mathbf{e}\,\mathbf{e}_d\,\mathbf{e}_f\,\mathbf{e}_g \,\mathbf{e}_i\, \pi$ corresponding to the new lattice loop $\ell'=e \, e_d \, e_f \, e_g \, e \, \pi = \ell \oplus_{\mathbf{e}} q$. 
Notice since $M$ is a non-separable planar embedded map, so is $M'_i$.\footnote{We stress that $M_i$ might be separable. Indeed, in the pinching operation one might create an enclosure loop (if the two vertices used in the pinching operation were connected by a sequence of blue faces sent to the same lattice edge). Nevertheless, even if this enclosure loop was created, it would be removed during the separating operation. Hence $M'_i$ is always non-separable.}  Hence
\begin{equation*}
M'_i\in\npe(\ell \oplus_{\mathbf{e}} q , K\sm \{q\}) \quad\text{ and }\quad q\in\cP_{\ZZ^d}(e,K).
\end{equation*}
We call the operation that associates $M$ with the map $M'_i$ the \textbf{positive deformation operation on embedded maps}. We finally note that 
\begin{equation}\label{eq: neg def total weight relation2}
\upbeta^{\area(M_i)}w_{\infty}(M_i) = \upbeta\cdot\upbeta^{\area(M'_i)}w_{\infty}(M'_i).
\end{equation}
The top part of Figure~\ref{fig-PPS-process-2} shows an example of this step of the PPS process. We conclude with an observation useful for future reference.

\begin{figure}[ht!]
\begin{center}
	\includegraphics[width=.99\textwidth]{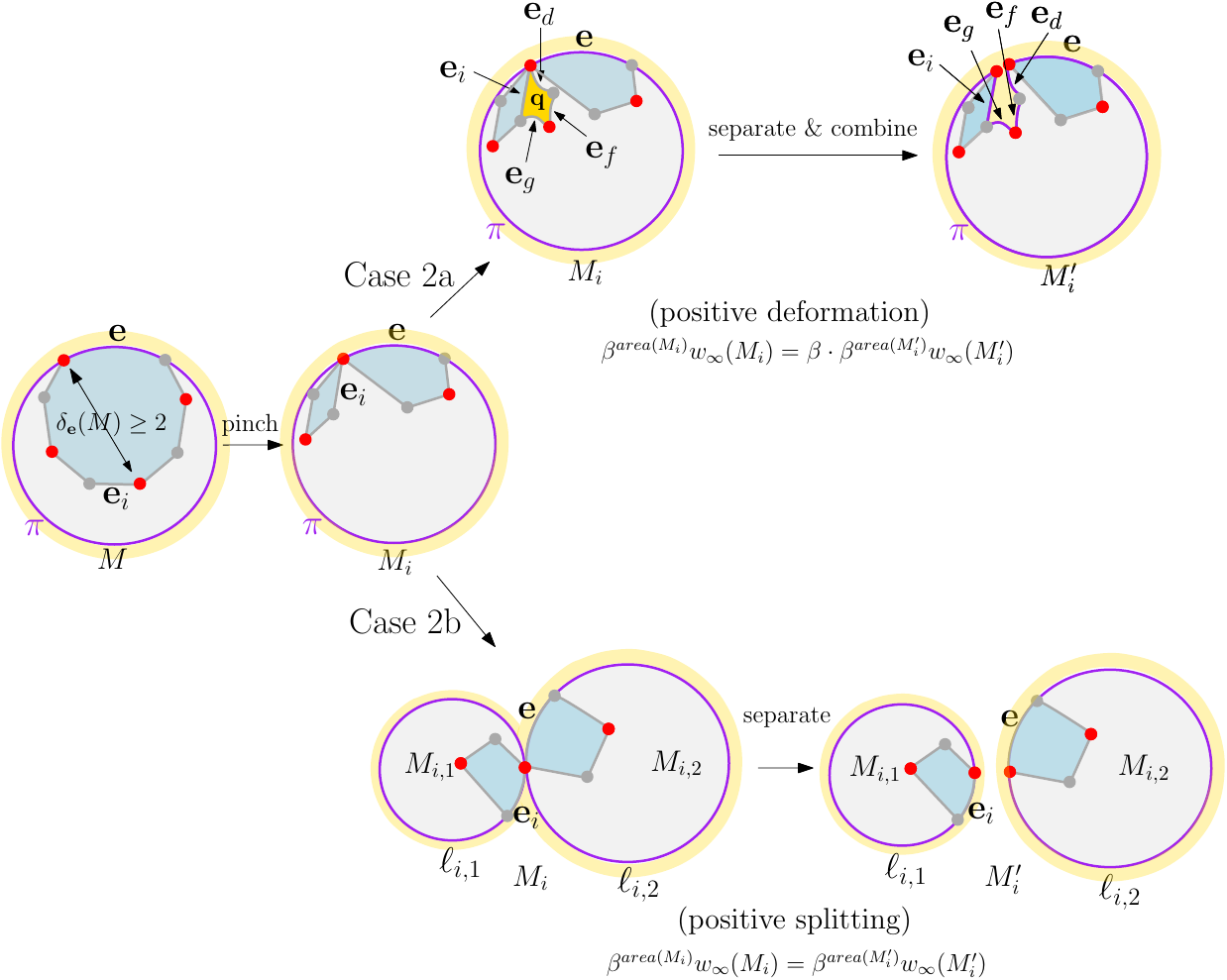}  
	\caption{\label{fig-PPS-process-2} A schema for the Steps 2 and 3 (Cases 2a and 2b) of the PPS process.}
\end{center}
\vspace{-3ex}
\end{figure}

\begin{obs}\label{obs:almot-bij2}
The composition of the pinching with the separate and combine operation, that is, the positive deformation operation on embedded maps, is an injective map from the set 
\begin{equation}\label{eq:crazyset}
	\Big\{M\in\npe(\ell,K)\,:\,\delta_{\mathbf{e}}(M)\in[2,n_e(\ell,K)],\, \mathbf{e}_i\notin\ell \text{ for some } i\in[\delta_{\mathbf{e}}(M)-1]\Big\} 
\end{equation}
to the set $\bigsqcup_{q\in\cP_{\ZZ^d}(e,K)}\npe(\ell \oplus_{\mathbf{e}} q , K\sm \{q\}),$ but is not surjective, as we are going to explain. 

Fix $q=\mathbf{e}' \, \mathbf{e}_d \, \mathbf{e}_f \, \mathbf{e}_g\in\cP_{\ZZ^d}(e,K)$ so that $\ell \oplus_{\mathbf{e}} q=\mathbf{e}' \, \pi \, \mathbf{e} \, \mathbf{e}_d \, \mathbf{e}_f \, \mathbf{e}_g$. All the embedded maps in $\npe(\ell \oplus_{\mathbf{e}} q , K\sm \{q\})$ that cannot be obtained from a positive deformation operation of a map in the set \eqref{eq:crazyset} are the ones having a connected sequence of blue faces sent to the lattice edge $e$ connecting the edges $\mathbf{e}$ and $\mathbf{e}'$. We denote this set by $\npe^{\text{bad}}(\ell\oplus_{\mathbf{e}} q, K \sm\{q\})$. See Figure~\ref{fig-PPS-process-4} for an example.
\end{obs}

\begin{figure}[ht!]
\begin{center}
	\includegraphics[width=.99\textwidth]{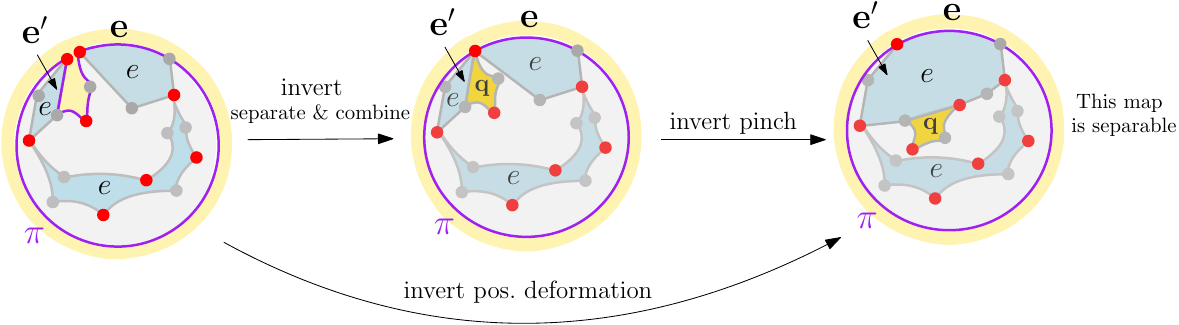}  
	\caption{\label{fig-PPS-process-4} On the left we have an example of an embedded map in $\npe^{\text{bad}}(\ell\oplus_{\mathbf{e}} q, K \sm\{q\})$ for $q=\mathbf{e}' \, \mathbf{e}_d \, \mathbf{e}_f \, \mathbf{e}_g\in\cP_{\ZZ^d}(e,K)$. Indeed, there is  a connected sequence of blue faces sent to the lattice edge $e$ connecting the edges $\mathbf{e}$ and $\mathbf{e}'$. Note that if we invert the positive deformation operation, we get on the right an embedded map which contains an enclosure loop and so it is separable.}
\end{center}
\vspace{-3ex}
\end{figure}

\medskip

\noindent\textbf{\underline{Step 3 (Case 2b):}} \textbf{\emph{Separating (positive splitting)}}. 
The PPS process simply separates the vertex shared by $\mathbf{e}$  and $\mathbf{e}_i$ (duplicating this vertex) splitting $M_i$. Doing this, we obtain two new embedded maps $M_{i,1}$ (with boundary $\mathbf{e}_i\,\ell_{i,1}$ corresponding to the lattice loop $e\,\ell_{i,1}$) and $M_{i,2}$ (with boundary $\mathbf{e}\,\ell_{i,2}$ corresponding to the lattice loop $e\,\ell_{i,2}$) with plaquette assignments $K_{i,1}$ and $K_{i,2}$ such that 
\begin{itemize}
\item $(e\,\ell_{i,1},e\,\ell_{i,2})\in\SS_+(\mathbf{e},\ell)$;
\item $K_{i,1}(p)+K_{i,2}(p) = K(p)$ for all $p\in \cP_{\ZZ^d}$;
\item $M_{i,1}$ and $M_{i,2}$ are both non-separable planar embedded maps (since $M$ is a non-separable planar embedded map);
\item the pairs $(\ell_{i,1},K_{i,1})$ and $(\ell_{i,2},K_{i,2})$ are balanced.
\end{itemize}

As a consequence of the four items above,
\begin{equation*}
M_{i,1}\in \npe(e\,\ell_{i,1},K_{i,1})\quad\text{ and }\quad M_{i,2}\in \npe(e\,\ell_{i,2},K_{i,2}),
\end{equation*} 
and setting $M'_i=(M'_{i,1},M'_{i,2})$ and $s=\{\ell_{i,1},\ell_{i,2}\}$, we conclude that (recall the definition of $\npe(s,K)$ from \eqref{eq:string-set-maps})
\begin{equation*}
M'_i\in\npe(s,K)\quad\text{ and }\quad s\in \SS_+(\mathbf{e},\ell).
\end{equation*}

We call the operation that associates $M$ with the map $M'_i=(M_{i,1},M_{i,2})$ the \textbf{positive splitting operation on embedded maps}.
We finally note that \begin{equation}\label{eq: neg split total weight relation2}
\upbeta^{\area(M_i)}w_{\infty}(M_i)
=
\upbeta^{\area(M'_i)}w_{\infty}(M'_i).
\end{equation}
The bottom part of Figure~\ref{fig-PPS-process-2} shows an example of this step of the PPS process. We conclude with an observation useful for future reference.

\begin{obs}\label{obs:it-is-a-bij2}
The composition of the pinching operation with the separate operation, that is, the positive splitting operation on embedded maps, is a bijection from the set
$$\Big\{M\in\npe(\ell,K)\,:\,\delta_{\mathbf{e}}(M)\in[2,n_e(\ell,K)],\, \mathbf{e}_i\in\ell \text{ for some } i\in[\delta_{\mathbf{e}}(M)-1]\Big\}$$ 
to the set $\bigsqcup_{s\in \SS_+(\mathbf{e},\ell)}\npe(s,K)$.
\end{obs}

\subsection{The surface sums satisfy the master loop equation}\label{subsec: surface sum Satisfies Master Loop Equation}

We now show that $\phi(\ell)$ satisfies the master loop equation from Theorem~\ref{thm: fixed K 't Hooft master loop equation for surface sum} and then we will easily deduce Corollary~\ref{cor: 't Hooft master loop equation surface sum}.
The results in this section assume the validity of the Master surface cancellation lemma~\ref{lemma:master-cancellation} whose proof is postponed to Section~\ref{sec: cancellation lemma}.

\medskip

Our main tool for showing the master loop equation will be the PPS process introduced above. Throughout this section, we fix a loop $\ell$ of the form $\ell = \mathbf{e} \, \pi$, where $\mathbf{e}$ is a copy of to the oriented edge $e\in E_{\ZZ^d}$. First, in Section~\ref{sec: Non-empty surface sum satisfies the master loop equation} we will show the master loop equation holds when $\npe(\ell,K)\neq \emptyset$. Then, in Section~\ref{sec: empty surface sum satisfies the master loop equation} we will show that it still holds when $\npe(\ell,K) = \emptyset$ (recall Remark~\ref{rmk:set-can-be-empty}). We conclude with the proof of Corollary~\ref{cor: 't Hooft master loop equation surface sum} in Section~\ref{sec: proof of cor old lit master loop equation }.

\subsubsection{Non-empty surface sums satisfy the master loop equation}\label{sec: Non-empty surface sum satisfies the master loop equation}

Fix a plaquette assignment $K:\cP_{\ZZ^d}\to \NN$ such that $\npe(\ell,K)\neq \emptyset$. 

Using the PPS process, we rewrite $\phi^K(\ell)$ in a more convenient way to see certain surface cancellations. Recall the definition of $n_e(\ell,K)$ from \eqref{defn:ne} and $\delta_{\mathbf{e}}(M)$ from Section~\ref{subsubsec: Splittings and Deformations on embedded maps}.  Then, by the definition of $\phi^K(\ell)$ from the statement of Theorem~\ref{thm: surface sum representation in 't hooft limit}, and recalling the two cases in Step 1 of the PPS process, we get that
\begin{align*}
\phi^K(\ell) &= \sum_{M\in \npe(\ell,K)}\upbeta^{\area(M)}w_{\infty}(M) \\
&=\sum_{\substack{M\in \npe(\ell,K):\\\delta_{\mathbf{e}}(M) = 1}}\upbeta^{\area(M)}w_{\infty}(M) + \sum_{\delta=2}^{n_e(\ell,K)}\sum_{\substack{M\in \npe(\ell,K):\\\delta_{\mathbf{e}}(M) = \delta}}\upbeta^{\area(M)}w_{\infty}(M) \\&= \phi^K_{\mathbf{e}}(\ell,1) + \sum_{\delta=2}^{n_e(\ell,K)}\phi^K_{\mathbf{e}}(\ell,\delta),
\end{align*}
where
\begin{equation*}
\phi^K_{\mathbf{e}}(\ell,\delta):=\sum_{\substack{M\in \npe(\ell,K):\\\delta_{\mathbf{e}}(M) = \delta}}\upbeta^{\area(M)}w_{\infty}(M).
\end{equation*}

We first focus on the term $\phi^K_{\mathbf{e}}(\ell,1)$. Following Step 2 when $\delta_{\mathbf{e}}(M)=1$ (Cases 1a and 1b) of the PPS process and applying (\ref{eq: neg def total weight relation}) and (\ref{eq: neg split total weight relation}) we get that \begin{align}
\phi^K_{\mathbf{e}}(\ell,1) &=  \sum_{\substack{M\in \npe(\ell,K):\\\delta_{\mathbf{e}}(M) = 1, \mathbf{e}^{-1}\notin \ell}}\upbeta^{\area(M)}w_{\infty}(M)+ \sum_{\substack{M\in \npe(\ell,K):\\\delta_{\mathbf{e}}(M) = 1, \mathbf{e}^{-1}\in \ell}}\upbeta^{\area(M)}w_{\infty}(M)\nonumber \\&= \upbeta\sum_{\substack{M\in \npe(\ell,K):\\\delta_{\mathbf{e}}(M) = 1, \mathbf{e}^{-1}\notin \ell}}\upbeta^{\area(M')}w_{\infty}(M')\label{eq: neg def sum}\\&\hspace{0.5cm}+ \sum_{\substack{M\in \npe(\ell,K):\\\delta_{\mathbf{e}}(M) = 1, \mathbf{e}^{-1}\in \ell}}\upbeta^{\area(M')}w_{\infty}(M')\label{eq: neg split sum},
\end{align}
where we recall that $\mathbf{e}^{-1}$ is the other edge of the blue 2-agon $B_{\mathbf{e}}(M)$ containing $\mathbf{e}$.

We now focus on the term $\phi^K_{\mathbf{e}}(\ell,\delta)$ for $\delta\in[2,n_e(\ell,K)]$. We fix $\delta\in[2,n_e(\ell,K)]$ and an embedded map  $M\in \npe(\ell,K)$ such that $\delta_{\mathbf{e}}(M) = \delta$. 
Preforming the pinching procedure of the PPS process (Step 2 from Case 2) and using Lemma~\ref{lemma: single vertex pinching cancellations} we get that 
\begin{equation}\label{eq: pinching procedure weight relation}
w_{\infty}(M) = -\sum_{i=1}^{\delta-1}w_{\infty}(M_i).
\end{equation}
Now following the separating procedure of the PPS process (Step 3 for Cases 2a and 2b) and using \eqref{eq: neg def total weight relation2} and \eqref{eq: neg split total weight relation2} combined with the last displayed equation, we get that 
\begin{align*}
\upbeta^{\area(M)}w_{\infty}(M)&=-\sum_{i=1}^{\delta-1}\upbeta^{\area(M_i)}w_{\infty}(M_i)\notag\\
&= - \upbeta\sum_{\substack{i\in [ \delta-1]:\\\mathbf{e}_i\notin \ell}}\upbeta^{\area(M_i')}w_{\infty}(M'_i)
- \sum_{\substack{i\in [\delta-1]:\\\mathbf{e}_i\in \ell}}\upbeta^{\area(M_i')}w_{\infty}(M'_i),
\end{align*}
where we used the compact notation $[\delta-1]=\{1,\dots,\delta-1\}$ and we recall that $\mathbf{e}_i$ is the edge in the blue face incident to $\mathbf{e}$ involved in the pinching operation.
Summarizing, we get that for all $\delta\in[2,n_e(\ell,K)]$,
\begin{align*}
\phi^K_{\mathbf{e}}&(\ell,\delta) \\
&= \sum_{\substack{M\in \npe(\ell,K):\\\delta_{\mathbf{e}}(M) = \delta}}\upbeta^{\area(M)}w_{\infty}(M) \nonumber\\&= -\sum_{\substack{M\in \npe(\ell,K):\\\delta_{\mathbf{e}}(M) = \delta}} \Bigg(\upbeta\sum_{\substack{i\in [\delta-1]:\\\mathbf{e}_i\notin \ell}}\upbeta^{\area(M_i')}w_{\infty}(M'_i) + \sum_{\substack{i\in [\delta-1]:\\\mathbf{e}_i\in \ell}}\upbeta^{\area(M_i')}w_{\infty}(M'_i)\Bigg) \nonumber\\&= -\upbeta\sum_{\substack{M\in \npe(\ell,K):\\\delta_{\mathbf{e}}(M) = \delta}}\sum_{\substack{i\in [\delta-1]:\\\mathbf{e}_i\notin \ell}}\upbeta^{\area(M_i')}w_{\infty}(M'_i)- \sum_{\substack{M\in \npe(\ell,K):\\\delta_{\mathbf{e}}(M) = \delta}}\sum_{\substack{i\in [\delta-1]:\\\mathbf{e}_i\in \ell}}\upbeta^{\area(M_i')}w_{\infty}(M'_i).
\end{align*} 
Thus
\begin{align}
\sum_{\delta=2}^{n_e(\ell,K)}\phi^K_{\mathbf{e}}(\ell,\delta) &= -\upbeta \sum_{\delta=2}^{n_e(\ell,K)}\sum_{\substack{M\in \npe(\ell,K):\\\delta_{\mathbf{e}}(M) = \delta}}\sum_{\substack{i\in [\delta-1]:\\\mathbf{e}_i\notin \ell}}\upbeta^{\area(M_i')}w_{\infty}(M'_i) \label{eq: pos def sum}\\&\hspace{0.5cm}-  \sum_{\delta=2}^{n_e(\ell,K)}\sum_{\substack{M\in \npe(\ell,K):\\\delta_{\mathbf{e}}(M) = \delta}}\sum_{\substack{i\in [\delta-1]:\\\mathbf{e}_i\in \ell}}\upbeta^{\area(M_i')}w_{\infty}(M'_i).\label{eq: pos split sum}
\end{align}Notice we have split $\phi^K(\ell)$ into sums of maps with boundary $\ell'$ corresponding to negative deformations \eqref{eq: neg def sum}, positive deformations \eqref{eq: pos def sum}, negative splittings (\ref{eq: neg split sum}), and positive splittings (\ref{eq: pos split sum}). Thus to show the recursion in the statement of Theorem~\ref{thm: fixed K 't Hooft master loop equation for surface sum} (recall the equivalent form in \eqref{eq:new-form}), we need to show the sum of \eqref{eq: neg def sum}, \eqref{eq: neg split sum}, \eqref{eq: pos def sum}, and \eqref{eq: pos split sum} equals the sum of the deformation and splitting terms in the master loop equation. 

First, since we observed that splittings are bijective operations between the sets of embedded maps (described in Observations ~\ref{obs:it-is-a-bij}~and~\ref{obs:it-is-a-bij2}), we get that
\begin{align}
(\ref{eq: neg split sum}) &= \sum_{s\in \SS_{-}(\mathbf{e},\ell)} \sum_{N \in \npe(s , K)}\upbeta^{\area(N)}w_{\infty}(N) = \sum_{s\in \SS_{-}(\mathbf{e},\ell)} \phi^{K}(s)\label{eq: neg split goal}, \\
(\ref{eq: pos split sum}) &= -\sum_{s\in \SS_{+}(\mathbf{e},\ell)} \sum_{N \in \npe(s , K)}\upbeta^{\area(N)}w_{\infty}(N) = - \sum_{s\in \SS_{+}(\mathbf{e},\ell)} \phi^{K}(s) \label{eq: pos split goal}.
\end{align}
Unfortunately, for the deformation terms, equality does not hold\footnote{This can be explicitly checked in some specific examples.} at the level of positive and negative deformations but only at the level of all deformations\footnote{Note that this is possible thanks to some specific surface cancellations detailed in the Matster cancellation lemma~\ref{lemma:master-cancellation}. As mentioned in the introduction, understanding such \emph{cancellations} is a fundamental step in our proof.}. That is, we will show that 
\begin{align}\label{eq: def goal}
\eqref{eq: neg def sum} + \eqref{eq: pos def sum} &= \upbeta\sum_{p\in\cP_{\ZZ^d}(e^{-1},K)} \sum_{N \in \npe(\ell \ominus_{\mathbf{e}} p , K\sm \{p\})}\upbeta^{\area(N)}w_{\infty}(N) \nonumber\\&-\hspace{0.5cm} \upbeta\sum_{q\in \cP_{\ZZ^d}(e,K)} \sum_{N \in \npe(\ell \oplus_{\mathbf{e}} q , K\sm \{q\})}\upbeta^{\area(N)}w_{\infty}(N) \nonumber\\&= \upbeta\sum_{p\in\cP_{\ZZ^d}(e^{-1},K)} \phi^{K\sm\{p\}}(\ell\ominus_{\mathbf{e}}p) - \upbeta\sum_{q\in \cP_{\ZZ^d}(e,K)} \phi^{K\sm\{q\}}(\ell\oplus_{\mathbf{e}}q).
\end{align}
We start by rewriting \eqref{eq: neg def sum}. Using Observation~\ref{obs:almost-bij}, we get that
\begin{align*}
\eqref{eq: neg def sum} =& \upbeta\sum_{\substack{M\in \npe(\ell,K):\\ \delta_{\mathbf{e}}(M) = 1, \mathbf{e}^{-1}\notin \ell}}\upbeta^{\area(M')}w_{\infty}(M')\\
=&  \upbeta
\sum_{p\in\cP_{\ZZ^d}(e^{-1},K)}
\sum_{N\in\npe(\ell \ominus_{\mathbf{e}} p , K\sm \{p\})}\upbeta^{\area(N)}w_{\infty}(N)\\
&-
\upbeta
\sum_{p\in\cP_{\ZZ^d}(e^{-1},K)}
\sum_{N\in\npe^{\text{bad}}(\ell \ominus_{\mathbf{e}} p , K\sm \{p\})}\upbeta^{\area(N)}w_{\infty}(N).
\end{align*}
Similarly, using Observation~\ref{obs:almot-bij2}, we get that
\begin{align*}
\eqref{eq: pos def sum} =& -\upbeta\sum_{\delta=2}^{n_e(\ell,K)}\sum_{\substack{M\in \npe(\ell,K):\\\delta_{\mathbf{e}}(M) = \delta}}\sum_{\substack{i\in [\delta-1]:\\\mathbf{e}_i\notin \ell}}\upbeta^{\area(M_i')}w_{\infty}(M'_i) \\=& -\upbeta\sum_{q\in\cP_{\ZZ^d}(e,K)}\sum_{N\in\npe(\ell \oplus_{\mathbf{e}} q , K\sm \{q\})}\upbeta^{\area(N)}w_{\infty}(N)\\
&+\upbeta\sum_{q\in\cP_{\ZZ^d}(e,K)}\sum_{N\in\npe^{\text{bad}}(\ell \oplus_{\mathbf{e}} q , K\sm \{q\})}\upbeta^{\area(N)}w_{\infty}(N).
\end{align*}
Comparing the last two displayed equations, it is immediate to realize that the proof of \eqref{eq: def goal} is complete (thus giving Theorem~\ref{thm: fixed K 't Hooft master loop equation for surface sum}), by the following surface cancellation result whose proof is postponed to Section~\ref{sec: cancellation lemma}.

\begin{lem}[\textsc{master cancellation lemma}]\label{lemma:master-cancellation}
Consider a loop $\ell=\mathbf{e} \, \pi$, where $\mathbf{e}$ is a copy of the lattice edge $e$, and a non-zero plaquette assignment $K$ such that $(\ell,K)$ is balanced. Then
\begin{align}\label{eq:cancellation-one-1}
	\sum_{p\in\cP_{\ZZ^d}(e^{-1},K)}&\sum_{N\in \npe^{\text{bad}}(\ell\ominus_{\mathbf{e}} p , K\sm\{p\})}\upbeta^{\area(N)}w_{\infty}(N) \notag\\&= \sum_{q\in\cP_{\ZZ^d}(e,K)}\sum_{N\in \npe^{\text{bad}}(\ell\oplus_{\mathbf{e}} q , K\sm\{q\})}\upbeta^{\area(N)}w_{\infty}(N),
\end{align}
where we recall that the sets $\npe^{\text{bad}}(\ell\ominus_{\mathbf{e}} p , K\sm\{p\})$ and $\npe^{\text{bad}}(\ell\oplus_{\mathbf{e}} q , K\sm\{q\})$ are defined in  Observations~\ref{obs:almost-bij}~and~\ref{obs:almot-bij2}, respectively.
\end{lem}

\subsubsection{Empty surface sums satisfy the master loop equation}\label{sec: empty surface sum satisfies the master loop equation}

Recall that the PPS process is only defined when $\npe(\ell,K)\neq \emptyset$. Thus, the approach of the previous section will not work when $\npe(\ell,K) = \emptyset$. However, we will show that the master loop equation from Theorem~\ref{thm: fixed K 't Hooft master loop equation for surface sum} still holds in this setting. The key step will be to consider some inverse operations from the PPS process.

Suppose that $\npe(\ell,K)=\emptyset$. If the four sums on the right-hand side of the master loop equation are empty, then we are done. Thus, suppose it is not the case, that is, there exists some splitting or deformation of $\ell$, call it $s'$, such that $\npe(s',K')\neq \emptyset$. 

We claim that $s'$ cannot be a splitting of $\ell$. Indeed, suppose that $s'$ is a (positive or negative) splitting of $\ell$ and $\npe(s',K)\neq \emptyset$. Consider $M'\in \npe(s',K)$. Then it is simple to see that we can construct a map in $M\in\npe(\ell,K)$ by reversing the PPS process for (positive or negative) splittings. But since we assumed that $\npe(\ell,K)=\emptyset$, we conclude that $s'$ cannot be a splitting of $\ell$.

Therefore, we must have that the right-hand side of the master loop equation is equal to
\begin{align*}
\upbeta \sum_{p\in \cP_{\ZZ^d}(e^{-1},K)}\phi^{K\sm p}(\ell \ominus_{\mathbf{e}}p)
&-\upbeta\sum_{q\in \cP_{\ZZ^d}(e,K)}\phi^{K\sm q}(\ell \oplus_{\mathbf{e}}q)\\
&=\upbeta \sum_{p\in \cP_{\ZZ^d}(e^{-1},K)}
\sum_{N\in \npe(\ell\ominus_{\mathbf{e}} p , K\sm\{p\})}\upbeta^{\area(N)}w_{\infty}(N)\\
&\quad-\upbeta\sum_{q\in \cP_{\ZZ^d}(e,K)}
\sum_{N\in \npe(\ell \oplus_{\mathbf{e}}q , K\sm \{q\})}\upbeta^{\area(N)}w_{\infty}(N).
\end{align*}
Now we can partition the sets $\npe(\ell\ominus_{\mathbf{e}} p , K\sm\{p\})$ and $\npe(\ell \oplus_{\mathbf{e}}q , K\sm \{q\})$ as follows (recall that the sets $\npe^{\text{bad}}(\ell\ominus_{\mathbf{e}} p , K\sm\{p\})$ and $\npe^{\text{bad}}(\ell\oplus_{\mathbf{e}} q , K\sm\{q\})$ are defined in  Observations~\ref{obs:almost-bij}~and~\ref{obs:almot-bij2}):
\begin{align*}
&\npe(\ell\ominus_{\mathbf{e}} p , K\sm\{p\})=\npe^{\text{bad}}(\ell\ominus_{\mathbf{e}} p , K\sm\{p\})\,\,\bigsqcup\,\, \overline{\npe}^{\text{bad}}(\ell\ominus_{\mathbf{e}} p , K\sm\{p\}),\\
&\npe(\ell \oplus_{\mathbf{e}}q , K\sm \{q\})=\npe^{\text{bad}}(\ell\oplus_{\mathbf{e}} q , K\sm\{q\})\,\,\bigsqcup\,\,\overline{\npe}^{\text{bad}}(\ell\oplus_{\mathbf{e}} q , K\sm\{q\}).
\end{align*}

We claim that the sets $\overline{\npe}^{\text{bad}}(\ell\ominus_{\mathbf{e}} p , K\sm\{p\})$ and $\overline{\npe}^{\text{bad}}(\ell\oplus_{\mathbf{e}} q , K\sm\{q\})$ must be empty for all plaquettes $p$ and $q$. Indeed, if either contained some map $M'$ then applying the inverse operation of the PPS process to $M'$ would produce some map $M\in \npe(\ell,K)$, a contradiction with the fact that $\npe(\ell,K)=\emptyset$. Thus, we get that the right-hand side of the master loop equation simplifies to
\begin{align*}
\sum_{p\in\cP_{\ZZ^d}(e^{-1},K)}&\sum_{N\in \npe^{\text{bad}}(\ell\ominus_{\mathbf{e}} p , K\sm\{p\})}\upbeta^{\area(N)}w_{\infty}(N) \notag
\\&- \sum_{q\in\cP_{\ZZ^d}(e,K)}\sum_{N\in \npe^{\text{bad}}(\ell\oplus_{\mathbf{e}} q , K\sm\{q\})}\upbeta^{\area(N)}w_{\infty}(N)
\end{align*}which we know equals $0$ by the Master surface cancellation lemma~\ref{lemma:master-cancellation}. Thus, we get that the master loop equation sill holds when $\npe(\ell,K)=\emptyset$, as desired.

\subsubsection{The master loop equation for fixed location}\label{sec: proof of cor old lit master loop equation }

Finally , we deduce Corollary~\ref{cor: 't Hooft master loop equation surface sum}.

\begin{proof}[Proof of Corollary~\ref{cor: 't Hooft master loop equation surface sum}]
From Theorem~\ref{thm: fixed K 't Hooft master loop equation for surface sum}, for a fixed non-null loop $\ell$ and plaquette assignment $K$, we have that
\begin{align*}
	\phi^K(\ell) =&\sum_{\{\ell_1,\ell_2\}\in \SS_+(\mathbf{e},\ell)}
	\sum_{K_1+K_2=K}
	\phi^{K_1}(\ell_1)\phi^{K_2}(\ell_2) - \sum_{\{\ell_1,\ell_2\}\in \SS_-(\mathbf{e},\ell)}
	\sum_{K_1+K_2=K}\phi^{K_1}(\ell_1)\phi^{K_2}(\ell_2) \\
	&+ \upbeta
	\sum_{p\in \cP_{\ZZ^d}(e^{-1},K)}
	\phi^{K\sm p}(\ell \ominus_{\mathbf{e}}p)
	-\upbeta\sum_{q\in \cP_{\ZZ^d}(e,K)}
	\phi^{K\sm q}(\ell \oplus_{\mathbf{e}}q).
\end{align*}
Summing the last equation over all possible plaquette assignments $K$ and exchanging the order of the sums (this is possible thanks to Lemma~\ref{lemma:phi-bound} since $\upbeta_0(d)\stackrel{\eqref{eq:beta_0-def}}{=}\min\{\upbeta_1(d),\upbeta_2(d)\}\leq \upbeta_2(d)$), we get that
\begin{align*}
	\phi(\ell) =&\sum_{\{\ell_1,\ell_2\}\in \SS_+(\mathbf{e},\ell)}
	\phi(\ell_1)\phi(\ell_2) - \sum_{\{\ell_1,\ell_2\}\in \SS_-(\mathbf{e},\ell)}
	\phi(\ell_1)\phi(\ell_2) \\
	&+ \upbeta
	\sum_{p\in \cP_{\ZZ^d}(e^{-1},K)}
	\phi(\ell \ominus_{\mathbf{e}}p)
	-\upbeta\sum_{q\in \cP_{\ZZ^d}(e,K)}
	\phi(\ell \oplus_{\mathbf{e}}q).
\end{align*}
Finally, recalling the definitions of the sets  $\DD_-(\mathbf{e},s)$ and $\DD_+(\mathbf{e},s)$, we get the master loop equation in the Corollary statement.
\end{proof}

\section{Cancellations for sums of embedded maps obtained from pinchings of a blue face}\label{sect:pinch-canc}

The main result of this section is Theorem~\ref{thm:master-sum-blue-faces} which establishes some fundamental surface cancellations obtained from pinchings of a blue face. This result will be one of the main tools that we will use later in Section~\ref{sec: cancellation lemma} to prove the Master surface cancellation lemma~\ref{lemma:master-cancellation}. 

The rest of this section is organized as follows. In Section~\ref{sect:seq-pinch}, we introduce the new notion of collections of pinchings and state our main result, i.e.\ Theorem~\ref{thm:master-sum-blue-faces}. Then in Section~\ref{sect:all-pinching} we prove a generalization of the Cancellation Lemma~\ref{lemma: single vertex pinching cancellations} and establish the first part of Theorem~\ref{thm:master-sum-blue-faces}. 
Finally, in Section~\ref{subsec: Weight of maps with at least one invalid pinching} we complete the proof of Theorem~\ref{thm:master-sum-blue-faces}.

\subsection{Collections of pinchings and a new surface cancellation result}\label{sect:seq-pinch}

Recall from Section~\ref{sect:pinchings} (see also Figure~\ref{fig-pinching}) that given two vertices $u$ and $v$ of a blue face $B$ of an embedded map $M\in \npe(s,K)$ of the same partite class, we defined the pinching operation $\pp_{u,v}$. 
Given a collection of pinchings $\left\{\pp_{u_1,v_1},\dots,\pp_{u_r,v_r}\right\}$ where each pair $(u_i,v_i)$ is formed by vertices of the same blue face $B$, we denote the operation that sequentially performs them by\footnote{We will shortly clarify (after certain preliminary comments) that the order in which the pinchings are performed is not relevant, this is why we defined $\cpp$ as a set $\{\pp_{u_i,v_i}\}_{i=1}^r$ and not as a sequence.} $\cpp=\{\pp_{u_i,v_i}\}_{i=1}^r$; see the top part of Figure~\ref{fig-pinching2} for an example. We stress that we do not require that all pinchings involved in a collection of pinchings are in the same partite class of vertices (see again the example at the top of Figure~\ref{fig-pinching2}).

Note that some collections of pinchings might not be realizable, as one pinching could preclude another. This occurs when there exists an $i$ such that after pinching the vertices $(u_i,v_i)$, the vertices $u_j$ and $v_j$ end up in two different blue faces.
We resolve this issue in two different ways, depending on whether $(u_i,v_i)$ and $(u_j,v_j)$ are all in the same partite class or not:
\begin{itemize}
\item If $u_i,v_i,u_j,v_j$ are all in the same partite class, then we replace the collection of pinchings $\{\pp_{u_i,v_i},\pp_{u_j,v_j}\}$ with the collection of three pinchings\footnote{Note that $\pp_{u_i,v_i}$ is an arbitrary choice which could be replaced by $\pp_{u_j,v_j}$.}  $\{\pp_{u_i, u_j},\pp_{v_i,v_j},\pp_{u_i,v_i}\}$ which is a realizable collection of pinchings that pinches all the four vertices together (see the second example in Figure~\ref{fig-pinching2});
\item If $(u_i,v_i)$ and $(u_j,v_j)$ are in opposite partite classes, then we say that the collection of pinchings $\{\pp_{u_i,v_i},\pp_{u_j,v_j}\}$ is \textbf{non-feasible} (see the third example in Figure~\ref{fig-pinching2}).
\end{itemize} 
We say that a collection of pinchings $\cpp=\{\pp_{u_i,v_i}\}_{i=1}^r$ is \textbf{feasible} if it does not contain any pair of non-feasible pinchings.
Note that if a collection of pinchings $\cpp=\{\pp_{u_i,v_i}\}_{i=1}^r$ is feasible then the individual pinchings can be performed in any order and always produce the same map. 

Given a feasible collection of pinchings $\cpp$ of the blue face $B$, we often write $\cpp(M)$ to denote the embedded map obtained by performing the collection of pinchings $\cpp$ to the blue face $B$ of the embedded map $M$, and $\cpp(B)$ to denote the family of blue faces obtained by performing the collection of pinchings $\cpp$.

Finally, let $\ACP(M,B)$ denote the set of all feasible collections of pinchings $\cpp$ of the blue face $B$ and $\AM(M,B)$ denote the set of all the embedded maps $\cpp(M)$ obtained from some collection of pinchings $\cpp\in\ACP(M,B)$ (where $\ACP$ stands for ``all pinchings'' and $\AM$ for ``all maps'').

\begin{figure}[ht!]
\begin{center}
	\includegraphics[width=.89\textwidth]{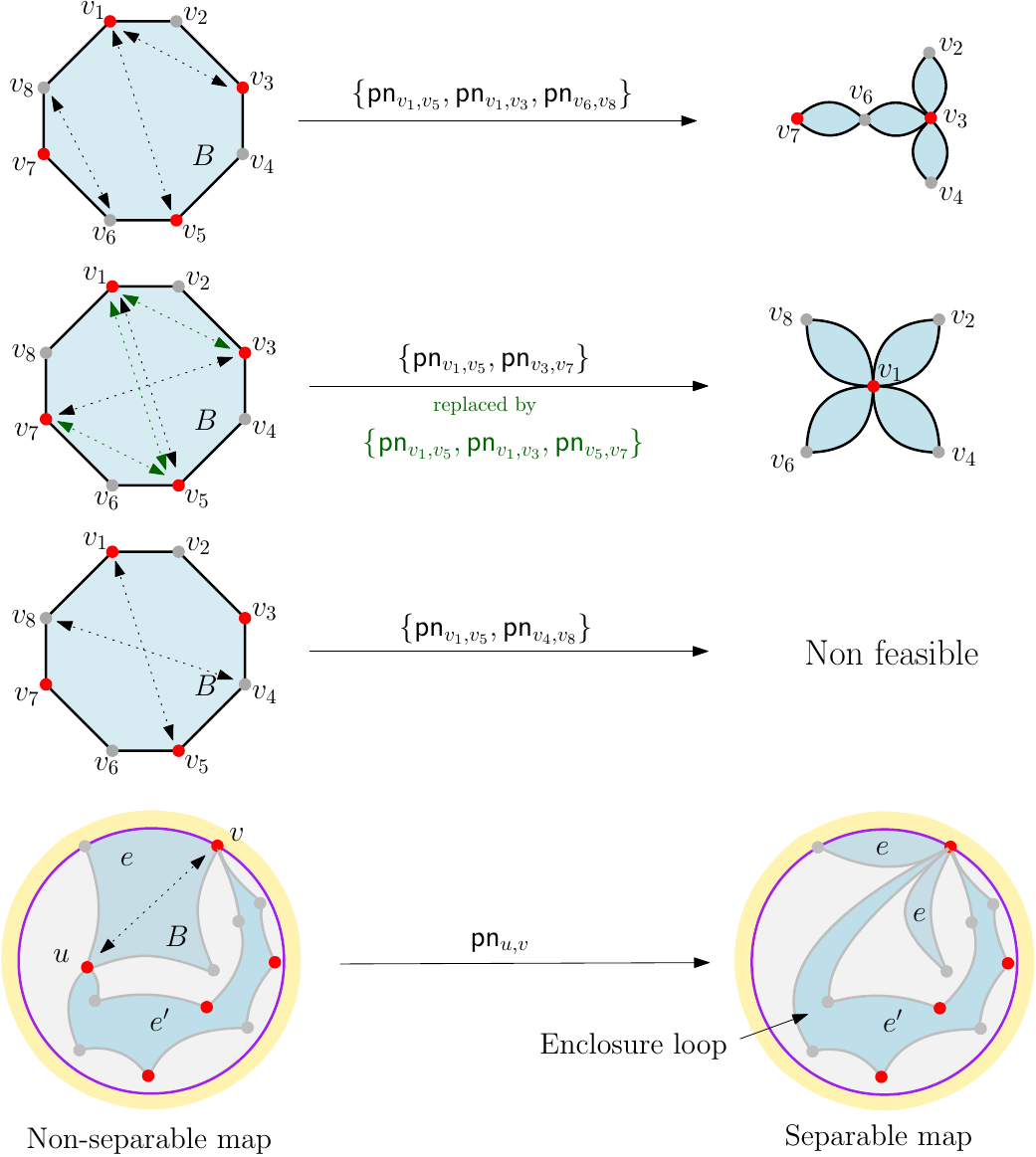}  
	\caption{\label{fig-pinching2} \textbf{Top:} An example of a  feasible collection of pinchings. \textbf{Middle-top:} An example of a feasible collection of pinchings that we explained how to resolve. \textbf{Middle-bottom:} An example of a non-feasible collection of pinchings. \textbf{Bottom:} An example of a non-separable map that becomes separable after a pinching operation.}
\end{center}
\vspace{-3ex}
\end{figure}

We make the following important observation.

\begin{obs}\label{obs:valid-pinching}
Note that a collection of feasible pinchings of a blue face of a non-separable planar embedded map $M\in\npe(s,K)$ might produce a separable planar embedded map; see the bottom part of Figure~\ref{fig-pinching2} for an example.
\end{obs}

Let $\VCP(M,B)$ denote the set of all feasible collections of pinchings of the blue face $B$ which lead to a non-separable embedded map and $\VM(M,B)$ denote the set of all non-separable embedded maps $\cpp(M)$ obtained from some $\cpp\in\VCP(M,B)$ (where $\VCP$ stands for ``valid pinchings''  and $\VM$ for ``valid maps'').

Similarly, let $\ICP(M,B)$ denote the set of all feasible collections of pinchings of the blue face $B$ which lead to a separable embedded map and $\IM(M,B)$ denote the set of all separable embedded maps $\cpp(M)$ obtained from some $\cpp\in\ICP(M,B)$ (where $\ICP$ stands for ``invalid pinchings''  and $\IM$ for ``invalid maps'').

\medskip

Our main result about collections of pinchings is the following fundamental surface cancellation theorem.

\begin{thm}[\textsc{cancellations for sums of embedded maps obtained from pinchings of a blue face}]\label{thm:master-sum-blue-faces}
Let $M\in\npe(s,K)$ be a non-separable planar embedded map and $B$ one of its blue faces. Then
\begin{equation*}
	\sum_{M'\in \VM(M,B)}w_{\infty}(M')=
	\begin{cases}
		\prod_{f\in \BF(M)\sm \{B\}}w_{\deg(f)/2}, \quad&\text{if all the pinchings of $B$ are valid},\\
		0, \quad&\text{if $B$ has at least one invalid pinching},
	\end{cases}
\end{equation*}
where we recall that $\BF(M)$ is the set of blue faces of $M$.
\end{thm}

\begin{rmk}
One natural temptation after seeing the result in Theorem~\ref{thm:master-sum-blue-faces} is to exclude from $\npe(s,K)$ all maps that contain at least one blue face with an invalid pinching, along with all maps obtained through further pinchings of such a face, in the hope that the total sum of the excluded maps would amount to zero. Figure~\ref{fig-tempation-fail} shows an example that clarifies why this is not possible.
\end{rmk}

\begin{figure}[ht!]
\begin{center}
	\includegraphics[width=.89\textwidth]{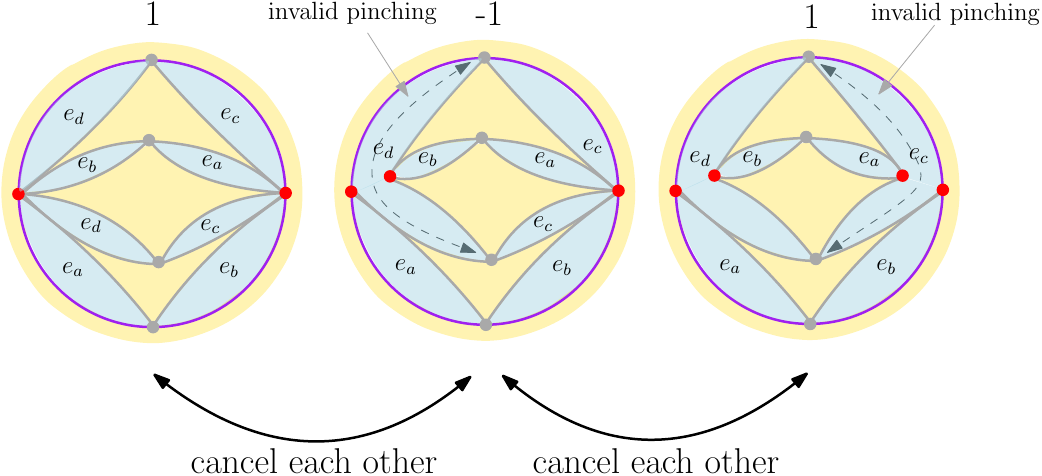}  
	\caption{\label{fig-tempation-fail} Here $s=p$ with $p=e_a\,e_b \, e_c \, e_d$ and $K$ is such that $K(p)=1$, $K(p^{-1})=2$ and $K$ is zero for all other plaquettes. We show three embedded maps in $\npe(s,K)$ with their corresponding weight written on top. \textbf{Left-middle:} The first map is obtained from the second map by pinching the two red vertices of the blue face sent to $e_d$. Note that the other pinching (between the two gray vertices of the blue face sent to $e_d$) in the second map is an invalid pinching. As established in Theorem~\ref{thm:master-sum-blue-faces}, the sum of the weights of the first and second map is indeed zero. \textbf{Middle-right:} The second map is obtained from the third map by pinching the two red vertices of the blue face sent to $e_c$. Note that the other pinching (between the two gray vertices of the blue face sent to $e_c$) in the third map is an invalid pinching. As established in Theorem~\ref{thm:master-sum-blue-faces}, the sum of the weights of the second and third map are zero. \textbf{Left-middle-right:} Note that the sum of the weights of the three maps is not zero, this is why we cannot exclude from $\npe(s,K)$ all maps that contain at least one blue face with an invalid pinching, along with all maps obtained through further pinchings of such a face.}
\end{center}
\vspace{-3ex}
\end{figure}

\subsection{The sum of all feasible pinchings of a blue face equals one}\label{sect:all-pinching}

The next result is a generalization of Lemma~\ref{lemma: single vertex pinching cancellations} and will immediately give us as a corollary the first case of   Theorem~\ref{thm:master-sum-blue-faces}.

Given a blue face $B$ with disjoint vertices, we denote by $\AF(B)$ the set of all families of blue faces that can be obtained from $B$ by applying collections of feasible pinchings to $B$.

\begin{lem}[\textsc{the sum of all feasible pinchings of a blue face equals one}]\label{lemma: weight of all possible pinchings is 1}
Fix a blue face $B$ with disjoint vertices. 
Then
\begin{equation}\label{eq: weight of all pinching of face is one}
	\sum_{B'\in \AF(B)}w_{\infty}(B')=1,
\end{equation}
where if $B'$ has multiple faces $\{B'_i\}_{i=1}^k$, we set $w_{\infty}(B'):= \prod_{i=1}^kw_{\deg(B'_i)/2}$.
As a consequence, if $M\in\npe(s,K)$ is a non-separable planar embedded map and $B$ is one of its blue faces, then 
\begin{equation}\label{eq: weight of all pinching of face is one2}
	\sum_{M'\in \AM(M,B)}w_{\infty}(M') = \prod_{f\in \BF(M)\sm \{B\}}w_{\deg(f)/2},
\end{equation}
where we recall $\BF(M)$ is the set of blue faces of $M$. 
\end{lem}
\begin{proof}
Fix a blue face $B$ with disjoint vertices. To prove \eqref{eq: weight of all pinching of face is one} we repeatedly use Lemma~\ref{lemma: single vertex pinching cancellations}. We inform the reader that the argument in this proof will also be used later in Section~\ref{subsec: Weight of maps with at least one invalid pinching}.

Fix a vertex $v_1$ on $B$ and let $\AF(B, v_1)$ be the subset of $\AF(B)$ consisting of the $B'\in \AF(B)$ such that the two edges incident to $v_1$ form a 2-gon. Set $\AF^c(B,v_1):=\AF(B)\sm \AF(B,v_1)$, so that
\[\AF(B)=\AF(B, v_1) \sqcup  \AF^c(B,v_1).\]
See Figure~\ref{fig-unpinch-bij} for some examples. We claim that
\begin{equation}\label{eq: total weight 1 intermediate}
	\sum_{B'\in \AF^c(B,v_1)}w_{\infty}(B')=0.
\end{equation} 

We can partition $\AF^c(B,v_1)$ as follows:
\begin{equation}\label{eq:webfuowbef}
	\AF^c(B,v_1)=\AF^c_{\mathsf{p}}(B,v_1) \sqcup  \AF^c_{\mathsf{np}}(B,v_1),
\end{equation}
where the set  $\AF^c_{\mathsf{p}}(B,v_1)$ contains all the maps in $\AF^c(B,v_1)$ where $v_1$ has been pinched and $\AF^c_{\mathsf{np}}(B,v_1)$ contains all the other  maps in $\AF^c(B,v_1)$.

Now, for $B' \in \AF^c_{\mathsf{p}}(B,v_1)$, let $\upp(B')$ denote the family of faces obtained from $B'$ by unpinching $v_1$. See again Figure~\ref{fig-unpinch-bij}. Notice $\upp$ is a map from $\AF^c_{\mathsf{p}}(B,v_1)$ to $\AF^c_{\mathsf{np}}(B,v_1)$.

\begin{figure}[ht!]
	\begin{center}
		\includegraphics[width=.89\textwidth]{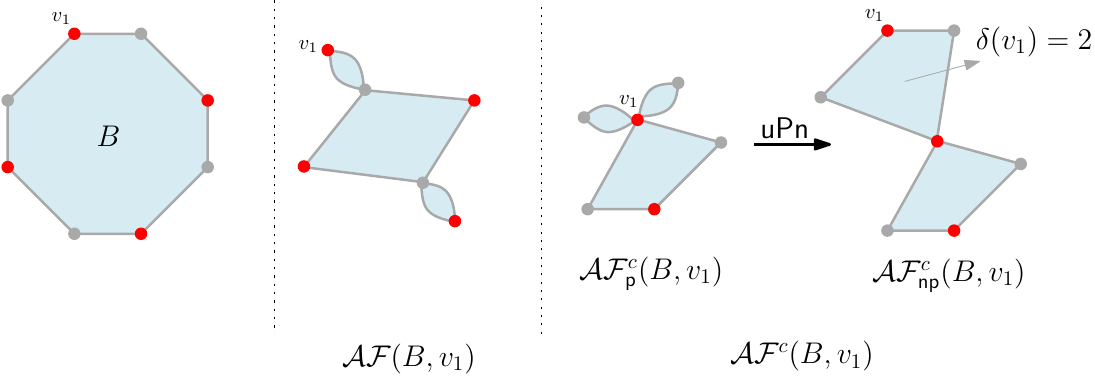}  
		\caption{\label{fig-unpinch-bij}\textbf{Left:} A blue face $B$ with disjoint vertices.
			\textbf{Middle:} A blue face in $\AF(B,v_1)$ such that the two edges incident to $v_1$ form a 2-gon.
			\textbf{Right:} Two blue face in $\AF^c(B,v_1)$. The first one is in $\AF^c_{\mathsf{p}}(B,v_1)$, while the second one is in $\AF^c_{\mathsf{np}}(B,v_1)$. Unpinching the first face at $v_1$,  we get the second face.}
	\end{center}
	\vspace{-3ex}
\end{figure}

Next, given $B' \in \AF^c_{\mathsf{np}}(B,v_1)$, let $\delta_{v_1}(B')$ denote half the degree of the blue face containing $v_1$. Note that since $B'$ is such that $v_1$ has not been pinched, this is well defined as $v_1$ is only contained in one face. There are exactly $\delta_{v_1}(B')-1$ collections of faces $B'' \in \AF^c_{\mathsf{p}}(B,v_1)$ such that $\upp(B'')=B'$. 
Thus, Lemma~\ref{lemma: single vertex pinching cancellations} gives us that \begin{align}\label{eq:eiwfvevfoueq}
	w_{\infty}(B') + \sum_{\substack{B''\in \AF^c_{\mathsf{p}}(B,v_1):\\ \upp(B'')=B'}}w_{\infty}(B'')=0.
\end{align}
Moreover, since $\AF^c_{\mathsf{p}}(B,v_1)=\bigsqcup_{B'\in \AF^c_{\mathsf{np}}(B,v_1)}\left\{B''\in\AF^c_{\mathsf{p}}(B,v_1):\upp(B'')=B'\right\}$, we can write
\begin{equation}\label{eq:eewfewefwefq}
	\sum_{B''\in \AF^c_{\mathsf{p}}(B,v_1)}w_{\infty}(B'')=\sum_{B'\in \AF^c_{\mathsf{np}}(B,v_1)}\sum_{\substack{B''\in \AF^c_{\mathsf{p}}(B,v_1):\\ \upp(B'')=B'}}w_{\infty}(B'').
\end{equation}
Combining the last two equations, we get that
\begin{align*}
	\sum_{B'\in \AF^c(B,v_1)}w_{\infty}(B') &\stackrel{\eqref{eq:webfuowbef}}{=} 
	\sum_{B'\in \AF^c_{\mathsf{np}}(B,v_1)}w_{\infty}(B') 
	+
	\sum_{B''\in \AF^c_{\mathsf{p}}(B,v_1)}w_{\infty}(B'')\\ 
	&\stackrel{\eqref{eq:eewfewefwefq}}{=}\sum_{B'\in \AF^c_{\mathsf{np}}(B,v_1)}\left[w_{\infty}(B') + \sum_{\substack{B''\in \AF^c_{\mathsf{p}}(B,v_1):\\ \upp(B'')=B'}}w_{\infty}(B'')\right]\stackrel{\eqref{eq:eiwfvevfoueq}}{=}0,
\end{align*} 
giving \eqref{eq: total weight 1 intermediate}. Hence
\begin{equation}\label{eq:ewjbweboebfewwe}
	\sum_{B'\in\AF(B)} w_{\infty}(B')=\sum_{B'\in\AF(B, v_1)}w_{\infty}(B').
\end{equation}
Now, we note that by removing the 2-gon containing $v_1$ from a family of faces in $\AF(B, v_1)$, we obtain a family of faces in $\AF(B^*)$, where $B^*$ is a blue face with disjoint vertices and has two fewer edges than $B$. The latter operation is a bijection from $\AF(B, v_1)$ to $\AF(B^*)$. Since the weight of a 2-gon is one, we get that 
\begin{equation*}
	\sum_{B'\in\AF(B)} w_{\infty}(B')=\sum_{B'\in\AF(B^*)} w_{\infty}(B').
\end{equation*}
Iterating the last formula, we can always arrive at a sum involving only a single 2-gon, i.e.\ to a sum that is equal to one. This proves \eqref{eq: weight of all pinching of face is one}.

Finally, fix a map $M\in \npe(s,K)$ and a blue face $B$ on $M$ (note $B$ is a blue face with disjoint vertices as $M$ is non-separable). Now \eqref{eq: weight of all pinching of face is one2} is a simple consequence of \eqref{eq: weight of all pinching of face is one}. Indeed,
\begin{align*}
	\sum_{M'\in \AM(M,B)}w_{\infty}(M') = \prod_{f\in \BF(M)\sm \{B\}}w_{\deg(f)/2}\cdot\sum_{B'\in \AF(B)}w_{\infty}(B')=\prod_{f\in \BF(M)\sm \{B\}}w_{\deg(f)/2}.
\end{align*}
\end{proof}

\begin{proof}[Proof of Theorem~\ref{thm:master-sum-blue-faces} (Case 1)]
Let $M\in\npe(s,K)$ be a non-separable planar embedded map and $B$ one of its blue faces such that  all its pinchings are valid. Then
\begin{align*}
	\sum_{M'\in \VM(M,B)}w_{\infty}(M')&=\sum_{M'\in \AM(M,B)}w_{\infty}(M')-\sum_{M'\in \IM(M,B)}w_{\infty}(M')\\
	&=\prod_{f\in \BF(M)\sm \{B\}}w_{\deg(f)/2},
\end{align*}
where the last equality follows from the results in \eqref{eq: weight of all pinching of face is one2}  of Lemma~\ref{lemma: weight of all possible pinchings is 1} and noting that the sum over $\IM(M,B)$ is zero since $\IM(M,B)=\emptyset$.
\end{proof}

\subsection{The sum of all valid pinchings of a blue face with one invalid pinching equals zero}\label{subsec: Weight of maps with at least one invalid pinching}

To prove the second case of Theorem~\ref{thm:master-sum-blue-faces}, we first need to better understand maps with at least one invalid pinching. 

Fix a non-separable planar embedded map $M\in\npe(\ell,K)$ and  one of its blue faces $B$. Note that since a (single) pinching of the vertices of the blue face $B$ creating a family of blue faces $B'$ cannot cause $B'$ to have non-distinct vertices, all the invalid collections of pinchings $\cpp\in\ICP(M,B)$ must create at least one enclosure loop in $\cpp(M)$ (recall the definition of enclosure loop from the paragraph below Definition~\ref{defn:non-separableplanar-embedded}). We chose one of these enclosure loops and we assume that it is made of blue faces all sent by the embedding to the lattice edge $e$.

Focusing on this enclosure loop, which we can assume to be a simple loop as remarked in Observation~\ref{obs:enclosure loops}, and unpinching one by one all the vertices of $B$ in $\cpp(B)$ that were pinched by $\cpp$ (in an arbitrarily fixed order), we see that this simple loop will be split into a collection of arcs connecting pairs of vertices of $B$, all of them being in the same partite class. See Figure~\ref{fig-minimal-pinching} for an example.

\begin{figure}[ht!]
\begin{center}
	\includegraphics[width=.89\textwidth]{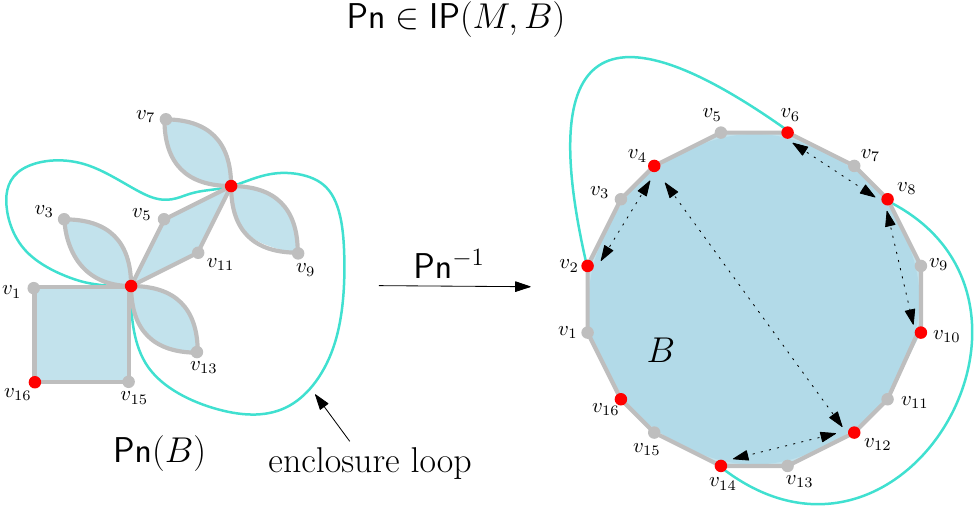}  
	\caption{\label{fig-minimal-pinching} An invalid pinching  $\cpp\in\ICP(M,B)$. \textbf{Left:} The pinched blue face $\cpp(B)$ has an enclosure loop.  \textbf{Right:} Unpinching all the vertices of $\cpp(B)$ that were pinched by $\cpp$, we see that this simple loop on the left becomes a collections of arcs connecting pairs of vertices of $B$ in the same partite class.} 
\end{center}
\vspace{-3ex}
\end{figure}

Each arc corresponds to a connected sequence of blue faces all sent by the embedding to the same lattice edge $e$ (note that $e$ must be different from the edge where the blue face $B$ is sent to since $M$ is non-separable).

From now on, we say that two vertices $u$ and $v$ of the same partite class on the boundary of $B$ are connected by an \textbf{arc sent to} $e$, if there exists a connected sequence of blue faces of $M$ all sent by the embedding to the lattice edge $e$ and connecting $u$ and $v$.

As a consequence of the discussion above, the set of all invalid collections of pinchings $\ICP(M,B)$ is uniquely determined by the set of all arcs of $B$ sent to any edge of the lattice.

\medskip

Next we introduce a useful definition. Given two vertices $u$ and $v$ on a blue face $B$ we define the \textbf{vertices between} $u$ \textbf{and} $v$ to be the vertices on $B$ that lie between $u$ and $v$ when going around the boundary of $B$ counter-clockwise starting from $u$ and ending at $v$. 

With this, we are able to state the following property about maps with invalid pinchings which will be crucial to prove Theorem~\ref{thm:master-sum-blue-faces}.

\begin{lem}\label{lemma: minimal arc}
Consider $M\in\npe(\ell,K)$ and one of its blue faces $B$. If $M$ is such that there is at least one invalid pinching of $B$ then there are two vertices $u$ and $v$ on $B$ of the same partite class that are connected by an arc and such that each vertex between $u$ and $v$ of the opposite partite class is not connected by an arc to any other vertex of $B$.
\end{lem}

\begin{proof}
By the discussion at the beginning of this section, we know that $B$ must have at least two vertices connected by an arc as it has an invalid pinching. Let $u_1$ and $v_1$ be two such vertices. Now if all the vertices between $u_1$ and $v_1$ of the opposite partite class are not connected by an arc to another vertex of $B$ then we are done. So assume that there is a vertex $w$ of the opposite partite class and between $u_1$ and $v_1$ such that it is connected by an arc to another vertex $z$ on $B$. 

First notice $z$ cannot be between $v_1$ and $u_1$. This is because, if it was the case, then the arc connecting $w$ and $z$ must intersect the arc connecting $u_1$ and $v_1$ and this is not possible. Indeed, without loss of generality, we can assume that $u_1$ and $v_1$ are sent to the starting vertex of $e$, and $w$ and $z$ are sent to the ending vertex of $e$.  We can also assume that the arc connecting $u_1$ and $v_1$ is sent to some lattice edge $e'\neq e$ and the one connecting $w$ and $z$ is sent to  $e''\neq e$. But then we must have that $e'$ is one of the lattice edges incident to the starting vertex of $e$ and $e''$ is one of the lattice edges incident to the ending vertex of $e$. Therefore, $e'$ and $e''$ do not share any vertex of $\ZZ^d$ and so the two arcs cannot cross each other. Indeed, this is only possible if the corresponding sequences of blue faces both have a vertex that is sent to the same lattice vertex. 

Thus we can assume that $z$ is in between $u_1$ and $v_1$. Without loss of generality, we assume that when going around the boundary of $B$ counter-clockwise, the vertex $w$ is visited before $z$. Then let $u_2=w$ and $v_2=v$. Now if all the vertices between $u_2$ and $v_2$ of the opposite partite class are not connected by an arc to any other vertex on $B$ we are done. If not we can repeat the above arguments giving vertices $u_3$ and $v_3$. 

Notice since $B$ has a finite number of vertices and the number of vertices between $u_{i+1}$ and $v_{i+1}$ is strictly less than the number of vertices between $u_i$ and $v_i$ this procedure must terminate in a finite number of steps. That is, at some point, we must find vertices $u_n$ and $v_n$ that are connected by an arc such that all the vertices between $u_n$ and $v_n$ of the opposite partite class are not connected by an arc to any other vertex of $B$.
\end{proof}

With this, we are ready to finish the proof of Theorem~\ref{thm:master-sum-blue-faces}.

\begin{proof}[Proof of Theorem~\ref{thm:master-sum-blue-faces} (Case 2)]
Suppose $M\in \npe(\ell,K)$ and one of its blue faces $B$ is such that $B$ has at least one invalid pinching. Then, by Lemma~\ref{lemma: minimal arc}, there exist two vertices $u$ and $v$ of $B$ in the same partite class that are connected by an arc and such that each vertex between $u$ and $v$ of the opposite partite class are not connected by an arc to any other vertex on $B$. Consider the set of vertices $\{w_i\}_{i=0}^n$ where $w_0=u$, $w_n=v$, and $w_i$ are the vertices between $u$ and $v$ of the same partite class in counter-clockwise order. Similarly let $\{z_i\}_{i=1}^{n}$ be the vertices between $u$ and $v$ of the opposite partite class in counter-clockwise order.

\medskip

We denote by $\VF(B)$ the set of collection of blue faces $\cpp(B)$ obtained from some $\cpp\in\VCP(M,B)$. Note that $\VF(B)$ depends on $M$.
Let $\VF(B, z_1)$ be the subset of $\VF(B)$ consisting of the $B'\in \VF(B)$ such that the two edges incident to $z_1$ form a 2-gon. Set $\VF^c(B,z_1):=\VF(B)\sm \VF(B,z_1)$, so that
\[\VF(B)=\VF(B, z_1) \sqcup  \VF^c(B,z_1).\]
We claim that
\begin{equation}\label{eq: total weight 1 intermediate2}
	\sum_{B'\in \VF^c(B,z_1)}w_{\infty}(B')=0.
\end{equation} 
If this is the case, then
\begin{equation*}
	\sum_{B'\in\VF(B)} w_{\infty}(B')=\sum_{B'\in\VF(B, z_1)}w_{\infty}(B').
\end{equation*}
Note that if $n=1$, then $\VF(B, z_1)=\emptyset$, since if the two edges incident to $z_1$ form a 2-gon, then $w_0=u$, $w_1=v$ are pinched together forming an invalid map (recall that $u$ and $v$ are connected by an arc). So, if $n=1$, we immediately get that $	\sum_{B'\in\VF(B)} w_{\infty}(B')=0$. This last idea combined with an iteration argument similar to the one used below \eqref{eq:ewjbweboebfewwe}, gives the general $n\geq2$ case.

\medskip

It remains to prove \eqref{eq: total weight 1 intermediate2}. We run an argument similar to the one used for the proof of Lemma~\ref{lemma: weight of all possible pinchings is 1}.\footnote{We need to repeat the argument because we are now working with only valid maps instead of all maps as in Lemma~\ref{lemma: weight of all possible pinchings is 1} and a priori this difference might create potential issues in the cancellations.} We partition $\VF^c(B,z_1)$ as follows:
\begin{equation}\label{eq:webfuowbef-2}
	\VF^c(B,z_1)=\VF^c_{\mathsf{p}}(B,z_1) \sqcup  \VF^c_{\mathsf{np}}(B,z_1),
\end{equation}
where the set  $\VF^c_{\mathsf{p}}(B,z_1)$ contains all the maps in $\VF^c(B,z_1)$ where $z_1$ has been pinched and $\VF^c_{\mathsf{np}}(B,z_1)$ contains all the other  maps in $\VF^c(B,z_1)$.

Now, for $B' \in \VF^c_{\mathsf{p}}(B,z_1)$, let $\upp(B')$, denote the family of faces obtained from $B'$ by unpinching $z_1$. Notice $\upp$ is a map from $\VF^c_{\mathsf{p}}(B,z_1)$ to $\VF^c_{\mathsf{np}}(B,z_1)$ since the unpinching operation cannot create invalid faces.

Next, given a collection of blue faces $B' \in \VF^c_{\mathsf{np}}(B,z_1)$, let $\delta_{z_1}(B')$ denote half the degree of the blue face $B'_{z_1}$ containing $z_1$. Since we know that $z_1$ is not connected by an arc to any other vertex on $B$, and so every (single) pinching of $z_1$ with another vertex of $B'_{z_1}$ in the same partite class leads to a new valid collection of faces, we get that there are exactly $\delta_{z_1}(B')-1$ collections of faces $B'' \in \VF^c_{\mathsf{p}}(B,z_1)$ such that $\upp(B'')=B'$. 
Thus, Lemma~\ref{lemma: single vertex pinching cancellations} gives us that \begin{align}\label{eq:eiwfvevfoueq-2}
	w_{\infty}(B') + \sum_{\substack{B''\in \VF^c_{\mathsf{p}}(B,z_1):\\ \upp(B'')=B'}}w_{\infty}(B'')=0.
\end{align}
Moreover, we can write
\begin{equation}\label{eq:eewfewefwefq-2}
	\sum_{B''\in \VF^c_{\mathsf{p}}(B,z_1)}w_{\infty}(B'')=\sum_{B'\in \VF^c_{\mathsf{np}}(B,z_1)}\sum_{\substack{B''\in \VF^c_{\mathsf{p}}(B,z_1):\\ \upp(B'')=B'}}w_{\infty}(B'').
\end{equation}
Combining the last two equations, we get that
\begin{align*}
	\sum_{B'\in \VF^c(B,z_1)}w_{\infty}(B') &\stackrel{\eqref{eq:webfuowbef-2}}{=} 
	\sum_{B'\in \VF^c_{\mathsf{np}}(B,z_1)}w_{\infty}(B') 
	+
	\sum_{B''\in \VF^c_{\mathsf{p}}(B,z_1)}w_{\infty}(B'')\\ 
	&\stackrel{\eqref{eq:eewfewefwefq-2}}{=}\sum_{B'\in \VF^c_{\mathsf{np}}(B,z_1)}\left[w_{\infty}(B') + \sum_{\substack{B''\in \VF^c_{\mathsf{p}}(B,z_1):\\ \upp(B'')=B'}}w_{\infty}(B'')\right]\stackrel{\eqref{eq:eiwfvevfoueq-2}}{=}0,
\end{align*}
giving \eqref{eq: total weight 1 intermediate}. 
\end{proof}

\section{Proof of the Master surface cancellation lemma}\label{sec: cancellation lemma}

The main goal of this section is to prove the Master surface cancellation lemma~\ref{lemma:master-cancellation}. First, in Section~\ref{sect:partitions} we introduce convenient partitions of the sets of bad embedded maps appearing in the statement of the Master surface cancellation lemma~\ref{lemma:master-cancellation}. Then, in Section~\ref{sect:proofoflemma}, we prove the Master surface cancellation lemma~\ref{lemma:master-cancellation} assuming two preliminary results (Lemmas~\ref{lem:cance1}~and~\ref{lem:cance2}) whose proofs are given in Section~\ref{sect:lemmas-one}. The surface cancellations established in Theorem~\ref{thm:master-sum-blue-faces} will be the fundamental tool needed to prove Lemma~\ref{lem:cance2}.

\subsection{Partitioning the sets of bad embedded maps}\label{sect:partitions}

The main goal of this section is to present a convenient way to partition the sets of maps appearing in the statement of the Master surface cancellation lemma~\ref{lemma:master-cancellation}.

Fix a loop $\ell$ such that $\ell=\mathbf{e} \, \pi$, where $\mathbf{e}$ is a copy of the lattice edge $e$ and a  plaquette assignment $K$ such that $(\ell,K)$ is balanced. We consider the embedded maps in the sets

\begin{align}\label{eq:new_notation_maps}
\begin{split}
	\cM^{\text{bad}}_{-}&:=\bigsqcup_{p\in\cP_{\ZZ^d}(e^{-1},K)} \npe^{\text{bad}}(\ell\ominus_{\mathbf{e}}p , K\sm\{p\}),\\
	\cM^{\text{bad}}_{+}&:=\bigsqcup_{q\in\cP_{\ZZ^d}(e,K)} \npe^{\text{bad}}(\ell\oplus_{\mathbf{e}}q , K\sm\{q\}).
\end{split}
\end{align}

See Figure~\ref{fig-bad-maps} for some examples.
Note that each embedded map  $M\in\cM^{\text{bad}}_{-}$ has boundary $\pi \, \nu_1$ where $\nu_1$ is such that $\mathbf{p}=\textbf{e}^{-1} \, \nu_1\in \cP_{\ZZ^d}(e^{-1},K)$. Here $\mathbf{p}\in \cP_{\ZZ^d}(e^{-1},K)$ is a slight abuse of notation to indicate that $\mathbf{p}$ is an internal yellow face mapped to the lattice plaquette $p\in \cP_{\ZZ^d}(e^{-1},K)$. Throughout this section, we will use such a convention.

Let $f_r(M)$ denote the internal yellow face  $\mathbf{p}$ above (where $f$ stands for ``face'' and $r$ stands for ``removed'', since $f_r(M)$ corresponds to the face that has been removed by the PPS process).

Similarly, each embedded map  $M\in\cM^{\text{bad}}_+$ has boundary $\textbf{e}\, \nu_2 \, \textbf{e}' \pi$ where $\nu_2$ is such that $\mathbf{q}=\textbf{e}'\, \nu_2\in \cP_{\ZZ^d}(e,K)$. In this case, we let $f_r(M)$ denote the internal yellow face $\mathbf{q}$.

\begin{figure}[ht!]
\begin{center}
	\includegraphics[width=.99\textwidth]{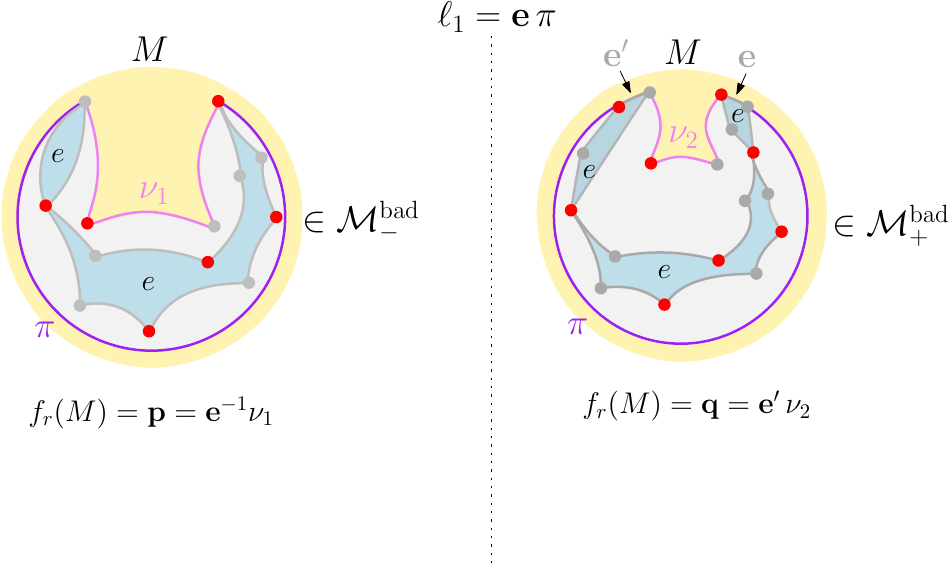}  
	\caption{\label{fig-bad-maps} Two examples of embedded maps, both obtained starting from a loop $\ell$ such that $\ell=\mathbf{e} \, \pi$, where $\mathbf{e}$ is a copy of the lattice edge $e$ and a fixed plaquette assignment $K$. \textbf{Left}: A map in $\cM^{\text{bad}}_{-}$ obtained from a map in $\npe^{\text{bad}}(\ell\ominus_{\mathbf{e}}p , K\sm\{p\})$ with $\mathbf{p}=\textbf{e}^{-1} \, \nu_1\in \cP_{\ZZ^d}(e^{-1},K)$.
		\textbf{Right}: A map in $\cM^{\text{bad}}_{+}$ obtained from a map in $\npe^{\text{bad}}(\ell\oplus_{\mathbf{e}}q , K\sm\{q\})$ with $\mathbf{q}=\textbf{e}'\, \nu_2\in \cP_{\ZZ^d}(e,K)$.}
\end{center}
\vspace{-3ex}
\end{figure}

\medskip

From now on, whenever we write $\pm$, we mean that the results that we are explaining hold if we replace all the $\pm$ by all $+$ or by all $-$.

Recall that each map in $\cM^{\text{bad}}_{-}$ has a connected sequence of blue faces sent to the lattice edge $e$ that connect the starting and final vertex of $\pi$, while each map in $\cM^{\text{bad}}_{+}$ has a connected sequence of blue faces sent to the lattice edge $e$ that connect the edge $\textbf{e}$ and $\textbf{e}'$. We denote by $\cM^{\text{bad}}_{\text{u},{\pm}}$ the set of embedded maps in $\cM^{\text{bad}}_{\pm}$ such that the aforementioned connected sequence of blue faces is formed by a \emph{single} blue face sent to the lattice edge $e$ which is not connected (through vertices) to any other blue face sent to the lattice edge $e$. We will always denote this specific single blue face by $B_{\pm}$.  See Figure~\ref{fig-bad-maps-2} for some examples.

\begin{figure}[ht!]
\begin{center}
	\includegraphics[width=.99\textwidth]{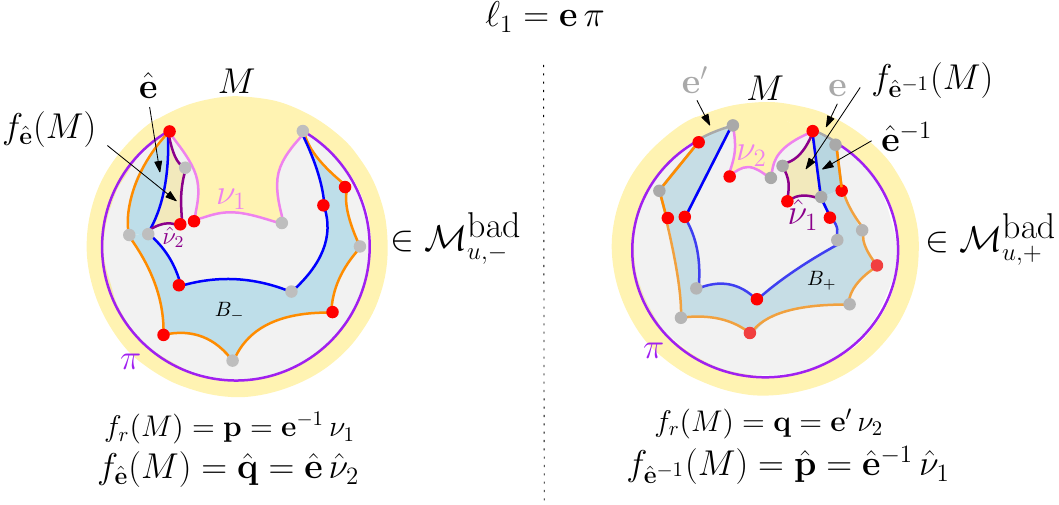}  
	\caption{\label{fig-bad-maps-2}
		Examples of embedded maps with a unique blue face $B_{\pm}$ sent to the lattice edge $e$. We are also assuming that $B_{\pm}$ is not connected (through vertices) to any other blue face sent to the lattice edge $e$. The top boundary of $B_{\pm}$ is highlighted in blue, while the bottom boundary is highlighted in orange.
		\textbf{Left}: An embedded map in $\cM^{\text{bad}}_{\text{u},-}$.
		\textbf{Right}: An embedded map in $\cM^{\text{bad}}_{\text{u},+}$.}
\end{center}
\vspace{-3ex}
\end{figure}

Recalling the notation $\VM(M,B)$ introduced below Observation~\ref{obs:valid-pinching}, we have the following partitions of the sets $\cM^{\text{bad}}_{\pm}$ introduced in \eqref{eq:new_notation_maps}:
\begin{equation*}
\cM^{\text{bad}}_{\pm}=\bigsqcup_{M\in\cM^{\text{bad}}_{\text{u},\pm}} \VM(M,B_{\pm}).
\end{equation*}
Note to obtain the aforementioned partition we used the fact that the blue face $B_{\pm}$ is not connected (through vertices) to any other blue face sent to the lattice edge $e$.

Next we further partition based on the size of the unique connecting face. Notice the boundary of $B_{-}$ (resp.\ $B_{+}$) is split into two pieces by the starting and final vertex of $\pi$ (resp.\ by the two edges $\textbf{e}$ and $\textbf{e}'$): the \textbf{bottom boundary} of $B_{\pm}$ which is on the same side as $\pi$ and the \textbf{top boundary} of $B_{\pm}$ which is on the same side as $\nu_1$ or $\nu_2$. We point out that the edges $\textbf{e}$ and $\textbf{e}'$ are not included in the top or bottom boundary of $B_+$. See Figure~\ref{fig-bad-maps-2} for some examples.

Let $\cM^{\text{bad}}_{\text{u},\pm}(t,b)$ denote the set of embedded maps in  $\cM^{\text{bad}}_{\text{u},\pm}$ with $t$ edges in the top boundary of $B_{\pm}$ and $b$ edges in the bottom boundary of $B_{\pm}$. Then, we get that
\begin{equation*}
\cM^{\text{bad}}_{\pm}=\bigsqcup_{t,b\geq 1}\bigsqcup_{M\in\cM^{\text{bad}}_{\text{u},\pm}(t,b)} \VM(M,B_{\pm}),
\end{equation*}
where we note that the sets $\cM^{\text{bad}}_{\text{u},\pm}(t,b)$ are non-empty only if $t$ and $b$ are both odd and $t+b\leq 2n_e(\ell,K)-2$. 

The next partition we construct will allow us to relate maps in $\cM^{\text{bad}}_{\text{u},-}$ to maps in $\cM^{\text{bad}}_{\text{u},+}$. For a map in  $\cM^{\text{bad}}_{\text{u},-}$, let $\hat{\textbf{e}}$ denote the edge of the top boundary incident to the starting vertex of $\pi$. Similarly, for a map in  $\cM^{\text{bad}}_{\text{u},+}$, let $\hat{\textbf{e}}^{-1}$ denote the edge of the top boundary which shares a vertex with $\mathbf{e}$. Our notation is justified by the following observation: the starting vertex of $\pi$ is sent by the embedding to the lattice vertex at the end of $e$ (because $\ell_1=\textbf{e} \, \pi$) and so $\hat{\textbf{e}}$ must be a copy of the lattice edge $e$. Similar reasoning gives that $\hat{\textbf{e}}^{-1}$ must be a copy of the lattice edge $e^{-1}$.

Since embedded maps have yellow/blue bipartite faces, the edge $\hat{\textbf{e}}$ must be on the boundary of an internal yellow face $\hat{\mathbf{q}}=\hat{\textbf{e}} \, \hat{\nu}_2\in\cP_{\ZZ^d}(e,K)$ for some path $\hat{\nu}_2$.  We denote the plaquette $\hat{\mathbf{q}}$ by $f_{\hat{\textbf{e}}}(M)$.

Similarly, the edge $\hat{\textbf{e}}^{-1}$ must be on the boundary of an internal yellow face $\hat{\mathbf{p}}=\hat{\textbf{e}}^{-1} \, \hat{\nu}_1\in\cP_{\ZZ^d}(e^{-1},K)$ for some path $\hat{\nu}_1$. We denote the plaquette $\hat{\mathbf{p}}$ by $f_{\hat{\textbf{e}}^{-1}}(M)$.

For $p\in\cP_{\ZZ^d}(e^{-1},K)$ and $q\in\cP_{\ZZ^d}(e,K)$, we set\footnote{We recall that, for instance, with $f_r(M)=\mathbf{p}$ we mean that $f_r(M)$ is an internal yellow face mapped to the lattice plaquette $p\in\cP_{\ZZ^d}(e^{-1},K)$.}
\begin{align*}
\cM^{\text{bad}}_{\text{u},-}(t,b,p,q)&:=
\left\{
M\in\cM^{\text{bad}}_{\text{u},-}(t,b)\,:\,
f_r(M)=\mathbf{p}\text{ and }  f_{\hat{\textbf{e}}}(M)=\mathbf{q}
\right\},\\
\cM^{\text{bad}}_{\text{u},+}(t,b,q,p)&:=
\left\{
M\in\cM^{\text{bad}}_{\text{u},+}(t,b)\,:\,
f_r(M)=\mathbf{q}\text{ and }  f_{\hat{\textbf{e}}^{-1}}(M)=\mathbf{p}
\right\}.
\end{align*}

This gives us the partition of the sets $\cM^{\text{bad}}_{\pm}$:
\begin{align}\label{eq:ineedaref}
\begin{split}
	\cM^{\text{bad}}_{-}&=
	\bigsqcup_{\substack{p\in\cP_{\ZZ^d}(e^{-1},K)\\ q\in\cP_{\ZZ^d}(e,K)}}
	\bigsqcup_{t,b\geq 1}\bigsqcup_{M\in\cM^{\text{bad}}_{\text{u},-}(t,b,p,q)} \VM(M,B_{-}),\\
	\cM^{\text{bad}}_{+}&=
	\bigsqcup_{\substack{p\in\cP_{\ZZ^d}(e^{-1},K)\\ q\in\cP_{\ZZ^d}(e,K)}}
	\bigsqcup_{t,b\geq 1}\bigsqcup_{M\in\cM^{\text{bad}}_{\text{u},+}(t,b,q,p)} \VM(M,B_{+}).
\end{split}
\end{align}
Lastly, we re-express the above partition in terms of the region external to $B_{\pm}$ and the plaquette $f_{\hat{\textbf{e}}}(M)/f_{\hat{\textbf{e}}^{-1}}(M)$. In particular, given an embedded map  $M\in\cM^{\text{bad}}_{\text{u},-}(t,b,p,q)$,
let $\un_-(M)$ be the (unknown) embedded map obtained by (c.f.\ with Figure~\ref{fig-bad-maps-4}) 
\begin{enumerate}
\item splitting the starting and final vertex of $\pi$ (in such a way that $B_-$ and the yellow face containing $\pi$ completely separate the map.);
\item removing the interior of the blue face $B_{-}$ and the external yellow face containing $\pi$ on the boundary; 
\item removing the edge $\hat{\mathbf{e}}$ and the interior of the yellow face $f_{\hat{\mathbf{e}}}(M)$.
\end{enumerate}
Similarly, given an embedded map  $M\in\cM^{\text{bad}}_{\text{u},+}(t,b,q,p)$,
let $\un_+(M)$ be the (unknown) embedded map obtained by 
\begin{enumerate}
\item removing the edges $\textbf{e}$ and $\textbf{e}'$ 
\item removing the interior of the blue face $B_{+}$ and of the external yellow face containing $\pi$ on the boundary. 
\item removing  the edge $\hat{\mathbf{e}}^{-1}$ and the interior of the yellow face $f_{\hat{\mathbf{e}}^{-1}}(M)$.
\end{enumerate} 
Note that $\un_{\pm}(M)$ always consists of two connected components, one including the top boundary of $B_{\pm}$, denoted by  $\un_{\pm}^t(M)$, and one including the bottom boundary of $B_{\pm}$, denoted by  $\un_{\pm}^b(M)$. Hence $\un_{\pm}(M)=(\un_{\pm}^t(M),\un_{\pm}^b(M))$.

\begin{figure}[ht!]
\begin{center}
	\includegraphics[width=.99\textwidth]{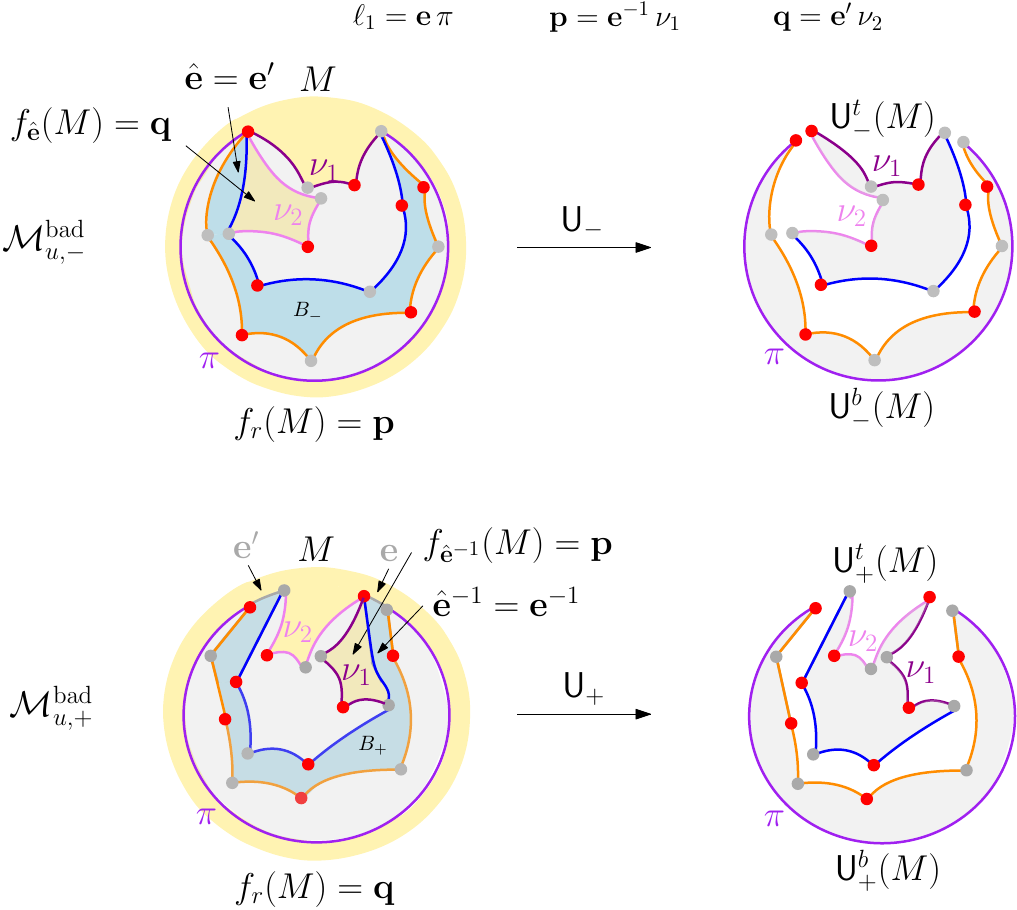}  
	\caption{\label{fig-bad-maps-4} Two examples on how the functions $\un_{\pm}$ act on embedded maps.
	}
\end{center}
\vspace{-3ex}
\end{figure}

Clearly the function $\un_{-}(\cdot)$ is a bijection from $$\cM^{\text{bad}}_{\text{u},-}(t,b,p,q)\quad\text{to}\quad\cU^{\text{bad}}_{\text{u},-}(t,b,p,q):=\un_-\left(
\cM^{\text{bad}}_{\text{u},-}(t,b,p,q)
\right)$$ 
and the function $\un_+(\cdot)$ is a bijection from 
$$\cM^{\text{bad}}_{\text{u},+}(t,b,q,p)\quad\text{to}\quad\cU^{\text{bad}}_{\text{u},+}(t,b,q,p):=\un_+\left(
\cM^{\text{bad}}_{\text{u},+}(t,b,q,p)
\right).$$
With this, we can rewrite (\ref{eq:ineedaref}) as:
\begin{align}\label{eq:partitions}
\begin{split}
	\cM^{\text{bad}}_{-}&=
	\bigsqcup_{\substack{p\in\cP_{\ZZ^d}(e^{-1},K)\\ q\in\cP_{\ZZ^d}(e,K)}}
	\bigsqcup_{t,b\geq 1}\bigsqcup_{U\in\cU^{\text{bad}}_{\text{u},-}(t,b,p,q)} \VM(\un^{-1}_-(U),B_{+}),\\
	\cM^{\text{bad}}_{+}&=
	\bigsqcup_{\substack{p\in\cP_{\ZZ^d}(e^{-1},K)\\ q\in\cP_{\ZZ^d}(e,K)}}
	\bigsqcup_{t,b\geq 1}\bigsqcup_{U\in\cU^{\text{bad}}_{\text{u},+}(t,b,q,p)} \VM(\un^{-1}_+(U),B_{-}).
\end{split}
\end{align}

\subsection{Proof of the Master surface cancellation lemma}\label{sect:proofoflemma}

In this section we prove the Master surface cancellation lemma~\ref{lemma:master-cancellation}.
The main reasons to introduce the above partitions in \eqref{eq:partitions} are the following two results whose proofs are postponed to Section~\ref{sect:lemmas-one}. The first tells us that the sets of possible unknown regions are the same.

\begin{lem}\label{lem:cance1}
Fix $t,b\geq 1$, $p\in\cP_{\ZZ^d}(e^{-1},K)$ and $q\in\cP_{\ZZ^d}(e,K)$. Then
\begin{equation*}
	\cU^{\text{bad}}_{\text{u},-}(t,b,p,q)
	=
	\cU^{\text{bad}}_{\text{u},+}(t,b,q,p).
\end{equation*}
\end{lem}

As a consequence, we are allowed to introduce the new simplified notation
\begin{equation*}
\cU^{\text{bad}}_{\text{u}}(t,b,p,q):=
\cU^{\text{bad}}_{\text{u},-}(t,b,p,q)
=
\cU^{\text{bad}}_{\text{u},+}(t,b,q,p).
\end{equation*}
\begin{lem}\label{lem:cance2}
Fix $t,b\geq 1$, $p\in\cP_{\ZZ^d}(e^{-1},K)$ and $q\in\cP_{\ZZ^d}(e,K)$.
Fix $U\in \cU^{\text{bad}}_{\text{u}}(t,b,p,q)$. Then
\begin{equation*}
	\sum_{M\in\VM(\un^{-1}_-(U),B_{-})}\upbeta^{\area(M)}w_{\infty}(M)
	=
	\sum_{M\in\VM(\un^{-1}_+(U),B_{+})}\upbeta^{\area(M)}w_{\infty}(M).
\end{equation*}
\end{lem}

We can now prove the Master surface cancellation lemma~\ref{lemma:master-cancellation}

\begin{proof}[Proof of the Master surface cancellation lemma~\ref{lemma:master-cancellation} assuming Lemmas~\ref{lem:cance1}~and~\ref{lem:cance2}]
Fix a loop $\ell$ such that $\ell=\mathbf{e} \, \pi$, where $\mathbf{e}$ is a copy of the lattice edge $e$. Also fix a plaquette assignment $K$ such that $(\ell,K)$ is balanced.

It is enough to prove that:
\begin{align}\label{eq:cancellation-one-2}
	\sum_{M\in \cM^{\text{bad}}_-}\upbeta^{\area(M)}w_{\infty}(M)
	-
	\sum_{M\in \cM^{\text{bad}}_+}\upbeta^{\area(M)}w_{\infty}(M)=0.
\end{align}
Using the partition in \eqref{eq:partitions}, we can rewrite the left-hand side of \eqref{eq:cancellation-one-2} as
\begin{multline*}
	\sum_{\substack{p\in\cP_{\ZZ^d}(e^{-1},K)\\ q\in\cP_{\ZZ^d}(e,K)}}
	\sum_{t,b\geq 1}
	\sum_{U\in\cU^{\text{bad}}_{\text{u},-}(t,b,p,q)} 
	\sum_{M\in \VM(\un^{-1}_-(U),B_{-})}
	\upbeta^{\area(M)}w_{\infty}(M)\\
	-
	\sum_{\substack{p\in\cP_{\ZZ^d}(e^{-1},K)\\ q\in\cP_{\ZZ^d}(e,K)}}
	\sum_{t,b\geq 1}
	\sum_{U\in\cU^{\text{bad}}_{\text{u},+}(t,b,q,p)} 
	\sum_{M\in\VM(\un^{-1}_+(U),B_{+})}\upbeta^{\area(M)}w_{\infty}(M).
\end{multline*}
The fact that the above sum is zero is now a simple consequence of Lemmas~\ref{lem:cance1}~and~\ref{lem:cance2}.
\end{proof}

\subsection{Proofs of the two remaining lemmas}\label{sect:lemmas-one}

We give the proof of Lemma~\ref{lem:cance1}.

\begin{proof}[Proof of Lemma~\ref{lem:cance1}]
Fix $t,b\geq 1$ (both odd) such that $t+b \leq 2n_e(\ell,K)-2$, $p\in\cP_{\ZZ^d}(e^{-1},K)$ and $q\in\cP_{\ZZ^d}(e,K)$.
We want to show that
\begin{equation*}
	\cU^{\text{bad}}_{\text{u},-}(t,b,p,q)
	=
	\cU^{\text{bad}}_{\text{u},+}(t,b,q,p),
\end{equation*}
where we recall that (recall also Figure~\ref{fig-bad-maps-4}),
\begin{equation*}
	\cU^{\text{bad}}_{\text{u},-}(t,b,p,q)=\un_-\left(
	\cM^{\text{bad}}_{\text{u},-}(t,b,p,q)
	\right), \quad\cU^{\text{bad}}_{\text{u},+}(t,b,q,p)=\un_+\left(
	\cM^{\text{bad}}_{\text{u},+}(t,b,q,p)
	\right).
\end{equation*}

To show the desired equality, we show both inclusions.
Fix $U=(U^t,U^b)\in\cU^{\text{bad}}_{\text{u},+}(t,b,q,p)$ such that $U=\un_{+}(M_+)$ for some $M_+\in \cM^{\text{bad}}_{\text{u},+}(t,b,q,p)$ that has internal yellow faces $f_r(M)=\textbf{q}=\mathbf{e}'\, \nu_2$ and $f_{\hat{\textbf{e}}^{-1}}(M)=\textbf{p}=\mathbf{e}^{-1} \, \nu_1$ mapped to $q$ and $p$, respectively. Then (c.f.\ the middle picture in Figure~\ref{fig-bad-maps-6})
\begin{itemize}
	\item $U^t$ has boundary $\nu_1 \, \nu_2\, e^{-1}\,e\,\dots\,e^{-1}\,e$, where $e^{-1}\,e$ is repeated $(t-1)/2$ times;
	\item $U^b$ has boundary $\pi \, e^{-1}\,e\,\dots\,e^{-1}\,e\,e^{-1}$, where $e^{-1}\,e$ is repeated $(b-1)/2$ times;
	\item $U^t$ and $U^b$ are connected and planar.
\end{itemize}

We want to show that $U\in\cU^{\text{bad}}_{\text{u},-}(t,b,p,q)$, that is, that there exists $M_-\in \cM^{\text{bad}}_{\text{u},-}(t,b,p,q)$ such that $\un_{-}(M_-)=U$. 
We first construct a map $F:\cU^{\text{bad}}_{\text{u},+}(t,b,q,p)\to\cM^{\text{bad}}_{\text{u},-}(t,b,p,q)$ as follows (c.f.\ Figure~\ref{fig-bad-maps-6}):
\begin{enumerate}
	\item Start with $U$ and add an edge $\hat{\textbf{e}}$ (a copy of the lattice edge $e$) on the exterior of $U^t$ between the starting and ending vertices of $\nu_2$. This creates a new face, declare it to be an internal yellow face;
	\item Identify the vertex shared by $\nu_1$ and $\nu_2$ with the starting vertex of $\pi$; 
	\item Identify the starting vertex of $\nu_1$ with the ending vertex of $\pi$;
	\item The last two steps create two new faces, declare the one containing $\hat{\textbf{e}}$ on the boundary to be a blue face and call it $B_{-}$. Declare the other face to be an external yellow face.
\end{enumerate} 

\begin{figure}[ht!]
	\begin{center}
		\includegraphics[width=.99\textwidth]{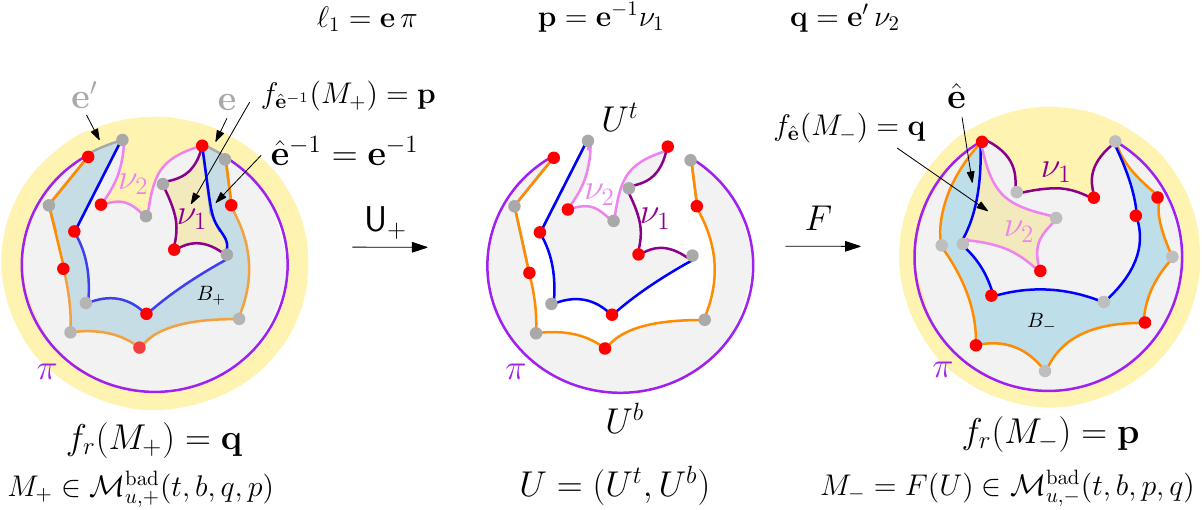}  
		\caption{\label{fig-bad-maps-6} A schema showing how to construct $M_-=F(U)$. 
		}
	\end{center}
	\vspace{-3ex}
\end{figure}

Assume for a moment that $M_-:=F(U)\in \cM^{\text{bad}}_{\text{u},-}(t,b,p,q)$. Recalling the description of the function $\un_-$ from the discussion below \eqref{eq:ineedaref}, we note that $F$ acts as $\un_{-}^{-1}$. Then, we immediately have that $\un_{-}(M_-)=U$. 

Hence we are only left to check that $M_-\in \cM^{\text{bad}}_{\text{u},-}(t,b,p,q)$. From the constructions above and the properties of the map $U$ described above, the only non-trivial property of $M_-$ that we need to check to be sure that $M_-\in \cM^{\text{bad}}_{\text{u},-}(t,b,p,q)$ is that $M_-$ is non-separable. Indeed, all the other necessary properties are simple to check.

Suppose, for contradiction, that $M_-$ is separable. Let $\EL$ denote the faces on an enclosure loop in $M_-$ corresponding to a lattice edge $e'$. Note that there are three different cases:
\begin{enumerate}
	\item $\EL$ does not include $B_{-}$ and is contained in one of the two regions of the embedded map $M_-$ corresponding to $U^t$ or $U^b$. See the left-hand side of Figure~\ref{fig-bad-maps-7}.
	\item $\EL$ includes $B_{-}$ and so is equal to $B_{-}$ plus a connected sequence of blue faces connecting two vertices the top or bottom boundary of $B_{-}$. On this case it must be that $e'=e$. See the middle picture of Figure~\ref{fig-bad-maps-7}.
	\item $\EL$ does not includes $B_{-}$ and is not contained in one of the two regions of the embedded map $M_-$ corresponding to $U^t$ or $U^b$. If this is the case then $\EL$ is formed by two connected sequence of blue faces, both connecting the starting and ending vertex of $\pi$, one included in the regions of the embedded map $M_-$ corresponding to $U^t$ and the other one included in the regions of the embedded map $M_-$ corresponding to $U^b$. See the right-hand side of Figure~\ref{fig-bad-maps-7}.
\end{enumerate}

\begin{figure}[ht!]
	\begin{center}
		\includegraphics[width=.99\textwidth]{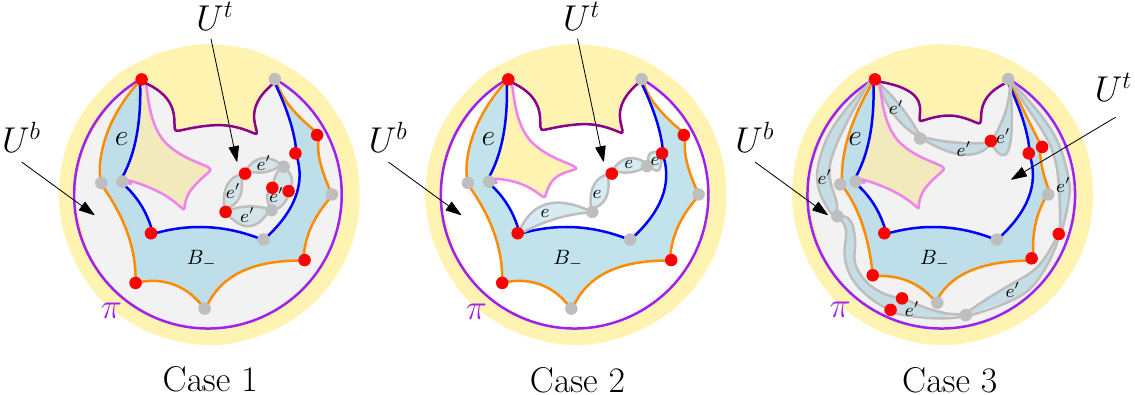}  
		\caption{\label{fig-bad-maps-7} 
			Three potential enclosure loops that might prevent $M_-$ from being non-separable.
		}
	\end{center}
	\vspace{-3ex}
\end{figure}

For Cases 1 and 2, it is straightforward to derive a contradiction. Indeed, in both cases, by examining the embedded map $U_+^{-1}(F^{-1}(M_-))=M_+\in \cM^{\text{bad}}_{\text{u},+}(t,b,q,p)$, one finds that $M_+$ would have an enclosure loop. This leads to a contradiction.

Case 3 is more subtle. Indeed, by examining the embedded map $\un_+^{-1}(F^{-1}(M_-))=M_+\in \cM^{\text{bad}}_{\text{u},+}(t,b,q,p)$, one cannot deduce that $M_+$ would have an enclosure loop (since the vertices at the beginning and end of $\pi$ in $M_-$ are split in $M_+$). To solve this issue, one has to note the following fact.

\medskip

\noindent\textbf{Claim}. In Case 3, the enclosure loop $\EL$ must correspond to the same  lattice edge $e$ as the one of $B_{-}$. 

\medskip

Note that it is simple to get a new contradiction from the claim which follows by noting that the vertices at the beginning and end of $\pi$ are sent to the lattice vertices at the beginning and end of $e$. Therefore $e'=e$.

\bigskip

We now proceed with the other inclusion (the proof is quite similar, but simpler). Fix $U=(U^t,U^b)\in\cU^{\text{bad}}_{\text{u},-}(t,b,p,q)$ such that $U=\un_{-}(M)$ for some $M_-\in \cM^{\text{bad}}_{\text{u},-}(t,b,p,q)$ that has internal yellow faces $f_r(M)=\textbf{p}=\mathbf{e}^{-1} \, \nu_1$ and $f_{\hat{\textbf{e}}}(M)=\textbf{q}=\mathbf{e}'\, \nu_2$ mapped to $p$ and $q$, respectively. Then it is simple to realize that $U^t$ and $U^b$ have the same properties as in the previous inclusion.

We want to show that $U\in\cU^{\text{bad}}_{\text{u},+}(t,b,p,q)$, that is, that there exists $M_+\in \cM^{\text{bad}}_{\text{u},+}(t,b,q,p)$ such that $\un_{+}(M_+)=U$. 
We first construct a map $G: \cU^{\text{bad}}_{\text{u},-}(t,b,p,q) \to \cM^{\text{bad}}_{\text{u},+}(t,b,q,p)$ as follows (c.f.\ Figure~\ref{fig-bad-maps-5}):
\begin{enumerate}
	\item Start with $U$ and add an edge $\hat{\textbf{e}}^{-1}$ (a copy of the lattice edge $e^{-1}$) on the exterior of $U^t$ between the starting and ending vertices of $\nu_1$. This creates a new face, declare it to be a yellow face. 
	\item Add an edge $\mathbf{e}'$ (a copy of the lattice edge $e$) from the starting vertex of $\pi$ to the final vertex of $\nu_2$.
	\item Add an edge $\mathbf{e}$ (a copy of the lattice edge $e$) from the ending vertex of $\pi$ to the vertex shared by $\nu_1$ and $\nu_2$. 
	\item The last two steps creates two new faces, declare the one containing $\hat{\textbf{e}}^{-1}$ on the boundary to be a blue face and call it $B_{+}$. Declare the other face to be a external yellow face.
\end{enumerate} 

\begin{figure}[ht!]
	\begin{center}
		\includegraphics[width=.99\textwidth]{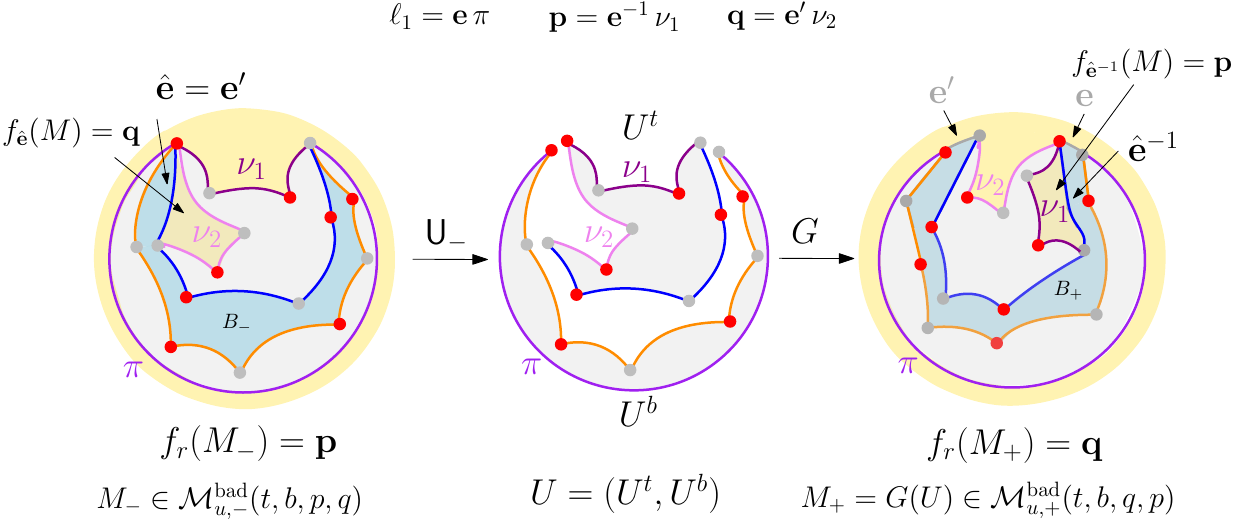}  
		\caption{\label{fig-bad-maps-5} A schema showing how to construct $M_+=G(U)$.
		}
	\end{center}
	\vspace{-3ex}
\end{figure}

Assume for a moment that $M_+:= G(U)\in \cM^{\text{bad}}_{\text{u},+}(t,b,q,p)$. Recalling the description of the function $\un_+$ from the discussion below \eqref{eq:ineedaref}, we note that $G$ acts as $\un_{+}^{-1}$. Then, we immediately have that $\un_{+}(M_+)=U$. 

Hence we are only left to check that $M_+\in \cM^{\text{bad}}_{\text{u},+}(t,b,q,p)$. This follows using the same proof used for the previous inclusion (with the advantage that the third type of enclosure loops no longer need to be considered). This ends the proof of the lemma.
\end{proof}

We finally give the proof of Lemma~\ref{lem:cance2}. 

\begin{proof}
Fix $t,b\geq 1$ such that $t+b\leq 2n_e(\ell,K)-2$, $p\in\cP_{\ZZ^d}(e^{-1},K)$ and $q\in\cP_{\ZZ^d}(e,K)$.
Fix $U\in \cU^{\text{bad}}_{\text{u}}(t,b,p,q)$, where we recall that
\begin{equation*}
	\cU^{\text{bad}}_{\text{u}}(t,b,p,q):=
	\cU^{\text{bad}}_{\text{u},-}(t,b,p,q)
	=
	\cU^{\text{bad}}_{\text{u},+}(t,b,q,p).
\end{equation*}
To show that
\begin{equation}\label{eq:ewouewbf}
	\sum_{M\in\VM(\un^{-1}_-(U),B_{-})}\upbeta^{\area(M)}w_{\infty}(M)
	=
	\sum_{M\in\VM(\un^{-1}_+(U),B_{+})}\upbeta^{\area(M)}w_{\infty}(M),
\end{equation}
it is enough to note that thanks to Theorem~\ref{thm:master-sum-blue-faces},
\begin{equation*}
	\sum_{M\in \VM(\un^{-1}_{\pm}(U),B_{\pm})}w_{\infty}(M)=
	\begin{cases}
		C_{\pm}, \,&\text{if all the pinchings of $B_{\pm}$ are valid},\\
		0, \,&\text{if $B_{\pm}$ has at least one invalid pinching},
	\end{cases}
\end{equation*}
where $C_{\pm}:=\prod_{f\in \BF(\un^{-1}_{\pm}(U))\sm \{B_{\pm}\}}w_{\deg(f)/2}$. Indeed, \eqref{eq:ewouewbf} immediately follows from the last displayed equation by noting that, thanks to the descriptions of the maps $\un^{-1}_{\pm}$ given in the above proof of Lemma~\ref{lem:cance1} (see also Figure~\ref{fig-bad-maps-8}), all the pinchings of $B_{-}$ are valid if and only if all the pinchings of $B_{+}$ are valid, and moreover, $C_{-}=C_{+}$ since the blue faces in $\BF(\un^{-1}_{-}(U))\sm \{B_{-}\}$ are identical to the blue faces in $\BF(\un^{-1}_{+}(U))\sm \{B_{+}\}$.
\end{proof}

\begin{figure}[ht!]
\begin{center}
	\includegraphics[width=.99\textwidth]{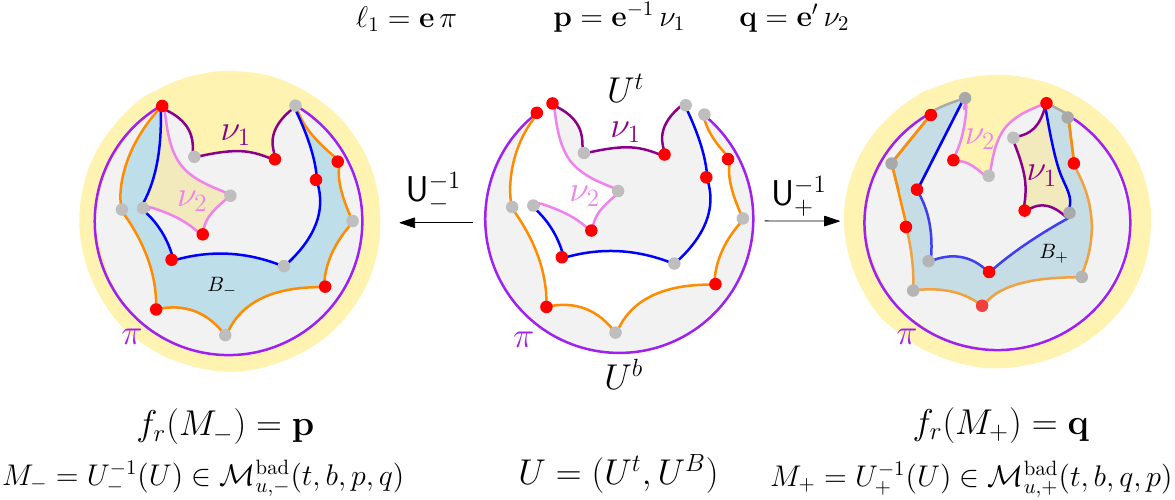}  
	\caption{\label{fig-bad-maps-8} 
		In the middle we fixed $U\in \cU^{\text{bad}}_{\text{u}}(t,b,p,q)$. On the left we showed the embedded map $M_-:=\un^{-1}_-(U)\in \cM^{\text{bad}}_{\text{u},-}(t,b,p,q)$ and on the right we showed the map $M_+:=\un^{-1}_+(U)\in \cM^{\text{bad}}_{\text{u},+}(t,b,q,p)$.
	}
\end{center}
\vspace{-3ex}
\end{figure}

\appendix

\section{Large-\texorpdfstring{$N$}{N} limit surface sum ansatz}\label{subsec: large N surface sum guess}

Recall that the proof of Theorem~\ref{thm: surface sum representation in 't hooft limit} in Section~\ref{sec: Large N Limit} does not provide any intuition for guessing the expression for $\phi(s) = \prod_{i=1}^n \phi(\ell_i)$ given in the statement of Theorem~\ref{thm: surface sum representation in 't hooft limit}.  In this appendix, we offer insight and intuition as to why this expression should be the correct ansatz. We acknowledge that several deductions and claims in this appendix are not rigorous; their primary purpose is to provide an intuitive justification for our ansatz. Additionally, none of the results in this appendix are used in any other part of the paper.
As a result, readers who are not interested in our intuitive explanation can simply skip this entire appendix.

\subsection{Limiting heuristic}\label{subsubsec: Limiting Heuristic}

First let's recall the heuristic given in \cite[Section 3]{cao2023random} for the large-$N$ limit of $\phi_N(s)=\phi_{\Lambda_N,N,\upbeta}(s)$. While the correct expression for $\phi(s)$ given in Theorem~\ref{thm: surface sum representation in 't hooft limit} is slightly different than the proposed one in \cite{cao2023random}, from their heuristic one can correctly identify the factorization $\phi(s)=\prod_{i=1}^n\phi(\ell_i)$, the form of the weights $w_{\infty}(M)$ associated with each map $M$, and illustrate why $\phi(\ell)$ only involves planar embedded maps with boundary $\ell$ (but their heuristic cannot predict that one has to further restrict to non-separable planar embedded maps). 

\medskip

In this heuristic, we assume that $s =\{\ell_1,\dots,\ell_n\}$. The idea is to start with the formula (from Theorem~\ref{thm: finite N surface sum})
\begin{equation}\label{eq:starting-point}
\lim_{N\to\infty}\phi_{\Lambda_N,N,\upbeta}(s) = \lim_{N\to\infty} Z_{\Lambda_N,N,\upbeta}^{-1}\sum_{K:\cP_{\Lambda_N}\to \NN}\sum_{M\in \cM(s,K)}\upbeta^{\area(M)}
\cdot
w_N(M)
\cdot
N^{\gec(M)},
\end{equation}
and simply to interchange sums and limits as desired and take the limit of the Weingarten weights separately from the $N^{\gec(M)}$ term, allowing us to factor out a copy of the partition function. To see this, let's first deal with the Weingarten weights. Recall from \eqref{eq:w-wei} that $	w_N(M)=\prod_{e\in E_{\Lambda}^+(s,K)} \nWg_N(\mu_e(M))$ is defined in terms of the normalized Weingarten function $\nWg_N$ introduced in \eqref{eq:wein-fct}. The nice thing about the latter function is that it has a nice formula in the $N\to\infty$ limit. That is (see e.g. \cite[Corollary 2.6]{collins_integration_2006}),
\begin{align*}
\overline{Wg}_N(\sigma) = \mob(\sigma) + O(N^{-2}) \quad\text{ as $N\to\infty$,}
\end{align*}where if $\sigma$ decomposes into cycles of length $C_1,\dots,C_k$, then \begin{align*}
\mob(\sigma) =\prod_{i=1}^k(-1)^{C_i-1}\cat(C_i-1).
\end{align*}
Also, recalling \eqref{eq:w-wei-inf}, we have that
\begin{align*}
w_{\infty}(M)=\prod_{e\in E_{\Lambda}^+(s,K)}\mob(\mu_e(M)).
\end{align*}
Thus, replacing the $w_N(M)$ terms in \eqref{eq:starting-point} by $w_{\infty}(M)$ we get 
\begin{align*}
\phi_{\text{ans}}(s) =\lim_{N\to\infty} Z_{\Lambda_N, N, \upbeta}^{-1}\sum_{K:\cP_{\Lambda_N}\to \NN}\sum_{M\in \cM(s,K)}\upbeta^{\area(M)}w_{\infty}(M)N^{\gec(M)}.
\end{align*}
(We will use $\phi_{\text{ans}}(s)$ as our ansatz for the large-$N$ limit throughout this section.) Next, since $w_{\infty}(M)$ is defined in \eqref{eq:w-wei-inf} as a product over each blue face of $M$, we can obtain the following factorization: if $M$ splits into $k$ connected components $\{M_i\}_{1\leq i\leq k}$, then 
\begin{align*}
w_{\infty}(M) = \prod_{i=1}^kw_{\infty}(M_i).
\end{align*}
This fact allows us to factor out a copy of the partition function from $\phi_{\text{ans}}(s)$ as we are going to explain. 
Any embedded map $M\in \cM(s,K)$ can be split into the union of two embedded maps:
\begin{itemize}
\item one embedded map $M_b$ where every component of $M_b$ has at least one boundary which is sent to one loop in $s$;
\item another embedded map $M_\emptyset$ where every component of $M_\emptyset$ has no boundary.
\end{itemize}
Hence, denoting by $\cM_b(s,K)$ the collection of embedded maps  in $\cM(s,K)$ where every component has at least one boundary sent to one loop in $s$, and by $\cM_\emptyset(\emptyset,K)$ the collection of embedded maps  in $\cM(s,K)$ having all the components with no boundary, we get that (using the same notation as in Theorem~\ref{thm: fixed K 't Hooft master loop equation for surface sum}),
\begin{equation*}
\cM(s,K)=\bigsqcup_{K_b+K_\emptyset=K}\cM_b(s,K_b) \times \cM_\emptyset(\emptyset,K_\emptyset).
\end{equation*}
Setting $W(M) := \upbeta^{\area(M)}w_{\infty}(M)$, this gives us that we can rewrite $\phi_{\text{ans}}(s)$ as
\begin{align*}
&Z_{\Lambda_N, N, \upbeta}^{-1}\sum_{K:\cP_{\Lambda_N}\to \NN} \sum_{K_b+K_\emptyset=K}\sum_{M_b\in \cM_b(s,K_b)}\sum_{M_\emptyset\in \cM(\emptyset,K_\emptyset)}W(M_b)W(M_\emptyset)N^{\gec(M_b)}N^{\gec(M_\emptyset)}\\
=&Z_{\Lambda_N, N, \upbeta}^{-1}\sum_{K_b:\cP_{\Lambda_N}\to \NN}\sum_{K_\emptyset:\cP_{\Lambda_N}\to \NN}\sum_{M_b\in \cM_b(s,K_b)}\sum_{M_\emptyset\in \cM(\emptyset,K_\emptyset)}W(M_b)W(M_\emptyset)N^{\gec(M_b)}N^{\gec(M_\emptyset)} \\
=&\sum_{K:\cP_{\Lambda_N}\to \NN}\sum_{M\in \cM_b(s,K_b)}\upbeta^{\area(M)}\cdot w_{\infty}(M) \cdot N^{\gec(M)},
\end{align*}
where to get the last equality we used that 
\[Z_{\Lambda_N,N, \upbeta} = \sum_{K:\cP_{\Lambda_N}\to \NN} \sum_{M\in \cM(\emptyset,K)} W(M)N^{\gec(M)}.\] 
Therefore, since we have been able to factor out a copy of the partition function, in the large-$N$ limit we should only consider embedded maps where every connected component has a boundary.

Finally, we need to understand the limiting behavior of the factor $N^{\gec(M)}$ when $M\in \cM_b(s,K_b)$. To do this, suppose that $M$ consists of $k$ connected components which we denote by $\{M_i\}_{i=1}^k$ and recall that every component has at least one boundary sent to one loop in $s$. Then recalling that $\gec(M) = \sum_{i=1}^k\left[2-2g(M_i)-b(M_i)\right]\leq 0$, we get that
\[\lim_{N\to\infty}N^{\gec(M)}= \lim_{N\to\infty}N^{\sum_{i=1}^k\left[2-2g(M_i)-2b(M_i)\right]}\in \{0,1\}.\] 
Note the only possible way to get the zero exponent is\footnote{Note that also embedded maps $M$ having a single connected component, genus $1$ and no boundary (i.e.\ tori) have $\gec(M)=0$, but this maps have been previously canceled when we factored out a copy of the partition function.} if  
\begin{equation}
2-2g(M_i)-2b(M_i)=0,\qquad\text{for each connected component }M_i.
\end{equation} 
That is, if the embedded map $M$ has $n$ connected components and each component has the topology of the disk with its boundary sent by the embedding to a single distinct loop of $s=\{\ell_1,\dots,\ell_n\}$. This observation, combined with the fact that the weights factor along components, gives us the expression
\begin{align*}
\phi_{\text{ans}}(s) = \prod_{i=1}^n\phi_{\text{ans}}(\ell_i), \quad\text{ with }\quad \phi_{\text{ans}}(\ell_i)= \sum_{K:\cP_{\ZZ^d}\to \NN}\sum_{M\in \cP\cM(\ell_i,K)}\upbeta^{\area(M)}w_{\infty}(M),
\end{align*}
where we recall that the set $\cP\cM(\ell,K)$ has been introduced in Definition~\ref{defn:planar-embedded}.

Note that the above expression is similar to the correct expression given in Theorem~\ref{thm: surface sum representation in 't hooft limit}, with the only difference being that we are summing over planar embedded maps in $\cP\cM(\ell_i,K)$ instead of over non-separable planar embedded maps in $\npe(\ell_i,K)$. This fact will be clarified in the next section.

\subsection{Necessity of extra condition}\label{sect:counter-exemp-sep}

Before explaining why we only consider non-separable maps, we present a simple proof showing that we must impose some further condition on the planar maps in $\cP\cM(\ell,K)$. That is, we will show that
\begin{align}\label{eq:kjvwiefbo}
\phi_{\text{ans}}(\ell) := \sum_{K:\cP_{\ZZ^d}\to \NN}\sum_{M\in \cP\cM(\ell,K)}\upbeta^{\area(M)}w_{\infty}(M)
\end{align}
contradicts facts about Wilson loop expectations. In particular, we will show that in dimension two, $\phi_{\text{ans}}(p)\neq \upbeta$ for $p\in \cP$. This contradicts the fact that the Wilson loop expectation of a single plaquette is $\upbeta$ in the large-$N$ limit in dimension two~\cite[Theorem 2.7]{Basu:2016dnp}\footnote{While \cite{Basu:2016dnp} considers $\SO(N)$ lattice Yang--Mills, their results also hold for $\unitary(N)$ lattice Yang--Mills because the limiting master loop equation is the same for both groups.}

Recall that if $M$ is a map with plaquette assignment $K$, then $\area(M) = \sum_{p\in \cP_{\Lambda}}K(p)$. Thus
\begin{align*}
\phi_{\text{ans}}(\ell) = \sum_{A=0}^{\infty}\upbeta^{A}\sum_{\substack{K:\cP_{\ZZ^d}\to \NN\\\sum_{p\in \cP_{\Lambda}}K(p) = A}}\sum_{M\in \cP\cM(\ell,K)}w_{\infty}(M).
\end{align*}
As $\phi_{\text{ans}}(p)$ should equal $\upbeta$, we must have that only the $A=1$ term is non-zero (this can be rigorously justified by taking derivatives in $\upbeta$). That being said, with our current definition in \eqref{eq:kjvwiefbo}, it is not hard to show that the $A=3$ term is non-zero.

Indeed, notice that there are only $5$ possible connected $K:\cP_{\ZZ^d}\to \NN$ (recall Definition~\ref{defn:connected}) such that $\sum_{p\in \cP_{\Lambda}}K(p) = 3$ and $(p,K)$ is balanced. These are: $K_0$ defined by  $K_0(p)=1$, $K_0(p^{-1})=2$ and $K_0(q)=0$ for all $q\in \cP_{\Lambda}\sm\{p,p^{-1}\}$; and $K_i$ for $i\in \{1,2,3,4\}$, defined by $K_i(p^{-1})=K_i(p_i)=K_i(p_i^{-1})=1$,  where $p_i$ is one of the four positively oriented plaquettes that share an edge with $p$.

With this, it is a simple computation to check that $\\sum_{M \in \cP\cM(p, K_i)} w_\infty(M) = -1$ for $i\in \{1,2,3,4\}$ and $\sum_{M \in \cP\cM(p, K_0)} w_\infty(M)=-4$. 
Thus, we get that 
\begin{align*}
\sum_{\substack{K:\cP_{\ZZ^d}\to \NN\\\sum_{p\in \cP_{\Lambda}}K(p) = 3}}\sum_{M\in \cP\cM(p,K)}w_{\infty}(M) = -8\neq 0.
\end{align*}
Thus we cannot simply consider all planar maps in $\cP\cM(\ell,K)$.

\subsection{Backtrack erasure}\label{subsubsec: backtrack Condition}

While the above discussion shows that we need to enforce more conditions than the ones found in Section~\ref{subsubsec: Limiting Heuristic}, it does not illuminate why we want to only sum over non-separable embedded maps (recall their definition from Definition~\ref{defn:non-separableplanar-embedded}). It turns out that the non-separable condition emerges if one looks for conditions to ensure that the result in Lemma~\ref{lemma: backtrack cancellations} holds, i.e.\ the ability to erase backtracks.

More precisely, if $\phi_{\text{ans}}(\ell)$ is the Wilson loop expectation in the large-$N$ limit, it is natural to require that 
\begin{equation}\label{eq:requested-coeff}
\phi_{\text{ans}}(\pi_1 \, \mathbf{e} \, \mathbf{e}^{-1} \, \pi_2) = \phi_{\text{ans}}(\pi_1 \, \pi_2),
\end{equation}
where $\pi_1$ and $\pi_2$ are two paths of edges, $\mathbf{e}$ corresponds to the oriented edge $e\in E_{\Lambda}$, and $\mathbf{e}^{-1}$ corresponds to $e^{-1}$.
While such invariance is clear for Wilson loop expectations at finite $N$, it is not so trivial to establish for the surface sum perspective for the large-$N$ limit.
To understand what is needed for such an invariance, let's first parse through what the condition  in \eqref{eq:requested-coeff} imposes: We need that the sum of weights for all the embedded maps with boundary $\pi_1 \, \mathbf{e} \, \mathbf{e}^{-1} \, \pi_2$ is equivalent to the sum of weights for all the embedded maps with boundary $\pi_1 \, \pi_2$. Observe that any map $M \in \cP\cM(\pi_1 \, \pi_2,K)$ with boundary $\pi_1 \, \pi_2$ can be naturally associated to the map $M' \in \cP\cM(\pi_1 \, \mathbf{e} \, \mathbf{e}^{-1} \, \pi_2,K)$ with boundary $\pi_1 \, \mathbf{e} \, \mathbf{e}^{-1} \, \pi_2$ where the two edges $\mathbf{e}$ and $\mathbf{e}^{-1}$ are in a blue 2-gon and the rest of the map is precisely $M$; see Figure~\ref{fig-map-bij} for a schematic illustration of this correspondence. 

\begin{figure}[ht!]
\begin{center}
	\includegraphics[width=.79\textwidth]{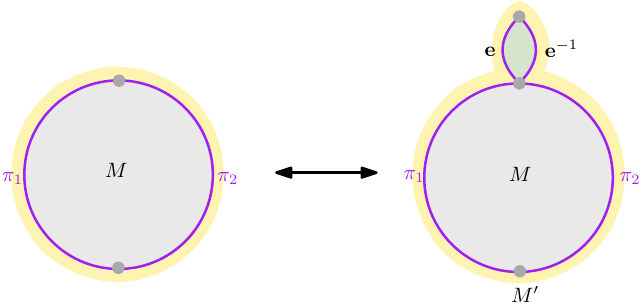}  
	\caption{\label{fig-map-bij} An example of two maps  $M \in \cP\cM(\pi_1 \, \pi_2,K)$ and $M' \in \cP\cM(\pi_1 \, \mathbf{e} \, \mathbf{e}^{-1} \, \pi_2,K)$ which are in correspondence. Note that the two maps have the same weight since the blue 2-gon has weight 1.}
\end{center}
\vspace{-3ex}
\end{figure}

Notice these two maps have the same weight. Indeed, they only differ by the blue 2-gon with boundary $\mathbf{e} \, \mathbf{e}^{-1}$ which has weight 1 (recall the definition of weights for the blue faces from \eqref{eq:w-wei-inf}). Thus, if we can show that the weights of all the embedded maps with boundary $\ell$ that do not have $\mathbf{e} \, \mathbf{e}^{-1}$ in a 2-gon sum to zero, then we get the condition in \eqref{eq:requested-coeff} holds.

These desired cancellations of weights have been established in  Lemma~\ref{lemma: single vertex pinching cancellations} for non-separable maps.
Notice the cancellations described in such lemma are only possible if the blue face $B$ has disjoint vertices. This is because if $B$ does not have disjoint vertices, the pinching from a single vertex does not necessarily provide weights that satisfy the Catalan number recursion \eqref{eq:catalan number recurstion}. For instance, in Figure~\ref{fig-btcancel-counter-2} we show one example where the proof of Lemma~\ref{lemma: single vertex pinching cancellations} would fail.

\begin{figure}[ht!]
\begin{center}
	\includegraphics[width=.35\textwidth]{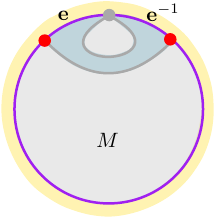}  
	\caption{\label{fig-btcancel-counter-2} An example of a separable map that would break the proof of Lemma~\ref{lemma: single vertex pinching cancellations}.}
\end{center}
\vspace{-3ex}
\end{figure}

By analyzing several examples similar to Figure~\ref{fig-btcancel-counter-2}, it becomes apparent that restricting our attention to the set of non-separable maps $\npe(\ell, K)$ is quite natural to get the cancellations in Lemma~\ref{lemma: single vertex pinching cancellations}. Implementing this restriction on $\phi_{\text{ans}}$ from \eqref{eq:kjvwiefbo}, we arrive at the claimed form $\phi_{\text{ans}}(\ell) = \phi(\ell)$, where the latter quantity has been introduced in the statement of Theorem~\ref{thm: surface sum representation in 't hooft limit}.

\begin{rmk}
We finally provide an alternative explanation for why the signed Catalan numbers should be considered the appropriate weights. 
If we assume that 
\begin{itemize}
	\item our limit $\phi(\ell)$ should have weights of the form $w_{\infty}(M) = \prod_{f\in \BF(M)}w(f)$, where $w(f)$ only depends on $\deg(f)$;
	\item and the cancellations of the weights of the maps holds in terms of the groupings we used above, that is, we assume that \eqref{eq: single pinching cancellation equation form} holds;
\end{itemize}
then the correct weights must be our proposed weights $w_{\infty}(M)$. Indeed, repeating the same arguments as above, we would get the equation
\begin{align*}
	w(2k) + \sum_{h=1}^{k-1}w(2h)w(2(k-h))=0.
\end{align*}
Now this recursion, with initial condition $w(2)=1$, is satisfied if and only if $w(2h)$ are the signed Catalan numbers.
\end{rmk}

\newpage

\addcontentsline{toc}{section}{References}
\bibliography{mybib}
\bibliographystyle{hmralphaabbrv}
 
\end{document}